\documentclass[a4paper,oneside,reqno]{amsart}

\usepackage[english]{babel}
\usepackage[utf8]{inputenc} 
\usepackage{courier} 
\usepackage{eulervm} 
\usepackage[bb=boondox,bbscaled=1.05,scr=dutchcal]{mathalfa}    %
\usepackage{dsfont}
\normalfont

\usepackage{tabularx,longtable,booktabs} 
\usepackage{caption,subcaption} 
\usepackage{fullpage} 
\usepackage[english]{varioref} 
\usepackage[dvipsnames]{xcolor} 
\usepackage{bookmark} 
\usepackage{tikz,pgfplots} 
\usetikzlibrary{arrows,decorations.pathmorphing,decorations.pathreplacing,decorations.markings,calc,backgrounds,patterns,positioning}
\usepackage{soul} 
\usepackage{enumitem} 

\setlist[itemize,1]{leftmargin=1.5em} 
\setlist[enumerate,1]{leftmargin=1.5em} 
\makeatletter
	\def\l@subsection{\@tocline{2}{0pt}{2.5pc}{5pc}{}} 
\makeatother
\usepackage{setspace}

\usepackage{textcomp} 

\usepackage{hyperref} 
\definecolor{webbrown}{rgb}{0.65, 0.16, 0.16} 
\hypersetup{
colorlinks=true, linktocpage=true, pdfstartpage=1, pdfstartview=FitV,breaklinks=true, pdfpagemode=UseNone, pageanchor=true, pdfpagemode=UseOutlines,plainpages=false, bookmarksnumbered, bookmarksopen=true, bookmarksopenlevel=1,hypertexnames=true, pdfhighlight=/O,urlcolor=webbrown, linkcolor=RoyalBlue, citecolor=ForestGreen}

\usepackage[textsize=footnotesize]{todonotes}   
\presetkeys%
		{todonotes}%
		{color=Apricot}{}%
\usepackage[style=alphabetic,maxalphanames=5,giveninits,sorting=nyt,%
maxbibnames=99,backend=bibtex]{biblatex}
\renewbibmacro{in:}{}
\usepackage{csquotes}

\usepackage[noabbrev]{cleveref}
\expandafter\def\csname ver@etex.sty\endcsname{3000/12/31}

\crefname{lemma}{lemma}{lemmata}
\Crefname{lemma}{Lemma}{Lemmata}
\crefname{introtheorem}{theorem}{theorems}
\Crefname{introtheorem}{Theorem}{Theorems}
\crefname{subsection}{subsection}{subsections}
\Crefname{subsection}{Subsection}{Subsections}
\crefname{conjecture}{conjecture}{conjectures}
\Crefname{conjecture}{Conjecture}{Conjectures}

\tikzset{
	>=stealth',
	punkt/.style={
		rectangle,
		rounded corners,
		draw=black, fill=white,
		minimum height=2em,
		text centered},
	pil/.style={
		->,
		shorten <=2pt,
		shorten >=2pt,}
}
\pgfplotsset{every axis/.append style={
	tick label style={font=\tiny}
}}
\pgfplotsset{compat=1.18}

\usepackage{amsmath,amssymb,amsthm,mathtools}   
\usepackage{bm,braket,faktor,stmaryrd}          
\numberwithin{equation}{section}                

\theoremstyle{plain}
\newtheorem{theorem}{Theorem}[section]
\newtheorem{proposition}[theorem]{Proposition}

\newtheorem{lemma}[theorem]{Lemma}
\newtheorem{corollary}[theorem]{Corollary}
\newtheorem{introtheorem}{Theorem}

\theoremstyle{definition}
\newtheorem{definition}[theorem]{Definition}
\newtheorem{remark}[theorem]{Remark}
\newtheorem{example}[theorem]{Example}

\DeclarePairedDelimiter\floor{\lfloor}{\rfloor}

\newcommand{\pFq}[2]{\vphantom{F}_{#1}F_{#2}}

\newcommand{\bbraket}[1]{\llbracket #1 \rrbracket}

\DeclareMathOperator\Disc{Disc}
\DeclareMathOperator\Tr{Tr}

\newcommand{\Id}{\mathrm{Id}}

\DeclareMathOperator*{\Res}{Res}

\newcommand{\bigO}{\mathrm{O}}
\newcommand{\smallo}{\mathrm{o}}

\def\N{\mathbb{N}}
\def\Z{\mathbb{Z}}
\def\Q{\mathbb{Q}}
\def\R{\mathbb{R}}
\def\C{\mathbb{C}}
\def\P{\mathbb{P}}
\newcommand{\iu}{\mathrm{i}}

\newcommand{\SL}{\mathrm{SL}}

\def\Mbar{\overline{{\mathcal{M}}}}

\newcommand{\mf}[1]{\mathfrak{#1}}
\newcommand{\mc}[1]{\mathcal{#1}}
\newcommand{\mb}[1]{\mathbb{#1}}
\newcommand{\ms}[1]{\mathsf{#1}}

\begin{document}
\renewcommand{\hbar}{\hslash}

\title[Resurgent large genus asymptotics of intersection numbers]{Resurgent large genus asymptotics of intersection numbers}

\author[B. Eynard]{B.~Eynard}%
\address[B. Eynard]{
	Universit\'e Paris-Saclay, CNRS, CEA, Institut de Physique Th\'eorique, Gif-sur-Yvette, France and %
	CRM, Centre de Recherches Math\'ematiques de Montr\'eal, Universit\'e de Montr\'eal, QC, Canada
	}%
\email{bertrand.eynard@ipht.fr}

\author[E. Garcia-Failde]{E.~Garcia-Failde}
\address[E. Garcia-Failde]{
	Sorbonne Universit\'e, CNRS, Institut de Math\'ematiques de Jussieu--Paris Rive Gauche, Paris, France %
}
\email{elba.garcia-failde@imj-prg.fr}

\author[A.~Giacchetto]{A.~Giacchetto}
\address[A.~Giacchetto]{
	Universit\'e Paris-Saclay, CNRS, CEA, Institut de Physique Th\'eorique, Gif-sur-Yvette, France %
}
\email{alessandro.giacchetto@ipht.fr}

\author[P.~Gregori]{P.~Gregori}
\address[P.~Gregori]{ 
	Universit\'e Paris-Saclay, CNRS, CEA, Institut de Physique Th\'eorique, Gif-sur-Yvette, France %
}
\email{paolo.gregori@ipht.fr}

\author[D.~Lewa{\'{n}}ski]{D.~Lewa{\'{n}}ski}
\address[D.~Lewa{\'{n}}ski]{
	Università degli Studi di Trieste, Sezione di Matematica, Trieste, Italy and Université de Genève, Section de Mathématiques, Genève, Switzerland %
}
\email{danilo.lewanski@units.it}









\begin{abstract}
	In this paper, we present a novel approach for computing the large genus asymptotics of intersection numbers. Our strategy is based on a resurgent analysis of the $n$-point functions of such intersection numbers, which are computed via determinantal formulae, and relies on the presence of a quantum curve. With this approach, we are able to extend the recent results of Aggarwal for Witten--Kontsevich intersection numbers with the computation of all subleading corrections, proving a conjecture of Guo--Yang, and to obtain new results on $r$-spin and Theta-class intersection numbers.
\end{abstract}

\maketitle

\vspace*{-.75cm}
\begin{spacing}{-1.2}
	\tableofcontents
\end{spacing}

\section{Introduction}
\addtocontents{toc}{\protect\setcounter{tocdepth}{1}}

\subsection*{Motivations and methods}
The aim of this paper is to study the large genus asymptotics of certain intersection numbers on the moduli space $\Mbar_{g,n}$ of genus $g$ stable complex curves with $n$ marked points. The main class of intersection numbers under consideration is that of $\psi$-class intersection numbers \cite{Wit91,Kon92}:
\begin{equation}\label{eqn:WK:int:numbrs}
	\braket{\tau_{d_1} \cdots \tau_{d_n}}
	=
	\int_{\Mbar_{g,n}} \prod_{i=1}^n \psi_i^{d_i} \,.
\end{equation}
Here $|d| = d_1 + \cdots + d_n = 3g - 3 + n$ uniquely determines the genus, and $\psi_i = c_1(\mc{L}_i)$ is the first Chern class of the line bundle $\mc{L}_i$ whose fibre over $[C,p_1\dots,p_n] \in \Mbar_{g,n}$ is the cotangent line $T_{p_i}^{\ast} C$. Such intersection numbers can be thought of as the intersection of $\psi$-classes with the fundamental class of the moduli space of curves. Another possibility, recently considered by Norbury \cite{Nor23}, is to intersect $\psi$-classes with a different cohomology class called the $\Theta$-class:
\begin{equation}\label{eqn:Theta:int:numbrs}
	\braket{\tau_{d_1} \cdots \tau_{d_n}}^{\Theta}
	=
	\int_{\Mbar_{g,n}} \Theta_{g,n} \prod_{i=1}^n \psi_i^{d_i} \,,
\end{equation}
where $|d| = g - 1$ due to $\Theta_{g,n}$ being of pure complex degree $2g-2+n$. Yet another collection of intersection numbers is that of $r$-spin intersection numbers \cite{Wit93}: for every $r \ge 2$,
\begin{equation}\label{eqn:rspin:int:numbrs}
	\braket{\tau_{d_1,a_1} \cdots \tau_{d_n,a_n}}^{r\textup{-spin}}
	=
	\int_{\Mbar_{g,n}} W^r_{g,n}(a_1,\dots,a_n) \prod_{i=1}^n \psi_i^{d_i} \,,
\end{equation}
where $W^r_{g,n}(a_1,\dots,a_n)$ is called the Witten $r$-spin class. Here $a_i \in \set{1,\dots,r-1}$, and $r|d| + |a| = (r+1)(2g-2+n)$. For $r = 2$, they coincide with $\psi$-class intersection numbers.

In the last few decades, there has been a growing realisation that many fundamental invariants in physics and geometry can be expressed through intersection numbers. The first connection was established by Witten \cite{Wit91}, who conjectured that the partition functions of topological $2d$ quantum gravity coincides with the generating series of $\psi$-class intersection numbers. Similarly, $r$-spin intersection numbers emerge in the context of topological $2d$ quantum gravity coupled to a certain gauge theory \cite{Wit92}. More recently, $\psi$-class intersection numbers have regained attention within the physics community due to new discoveries linking them to Jackiw--Teitelboim gravity \cite{SSS}. In similar fashion, super Jackiw--Teitelboim gravity has been linked to $\Theta$-class intersection numbers \cite{SW20}.

Geometrically, $\psi$-class intersection numbers represent the easiest possible instance of Gromov--Witten invariants, namely for the target space being a point. They appear in the works of Kontsevich \cite{Kon92} and Mirzakhani \cite{Mir07b} on the symplectic volumes of the moduli space of metric ribbon graphs and hyperbolic Riemann surfaces, as well as in the the asymptotic counting of geodesic curves in flat and hyperbolic random geometry and in the theory of Masur--Veech volumes \cite{Mir08b,DGZZ21,And+23}. All such intersection numbers also appear in the context of combinatorics of maps and random matrix theory \cite{Eyn16,BCEG}, as well as in the theory of integrable systems \cite{Kon92,FSZ10,CGG}.

Given their ubiquitous presence in mathematics and physics, a natural question arises:
\begin{center}
	\vspace{-.5em}
	\textsf{How does one compute such intersection numbers?}
	\vspace{-.5em}
\end{center}
The connection with integrable systems has in many instances been the key for the evaluation of these invariants. Both $\psi$- and $\Theta$-class intersection numbers are specific solutions of the Korteweg--De~Vries (KdV) integrable hierarchy \cite{Kon92,CGG}, while $r$-spin intersection numbers constitute a solution of the $r$-KdV hierarchy \cite{FSZ10}. Concretely, their generating series obey an infinite set of partial differential equations in infinitely many variables that recursively determine the intersection numbers uniquely after fixing initial conditions. An equivalent way of specifying these solutions of integrable hierarchies is by means of Virasoro constraints \cite{DVV91,GN92,AvM92}: the generating series of intersection numbers is annihilated by an infinite set of linear differential operators, forming a representation of the so-called Virasoro (or more generally $W$-) algebra. Alternative recursive methods include the topological recursion \cite{EO07}, the cut-and-join equation \cite{Ale11}, a recursion by Liu--Xu \cite{LX14}, as well as some exact formulae for the generating series such as the determinantal formula \cite{BE,BDY16}, integral representations by Okounkov \cite{Oko02}, Fock space inspired generating series by Buryak \cite{Bur17}, and the symmetric functions approach of Eynard--Mitsios \cite{EM}.

While the methods mentioned above theoretically allow for an explicit evaluation of intersection numbers, it is widely acknowledged that all these quantities are typically intricate. With a few exceptional cases, no simple closed-form expression is expected to exist. Even if on the one hand it might be hopeless to control these invariants exactly, on the other it might be possible to gain an understanding of their behaviour as the genus increases. Thus, another natural question arises:
\begin{center}
	\vspace{-.5em}
	\textsf{Is it possible to access intersection numbers asymptotically in the large genus limit?}
	\vspace{-.5em}
\end{center}
For $\psi$-class intersection numbers, it was predicted by Delecroix--Goujard--Zograf--Zorich \cite{DGZZ21} that in the large genus limit the intersection numbers simplify considerably:
\begin{equation}
	\braket{\tau_{d_1} \cdots \tau_{d_n}}
	\sim
	\frac{(6g-5+2n)!!}{24^g \, g! \, \prod_{i=1}^{n} (2d_i + 1)!!}
\end{equation}
uniformly in $d_1,\dots,d_n$. Special cases of the asymptotic formula were proved in \cite{LX14} and in \cite{DGZZ20} when the genus is mostly concentrated in $d_1$ and in $d_1 + d_2$ respectively. The above asymptotic formula has been recently proved by Aggarwal \cite{Agg21}, whose strategy employs a careful combinatorial and probabilistic analysis of the associated Virasoro constraints (more precisely, Aggarwal proved the asymptotic formula under the more stringent regime $n = \smallo(g^{1/2})$). An alternative proof has been later proposed by Guo--Yang \cite{GY22}, whose approach is based on a combinatorial analysis of the determinantal formula. The large genus asymptotics of $r$-spin intersection numbers was considered by Dubrovin--Yang--Zagier in \cite{DYZ} in the specific case of $n = 1$.

The study of intersection numbers in the large genus limit holds significance for several reasons. Firstly, the intricate nature of intersection numbers simplifies enormously in the large genus limit, leading to closed-form asymptotic evaluations. Secondly, many interesting quantities associated to several geometric models appear to be exclusively attainable in the asymptotic regime. For instance, the length of the shortest geodesic on a fixed hyperbolic surface is notoriously hard to compute, but its average value in the large genus limit can be explicitly computed as $\approx 1.61498\ldots$ \cite{MP19}. This computation crucially relies on the large genus behaviour of Weil--Petersson volumes predicted by Mirzakhani--Zograf \cite{MZ15}. Thirdly, in the large genus limit many universality phenomena can be unveiled. For instance, all intersection numbers studied so far manifest a $(2g)!$ factorial growth, as well as a (conjectural) polynomiality structure \cite{GY22}. Another example of universality regards the length spectra of random hyperbolic and combinatorial geodesics, both of which converge to a Poisson distribution \cite{MP19,JL23}. The large genus regime is also interesting from the physics point-of-view, as it is intimately linked to non-perturbative effects in quantum theory with new phenomena that are not captured by Feynman diagram expansions appear (see for instance \cite{CESV16,GM,GKKM,GS,EGGLS,IM}).

The aforementioned methods used to derive the large genus asymptotics of $\psi$-class intersection numbers heavily depend on ad hoc estimates for various quantities involved in the specific problem. Therefore, it would be desirable to have a new, more general strategy for approaching large genus asymptotics. In this paper, we aim to address this issue by introducing a novel method that utilises two main techniques:
\begin{enumerate}
	\item The \textit{Borel transform method}, that is the computation of large-order asymptotics through the study of the singularity structure in the Borel plane.

	\item The existence of \textit{determinant formulae} (a.k.a.~the matrix resolvent method) for computing the $n$-point functions of the corresponding enumerative problem, with building blocks being the solutions of a quantum curve (a.k.a.~topological ODE).
\end{enumerate}
The Borel transform method essentially combines classical results from Borel (concerning resummation of divergent series) and from Darboux (concerning asymptotic analysis), and can be naturally placed within Écalle's theory of resurgence \cite{Eca81,MS16}. On the other hand, determinantal formulae have been established in connection with topological recursion \cite{BE,BBE15} and integrable systems \cite{BDY16}. In the present paper, we employ the proposed strategy to address the three enumerative problems mentioned earlier: $\psi$-, $\Theta$-, and $r$-spin intersection numbers. Many more enumerative problems are solved by determinantal formulae. These include higher Weil--Petersson volumes \cite{BDY16}, GUE correlators \cite{DY17}, Gromov--Witten theory of $\mb{P}^1$ \cite{DYZ20}, Fan--Jarvis--Ruan--Witten theory \cite{BDY21}, and more. Given this broad applicability, we anticipate that our approach will pave the way for further investigations into large genus phenomena in the future.

The strategy outlined here offers multiple advantages. Firstly, it applies to several enumerative problems with minimal modifications, such as the study of the structure of the Borel singularities of the building blocks appearing in the determinantal formula. A second, and perhaps more noteworthy, advantage is that it enables the computation of subleading and exponentially subleading corrections through an implementable algorithm, revealing the polynomiality phenomena conjectured in \cite{GY22}. Thirdly, it sheds light on the universal features of the large genus regime, as well as the model-dependent ones, as exemplified in \cref{thm:intro:psi:Theta,thm:intro:rspin}.

\subsection*{Results}
As explained above, the geometry of the quantum curve underlying the different enumerative problems under consideration propagates to their large genus asymptotics. Consider, for example, the case of $\psi$- and $\Theta$-class intersection numbers. The underlying quantum curves are the Airy and the Bessel ODE respectively:
\begin{equation}
	\biggl( \biggl(\hbar \frac{d}{dx}\biggl)^2 - \; x \biggr) \psi_{\textup{Airy}}(x;\hbar) = 0 \,,
	\qquad\qquad
	\biggl( \biggl(x \, \hbar \frac{d}{dx}\biggr)^2 - \; x \biggr) \psi_{\textup{Bessel}}(x;\hbar) = 0 \,.
\end{equation}
For the corresponding enumerative problems, the large genus formula is given as follows (see \cref{thm:large:g:WK,thm:large:g:Nor} for more precise statements). In the following, we denote $(x)^{\underline{m}} = x (x-1)\cdots (x-m+1)$ as the falling factorial.

\begin{introtheorem}\label{thm:intro:psi:Theta}
	For any given $n \ge 1$ and $K \ge 0$, uniformly in $d_1,\dots,d_n$ as $g \to \infty$:
	\begin{align}
		\label{eqn:intro:psi:large:g}
		\begin{split}
			\braket{\tau_{d_1} \cdots \tau_{d_n}}
			=
			\ms{S}_{\ms{A}} \, \frac{2^n}{4\pi} \, \frac{\Gamma(2g-2+n)}{\ms{A}^{2g-2+n} \, \prod_{i=1}^n (2d_i + 1)!!}
			\bigg(
				1
				& +
				\frac{\ms{A}}{2g-3+n} \, \alpha_{1}
				+ \cdots \\
			&
				+
				\frac{\ms{A}^K}{(2g-3+n)^{\underline{K}}} \, \alpha_{K}
				+
				\bigO\Bigl( \frac{1}{g^{K+1}} \Bigr)
				\bigg) \,,
		\end{split}
		\\
		\label{eqn:intro:Theta:large:g}
		\begin{split}
			\braket{\tau_{d_1} \cdots \tau_{d_n}}^{\Theta}
			=
			\ms{S}_{\ms{B}} \, \frac{2^n}{4\pi} \, \frac{\Gamma(2g-2+n)}{\ms{B}^{2g-2+n} \, \prod_{i=1}^n (2d_i + 1)!!}
			\bigg(
				1
				& +
				\frac{\ms{B}}{2g-3+n} \, \beta_{1}
				+ \cdots \\
			&
				+
				\frac{\ms{B}^K}{(2g-3+n)^{\underline{K}}} \, \beta_{K}
				+
				\bigO\Bigl( \frac{1}{g^{K+1}} \Bigr)
				\bigg) \,.
		\end{split}
	\end{align}
	Moreover, the constants $\ms{A}$, $\ms{B}$, $\ms{S}_{\ms{A}}$, $S_\ms{B}$ and the sequences $(\alpha_k)_{k \ge 0}$ and $(\beta_k)_{k \ge 0}$ have the following values and geometric interpretations:
	\begin{itemize}
		\item The constants $\ms{A} = \frac{2}{3}$ and $\ms{B} = 2$ determine the exponential growth of the Airy and Bessel functions respectively:
		\begin{equation}
			\psi_{\textup{Airy}}(x;\hbar)
			\sim
			\frac{1}{\sqrt{2} \, x^{1/4}} e^{\pm \frac{\ms{A}}{\hbar} x^{3/2}} \,,
			\qquad\quad
			\psi_{\textup{Bessel}}(x;\hbar)
			\sim
			\frac{1}{\sqrt{2} \, x^{1/4}} e^{\pm \frac{\ms{B}}{\hbar} x^{1/2}} \,,
			\qquad\quad
			\text{as } x \to \infty \,.
		\end{equation}

		\item The constants $\ms{S}_{\ms{A}} = 1$ and $\ms{S}_{\ms{B}} = 2$ are the Stokes constants appearing in the Stokes phenomenon underlying the Airy and Bessel ODEs.

		\item Each term $\alpha_k$ and $\beta_k$ appearing in the asymptotic expansions is a polynomial function of $n$ and the multiplicities $p_m = \#\set{ d_i = m }$:
		\begin{equation}
			\alpha_k = \alpha_k\bigl( n, p_0, \ldots, p_{\floor{\frac{3}{2}k} - 1} \bigr) \,,
			\qquad\qquad
			\beta_k = \beta_k\bigl( n, p_0, \ldots, p_{\floor{\frac{1}{2}k} - 1} \bigr) \,,
		\end{equation}
		The values $3/2$ and $1/2$ are the exponents of $x$ in the exponential growth of the Airy and Bessel functions respectively. Moreover, we provide an algorithm to effectively compute $\alpha_k$ and $\beta_k$ from the coefficients of the asymptotic expansion of the Airy and Bessel functions respectively (see \cref{table:intro} for the first few values).
	\end{itemize}
\end{introtheorem}

\begin{table}
\centering
{\renewcommand{\arraystretch}{1.2}
\begin{tabularx}{.57\textwidth}[t]
	{ c | >{\raggedright\arraybackslash}X }
	\toprule
	$k$ & $\alpha_{k}$ \\
	\midrule
	$0$ & $1$
	\\
	\midrule
	$1$ & $
		- \tfrac{17 - 15n + 3n^2}{12}
		- \tfrac{(3 - n)(n - p_0)}{2}
		- \tfrac{(n - p_0)^{\underline{2}}}{4} $
	\\
	\midrule
	$2$ & $
		\tfrac{1225 - 1632 n + 741 n^2 - 138 n^3 + 9 n^4}{288}
		+
		\tfrac{(105 - 98 n + 30 n^2 - 3 n^3)(n - p_0)}{24}
		+
		\tfrac{3(10 - 7 n + n^2)(n - p_0 - p_1)}{8}
		+
		\tfrac{(59 - 51 n + 9 n^2)(n - p_0)^{\underline{2}}}{48}
		+
		\tfrac{5(n - p_0 - p_1 - p_2)}{8}
		+
		\tfrac{3(4 - n)(n - p_0 - 1)(n - p_0 - p_1)}{4}
		+
		\tfrac{(7 - 3 n)(n - p_0)^{\underline{3}}}{24}
		+
		\tfrac{3(n - p_0 - 1)^{\underline{2}}(n - p_0 - p_1)}{48}
		+
		\tfrac{3(n - p_0)^{\underline{4}}}{96} $
	\\
	\bottomrule
\end{tabularx}
}
\hfill
{\renewcommand{\arraystretch}{1.2}
\begin{tabularx}{.41\textwidth}[t]
	{ c | >{\raggedright\arraybackslash}X }
	\toprule
	$k$ & $\beta_k$ \\
	\midrule
	$0$ & $1$
	\\
	\midrule
	$1$ & $- \frac{1}{4}$
	\\
	\midrule
	$2$ & $\tfrac{9 - 4 n}{8} + \tfrac{n - p_0}{8}$
	\\
	\midrule
	$3$ & $
		- \tfrac{57 - 44 n + 8 n^2}{128}
		- \tfrac{(13 - 4 n)(n - p_0)}{32}
		- \tfrac{(n - p_0)^{\underline{2}}}{16} $
	\\
	\bottomrule
\end{tabularx}
}
\caption{
	The subleading corrections $\alpha_k$ and $\beta_k$ for the first few values of $k$.
}
\label{table:intro}
\end{table}

It can be easily shown that formula \labelcref{eqn:intro:psi:large:g} for $K = 0$ reduces to Aggarwal's result. The new insight that we provide is the geometric interpretation of the constants $\ms{S}_{\ms{A}} = 1$ and $\ms{A} = 2/3$. The subleading correction and their polynomiality behaviour is a novel result, and provides a proof of a conjecture by Guo--Yang \cite[conjecture~1]{GY22}. The case of $\Theta$-class intersection numbers is completely new. As noted previously, our proof strategy highlights the universality of the large genus asymptotics, as well as the the model-dependent ingredients. Moreover, it provides a concrete algorithm to compute the subleading corrections $\alpha_k$ and $\beta_k$ appearing in the asymptotic expansions.

We also provide the large genus asymptotics for $r$-spin intersection numbers, which were previously unknown for $n>1$. In this case, the underlying quantum curve is the $r$-Airy ODE:
\begin{equation}
	\biggl( \biggl(\hbar \frac{d}{dx}\biggl)^r - x \biggr) \psi_{r\textup{-Airy}}(x;\hbar) = 0 \,.
\end{equation}
One notable distinction compared to the previous case is the increased order of the differential equation, which is now $r$. This feature introduces a new element in the asymptotic analysis, specifically the appearance of exponentially subleading terms. Nonetheless, these terms can be effectively addressed using the Borel transform method. See \cref{thm:large:g:rspin} for a more precise statement. In the following, we denote by
\begin{equation}\label{eqn:r:fact}
	m!_{(r)} =
	\begin{cases}
		m (m - r)!_{(r)} & \text{if } m > r \,, \\
		m & \text{if } 0 < m \le r \,, \\
	\end{cases}
\end{equation}
the $r$-factorial, which generalises the double factorial for $r = 2$.

\begin{introtheorem}\label{thm:intro:rspin}
	For any given $n \ge 1$, $K \ge 0$, and $a_1,\dots,a_n$, uniformly in $d_1,\dots,d_n$ as $g \to \infty$:
	\begin{equation} \label{eqn:intro:rspin:large:g}
	\begin{split}
		\braket{\tau_{d_1,a_1} \cdots \tau_{d_n,a_n}}^{r\textup{-spin}}
		= \,
		& \frac{2^n}{2\pi} \frac{\Gamma(2g-2+n)}{r^{g-1-|d|} \, \prod_{i=1}^n (r d_i + a_i)!_{(r)}} \\
		\times
		\Bigg[
			&
			\frac{\ms{S}_{r,1}}{|\ms{A}_{r,1}|^{2g-2+n}}
			\biggl(
				\gamma^{(r,1)}_0
				+ \cdots +
				\frac{|\ms{A}_{r,1}|^K}{(2g-3+n)^{\underline{K}}} \,
				\gamma^{(r,1)}_K
				+ \bigO\Bigl( \frac{1}{g^{K+1}} \Bigr)
			\biggr) \\[1.2ex]
		+ \, & 
			\cdots \\
		+ \, &
			\frac{\ms{S}_{r,\floor{\frac{r-1}{2}}}}{|\ms{A}_{r,\floor{\frac{r-1}{2}}}|^{2g-2+n}}
			\biggl(
				\gamma^{(r,\floor{\frac{r-1}{2}})}_0
				+ \cdots +
				\frac{|\ms{A}_{r,\floor{\frac{r-1}{2}}}|^K}{(2g-3+n)^{\underline{K}}} \,
				\gamma^{(r,\floor{\frac{r-1}{2}})}_K
				+ \bigO\Bigl( \frac{1}{g^{K+1}} \Bigr)
			\biggr) \\
		+ \, &
			\frac{\delta_{r}^{\textup{even}}}{2}
			\frac{\ms{S}_{r,\frac{r}{2}}}{|\ms{A}_{r,\frac{r}{2}}|^{2g-2+n}}
			\biggl(
				\gamma^{(r,\frac{r}{2})}_0
				+ \cdots +
				\frac{|\ms{A}_{r,\frac{r}{2}}|^K}{(2g-3+n)^{\underline{K}}} \,
				\gamma^{(r,\frac{r}{2})}_K
				+ \bigO\Bigl( \frac{1}{g^{K+1}} \Bigr)
			\biggr)
			\Bigg] \,.
	\end{split}
	\end{equation}
	Here $\delta_{r}^{\textup{even}}$ gives one if $r$ is even and zero otherwise. Moreover, the constants $\ms{A}_{r,\alpha}$, $\ms{S}_{r,\alpha}$, and the sequences $(\gamma^{(r,\alpha)}_k)_{k \ge 0}$ have the following values and geometric interpretations:
	\begin{itemize}
		\item The constants $\ms{A}_{r,\alpha} = \frac{r}{r+1} (1 - e^{2\pi\iu \frac{\alpha}{r}})$ are determined by the exponential growth of the $r$-Airy functions:
		\begin{equation}
			\psi_{r\textup{-Airy}}(x;\hbar)
			\sim
			\frac{1}{\sqrt{r} \, x^{\frac{r-1}{2r}}}
			e^{\frac{1}{\hbar} \frac{r}{r+1} e^{2\pi\iu \frac{\alpha}{r}} x^{\frac{r+1}{r}}} \,,
			\qquad\quad
			\alpha = 1,\dots,r \,,
			\quad
			\text{ as } x \to \infty \,.
		\end{equation}
		Furthermore, the various exponentially asymptotic behaviours are arranged in ascending order of distance from the origin, coupled together via the symmetry relation $|\ms{A}_{r,\alpha}| = |\ms{A}_{r,r-\alpha}|$. The behaviour depends on the parity of $r$, as pictured in \cref{fig:actions:rspin}.

		\item The constants $\ms{S}_{r,\alpha} = 1$ are the Stokes constants appearing in the Stokes phenomenon underlying the $r$-Airy ODE.

		\item Each term $\gamma^{(r,\alpha)}_k$ is a function of $n$ and the multiplicities $p_m = \#\set{ r d_i + a_i = m }$. Moreover, we provide an algorithm to effectively compute $\gamma_k$ from the coefficients of the asymptotic expansion of the $r$-Airy function. The leading coefficient is explicitly given by
		\begin{equation}
			\gamma^{(r,\alpha)}_0
			=
			(-1)^{(\alpha-1)(|d|+n)}
			\frac{\prod_{i=1}^n \sin\left(\frac{\alpha a_i}{r}\pi \right)}{\sin(\frac{\alpha}{r}\pi)} \,.
		\end{equation}
		We refer to \cref{thm:large:g:rspin} for a more precise statement about its polynomial dependence on the multiplicities.
	\end{itemize}
\end{introtheorem}

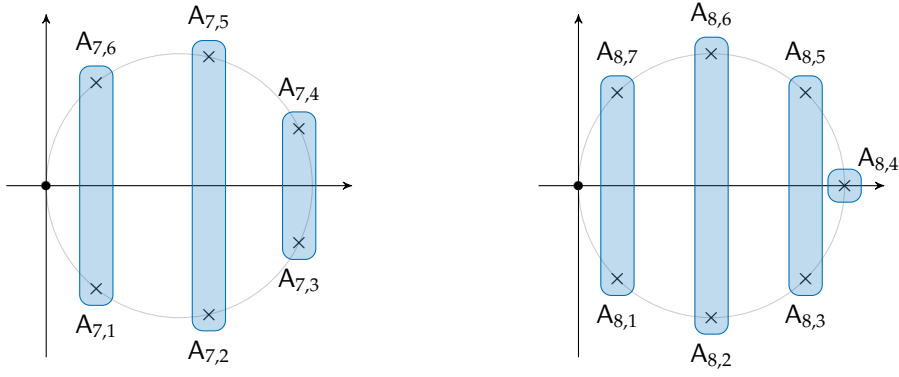
\begin{figure}
\centering
\begin{tikzpicture}[scale=1.75]
	\def\boxmargin{0.125};

	\draw[->] (-.3,0) -- (2.3,0);
	\draw[->] (0,-1.3) -- (0,1.3);

	\draw[opacity=.2] (1,0) circle (1cm);

	\node at (0,0) {\footnotesize $\bullet$};

	\node at (0.37651, -0.781831) {\footnotesize $\times$};
	\node at (0.37651, 0.781831) {\footnotesize $\times$};
	\draw[draw=NavyBlue, fill=NavyBlue, fill opacity=.25, rounded corners] (0.37651 - \boxmargin, -0.781831  - \boxmargin) rectangle ++(2*\boxmargin,2*0.781831 + 2*\boxmargin);
	\node at (0.37651, 0.781831 + \boxmargin) [above] {$\ms{A}_{7,6}$};
	\node at (0.37651, -0.781831 - \boxmargin) [below] {$\ms{A}_{7,1}$};

	\node at (1.22252, -0.974928) {\footnotesize $\times$};
	\node at (1.22252, 0.974928) {\footnotesize $\times$};
	\draw[draw=NavyBlue, fill=NavyBlue, fill opacity=.25, rounded corners] (1.22252 - \boxmargin, -0.974928  - \boxmargin) rectangle ++(2*\boxmargin,2*0.974928 + 2*\boxmargin);
	\node at (1.22252, 0.974928 + \boxmargin) [above] {$\ms{A}_{7,5}$};
	\node at (1.22252, -0.974928 - \boxmargin) [below] {$\ms{A}_{7,2}$};

	\node at (1.90097, -0.433884) {\footnotesize $\times$};
	\node at (1.90097, 0.433884) {\footnotesize $\times$};
	\draw[draw=NavyBlue, fill=NavyBlue, fill opacity=.25, rounded corners] (1.90097 - \boxmargin, -0.433884  - \boxmargin) rectangle ++(2*\boxmargin,2*0.433884 + 2*\boxmargin);
	\node at (1.90097, 0.433884 + \boxmargin) [above] {$\ms{A}_{7,4}$};
	\node at (1.90097, -0.433884 - \boxmargin) [below] {$\ms{A}_{7,3}$};

	\begin{scope}[xshift = 4cm]
		
		\draw[->] (-.3,0) -- (2.3,0);
		\draw[->] (0,-1.3) -- (0,1.3);

		\draw[opacity=.2] (1,0) circle (1cm);

		\node at (0,0) {\footnotesize $\bullet$};

		\node at (0.292893, -0.707107) {\footnotesize $\times$};
		\node at (0.292893, 0.707107) {\footnotesize $\times$};
		\draw[draw=NavyBlue, fill=NavyBlue, fill opacity=.25, rounded corners] (0.292893 - \boxmargin, -0.707107  - \boxmargin) rectangle ++(2*\boxmargin,2*0.707107 + 2*\boxmargin);
		\node at (0.292893, 0.707107 + \boxmargin) [above] {$\ms{A}_{8,7}$};
		\node at (0.292893, -0.707107 - \boxmargin) [below] {$\ms{A}_{8,1}$};

		\node at (1, -1) {\footnotesize $\times$};
		\node at (1, 1) {\footnotesize $\times$};
		\draw[draw=NavyBlue, fill=NavyBlue, fill opacity=.25, rounded corners] (1 - \boxmargin, -1  - \boxmargin) rectangle ++(2*\boxmargin,2*1 + 2*\boxmargin);
		\node at (1, 1 + \boxmargin) [above] {$\ms{A}_{8,6}$};
		\node at (1, -1 - \boxmargin) [below] {$\ms{A}_{8,2}$};

		\node at (1.70711, -0.707107) {\footnotesize $\times$};
		\node at (1.70711, 0.707107) {\footnotesize $\times$};
		\draw[draw=NavyBlue, fill=NavyBlue, fill opacity=.25, rounded corners] (1.70711 - \boxmargin, -0.707107  - \boxmargin) rectangle ++(2*\boxmargin,2*0.707107 + 2*\boxmargin);
		\node at (1.70711, 0.707107 + \boxmargin) [above] {$\ms{A}_{8,5}$};
		\node at (1.70711, -0.707107 - \boxmargin) [below] {$\ms{A}_{8,3}$};

		\node at (2, 0) {\footnotesize $\times$};
		\draw[draw=NavyBlue, fill=NavyBlue, fill opacity=.25, rounded corners] (2 - \boxmargin, 0 - \boxmargin) rectangle ++(2*\boxmargin,2*\boxmargin);
		\node at (2.025, .025) [above right] {$\ms{A}_{8,4}$};

	\end{scope}
\end{tikzpicture}
\caption{
	The position of the constants $\ms{A}_{r,\alpha} = \frac{r}{r+1} (1 - e^{2\pi\iu \frac{\alpha}{r}})$ for $r =7$ (left) and $r=8$ (right). In terms of their distance from the origin, they are all paired together, except $\ms{A}_{r,r/2}$ for even $r$.}
\label{fig:actions:rspin}
\end{figure}



\subsection*{Acknowledgements}
The authors would like to thank Raphaël Belliard, Philip Boalch, Veronica Fantini, Kohei Iwaki, Maxim Kontsevich, and David Sauzin for useful discussions. We are also grateful to IMJ-PRG, IPhT, IHES, SISSA, and the University of Trieste for their kind hospitality.

B.~E., A.~G., and P.~G. are supported by the ERC-SyG project ``Recursive and Exact New Quantum Theory'' (ReNewQuantum) which received funding from the European Research Council under the European Union's Horizon 2020 research and innovation programme under grant agreement No 810573. E.~G.-F. was funded by the public grant ``Jacques Hadamard'' as part of the ``Investissement d'Avenir'' project, reference ANR-11-LABX-0056-LMH, LabEx LMH, and supported by the ERC-StG project ``New Interactions of Combinatorics Through Topological Expansions'' (CombiTop) which received funding from the European Union's Horizon 2020 research and innovation programme under the grant agreement No 716083. D.~L. is funded by the University of Trieste, by the Swiss National Foundation Ambizione project ``Resurgent Topological Recursion, Enumerative Geometry and Integrable Hierarchies'' (EnTIRe) under the grant agreement No PZ00P2-202123, by the INdAM group GNSAGA and by the Trieste node of the INFN project MMNLP.
\addtocontents{toc}{\protect\setcounter{tocdepth}{2}}

\section{The Borel transform method}

In this section we present the \textit{Borel transform method}, namely a technique which allows to obtain the large-order asymptotics of the coefficients of a formal power series in terms of the singularity structure of its Borel transform. In order to do so, we first need to introduce several basic concepts from the theory of resurgence. Those are mostly taken from \cite{MS16,Mar}, to which we refer the interested reader for more details (see also \cite{ABS19} for a physicist-oriented comprehensive review of resurgence).

\subsection{The Borel transform and resurgent functions}

Define the \textit{Borel transform} as the following linear invertible operator acting on formal power series\footnote{
	Note that for our convenience we choose a convention for the Borel transform that is slightly different from the usual one used in the resurgence literature, which acts as $\hbar^{m+1} \mapsto \frac{s^m}{m!}$. This leads to no substantial differences with respect to the literature we cite in this section.
}
\begin{equation}
	\mf{B} \colon \C\bbraket{\hbar} \longrightarrow \C\bbraket{s} \,,
	\qquad
	\hbar^m \longmapsto \frac{s^m}{m!}\,.
\end{equation}
Throughout the text we denote (when possible) quantities in the $\hbar$-plane with a tilde and the corresponding quantities in the $s$-plane (also called Borel plane) with a hat. For instance, writing the starting series as
\begin{equation}
	\widetilde{\varphi}(\hbar) = \sum_{m \ge 0} \varphi_m \hbar^m \,,
\end{equation}
its Borel transform is given by 
\begin{equation}
	\widehat{\varphi}(s) = \sum_{m \ge 0} \varphi_m \frac{s^m}{m!} \,.
\end{equation}
Due to the division by the factorial, $\widehat{\varphi}$ is ``more likely'' to be convergent. This observation leads to the definition of Gevrey-1 series.

\begin{definition}
	A formal power series $\widetilde{\varphi}(\hbar) = \sum_{m \ge 0} \varphi_m \hbar^m$ is \textit{Gevrey-1} if there exists a constant $a > 0$ such that the large $m$ behaviour of the coefficients $\varphi_m$ is given by $\varphi_m = \bigO(m! \, a^{-m})$.
\end{definition}

The above definition is designed in such a way that the Borel transform maps Gevrey-1 series to holomorphic functions at the origin (and vice versa, every holomorphic function at the origin is the image of a Gevrey-1 series). It is then natural to study the analytic continuation of the Borel transform of Gevrey-1 series. The obstruction to analytic continuation is dictated by singularities in the Borel plane. As we are going to see shortly, singularities in the Borel plane play a crucial role in the theory. Among other things, they control the large order asymptotics of the coefficients of the original series.

In order to simplify our discussion, we restrict ourselves to a specific class of functions: that of simple resurgent functions.

\begin{definition}
	A \textit{resurgent function} is a function $\widehat{\varphi}$ which is holomorphic at the origin and satisfies the following property: along every ray issuing from the origin, there is a finite set of singular points such that $\widehat{\varphi}$ can be analytically continued along any path that follows the ray, while circumventing (from above or from below) the singular points. We call a resurgent function \textit{simple} if its singularities are either simple poles or logarithmic branch cuts. A Gevrey-1 series $\widetilde{\varphi}$ is called a simple resurgent series if its Borel transform $\widehat{\varphi}$ is a simple resurgent function.
\end{definition}

For simplicity, we restrict ourselves to simple resurgent functions with logarithmic singularities only. Including simple poles in the discussion does not pose any particular difficulties, as shown in \cite{MS16}. For a fixed simple resurgent series $\widetilde{\varphi}$ whose Borel transform has a logarithmic singularity at $A \in \C^{\times}$, we can write locally at $s = A$
\begin{equation}
	\widehat{\varphi}(s)
	=
	- \frac{S_A}{2\pi} \widehat{\varphi}_A(s - A) \, \log(s - A)
	+
	\text{holomorphic at } A \,,
\end{equation}
where $\widehat{\varphi}_A(s)$, called the \textit{minor} at $A$, is a holomorphic function at $s = 0$ (so that $\widehat{\varphi}_A(s - A)$ is holomorphic at $s = A$). Notice that we might want a specific choice of normalisation for $\widehat{\varphi}_A$, which motivates the introduction of an additional constant $S_A \in \C$ called the \textit{Stokes constant}. In what follows, we adopt the physics jargon and refer to the singularities $A$ as \textit{instanton actions}.

Since $\widehat{\varphi}_A$ is holomorphic at the origin, we can consider its Taylor expansion:
\begin{equation}
	\widehat{\varphi}_A(s) = \sum_{m \ge 0} \varphi_{A,m} \frac{s^m}{m!} \,.
\end{equation}
We consider $\widehat{\varphi}_A$ as the Borel transform of a formal power series
\begin{equation}
	\widetilde{\varphi}_A(\hbar) = \sum_{m \ge 0} \varphi_{A,m} \hbar^m \,.
\end{equation}
The key point so far is that, given a simple resurgent series $\widetilde{\varphi}$, we can associate to it the data of its instanton actions, Stokes constants, and minors:
\begin{equation}
	\widetilde{\varphi} \longrightarrow \bigl\{ \, (S_A, \widetilde{\varphi}_A) \, \bigr\}_{A \in \mc{A}} \,.
\end{equation}
Here $\mc{A}$ denotes the set of singularities of the Borel transform of $\widetilde{\varphi}$. We call this set of data the \textit{Borel plane singularity structure} of $\widetilde{\varphi}$.

\subsection{Exponential integrals}\label{subsec:exp:int}
A common source of simple resurgent functions, both in mathematics and in physics, is that of exponential integrals. We refer the reader to \cite[section~3]{ABS19} for an in-depth review of the main results and the available literature on the topic, and to \cite{KS22} for a short treatment. The result presented here essentially rely on results from \cite{DH02}. Let $X = \C \setminus \set{ p_1,\dots,p_M }$ be the punctured complex plane, $V \colon X \to \C$ an algebraic function (the potential) and $\mu$ an algebraic $1$-form on $X$. We are interested in computing integrals of the form
\begin{equation}
	\varphi(\hbar; \gamma) = \int_{\gamma} e^{-\frac{1}{\hbar} V(t)} \mu(t)
\end{equation}
for $\gamma$ a cycle in $X$. We assume that $V$ is Morse, i.e.~has only finitely many critical points $(t_{\alpha})_{\alpha \in I}$ with $V''(t_{\alpha}) \ne 0$, and we assume that the critical values $V_{\alpha} = V(t_{\alpha})$ are pairwise distinct. The examples we have in mind are the $r$-Airy integrals, for an integer $r \ge 2$, for which we have
\begin{equation}
	X = \C \,,
	\qquad\quad
	V(t) = \frac{t^{r+1}}{r+1} - t \,,
	\qquad\quad
	\mu(t) = dt \,.
\end{equation}
We are going to associate to each critical point a formal power series as follows. Fix a critical point $t_{\alpha}$. From the general theory of Morse functions, one can show that for all values of $\theta = \arg(\hbar)$ such that $\theta \ne \arg(V_{\beta} - V_{\alpha})$ there is a well-defined integral cycle $\gamma_{\alpha,\theta} \subset X$, called Lefschetz thimble or steepest descent path, passing through $t_{\alpha}$ along which the integral $\varphi(\hbar;\gamma_{\alpha,\theta})$ converges. We define the normalised integrals
\begin{equation}
	\varphi^{(\alpha)}(\hbar)
	\coloneqq
	\frac{e^{\frac{V_{\alpha}}{\hbar}}}{\sqrt{2\pi \hbar}}
	\int_{\gamma_{\alpha,\theta}} e^{-\frac{1}{\hbar} V(t)} \mu(t) \,.
\end{equation}
The function $\varphi^{(\alpha)}$ is holomorphic in $\hbar$ in each sector defined by the Stokes rays, that is the ray defined by $\arg(\hbar) = \arg(V_{\beta} - V_{\alpha})$. Again from the general theory of Morse functions and Laplace method, the above integral has a well-defined divergent asymptotic expansion series:
\begin{equation}
	\varphi^{(\alpha)}(\hbar)
	\sim
	\sum_{m \ge 0} \varphi_{m}^{(\alpha)} \, \hbar^{m}
	\eqqcolon
	\widetilde{\varphi}^{(\alpha)}(\hbar)
	\,,
\end{equation}
where the coefficients $\varphi_{m}^{(\alpha)}$ are independent of $\hbar$. The first coefficient is given by $\varphi_0^{(\alpha)} = V''(t_{\alpha})^{-1/2}$. Notice that the asymptotics is the same in every sector. As usual, let $\widehat{\varphi}^{(\alpha)}$ denote the Borel transform of $\widetilde{\varphi}^{(\alpha)}$. The analytic properties of $\widehat{\varphi}^{(\alpha)}$ are summarised as follows.

\begin{theorem}[Singularity structure for exponential integrals]\label{thm:sing:exp:int}
	The asymptotic expansion series $\widetilde{\varphi}^{(\alpha)}$ are simple resurgent. More precisely, the Borel transform $\widehat{\varphi}^{(\alpha)}$ has finitely many singularities at $A_{\alpha,\beta} = V_{\beta} - V_{\alpha}$ of logarithmic type, with behaviour given by
	\begin{equation}\label{eqn:sing:exp:int}
		\widehat{\varphi}^{(\alpha)}(s)
		=
		\frac{\ms{S}_{\alpha,\beta}}{2\pi\iu} \,
		\widehat{\varphi}^{(\beta)}(s - A_{\alpha,\beta}) \,
		\log(s - A_{\alpha,\beta})
		+
		\textup{holomorphic at } A_{\alpha,\beta} \,.
	\end{equation}
	Here $\ms{S}_{\alpha,\beta}$ are the integers defined by the intersection of two Lefschetz thimbles: $\ms{S}_{\alpha,\beta} = \gamma_{\alpha,\theta_+} \cdot \gamma_{\beta,\theta_- + \pi}$ for $\theta = \arg(A_{\alpha,\beta})$.
\end{theorem}

In other words, exponential integrals naturally generate simple resurgent series $\widetilde{\varphi}^{(\alpha)}$ labelled by critical points of the potential. Moreover, the Borel plane singularity structure of $\widetilde{\varphi}^{(\alpha)}$ is well understood in terms of:
\begin{itemize}
	\item \textit{instanton actions} -- difference of critical values for the location of singularities: $A_{\alpha,\beta} = V_{\beta} - V_{\alpha}$,

	\item \textit{Stokes constants} -- intersection numbers of Lefschetz thimbles: $\ms{S}_{\alpha,\beta}$,

	\item \textit{minors} -- the asymptotic expansion $\widetilde{\varphi}^{(\beta)}$ for the series attached to the singularity at $A_{\alpha,\beta}$.
\end{itemize}
As an example, let us apply the above result to the aforementioned generalised Airy integrals. The critical points are given by $V'(t) = t^r - 1 = 0$, that is $t_{\alpha} = \zeta^{\alpha}$ for $\zeta = e^{\frac{2\pi\iu}{r}}$ and $\alpha = 1,\dots, r$. The critical values are $V_{\alpha} = - \frac{r}{r+1} \zeta^{\alpha}$. According to the above theorem, the Borel transforms $\widehat{\varphi}^{(\alpha)}$ have logarithmic singularities at $A_{\alpha,\beta} = \frac{r}{r+1} (\zeta^{\alpha} - \zeta^{\beta})$ and Stokes constants
\begin{equation}\label{eqn:Stokes:rAiry}
	\ms{S}_{\alpha,\beta}
	= 
	\begin{cases}
		+1 & \text{if } \alpha > \beta \,,\\
		-1 & \text{if } \alpha < \beta \,.
	\end{cases}
\end{equation}
See \cite[section 3.2]{BEH03} for a detailed derivation of the intersection numbers for generalised Airy integrals. It should also be noted that the problem of the Stokes phenomenon for the higher Airy functions was addressed in the literature also beyond the exponential integral context, starting directly from the associated ODE. See for instance \cite{Tur50,Ohy95,Suz01}.

\begin{remark}[Normalisation conventions]
	In the theory of exponential integrals, it is customary to normalise \cref{eqn:sing:exp:int} with $\frac{\ms{S}_{\alpha,\beta}}{2\pi\iu}$, so that the numerator coincide with the (integral) intersection number of two Lefschetz thimbles. As we are going to see in the next section, in the theory of large-order asymptotics an alternative normalisation is more convenient. Specifically, we adopt the normalisation $- \frac{S_{\alpha,\beta}}{2\pi}$. Throughout the rest of the paper, we employ the latter normalisation while utilising \cref{eqn:sing:exp:int} to calculate the Stokes constants of interest.
\end{remark}

\subsection{Large-order asymptotics from the singularity structure}

We are now ready to establish the main theorem of this section, which gives the large-order asymptotics of the coefficients of a simple resurgent function $\widetilde{\varphi}$ in terms of the singularity structure of its Borel transform. The following result justifies the term ``resurgence'' too. Indeed, we have seen that the singularities of the Borel transform lead to new power series. These new series resurge in the original series through the large-order asymptotics of its coefficients. This is essentially an old theorem of Darboux, which relates the large-order asymptotics of the coefficients of an analytic function at the origin to the behaviour near the closest singularity. It is a well-known fact in the physics community, and perhaps not as widely known in the mathematics community. For the reader's convenience, we provide the proof here, which essentially relies on the application of Cauchy's theorem in the Borel plane.

\begin{theorem}[Borel transform method] \label{thm:large:order}
	Assume that $\widetilde{\varphi}$ is a simple resurgent series starting at order $\hbar^{\beta_0}$, with $\beta_0 \in \N$ called the critical exponent:
	\begin{equation}
		\widetilde{\varphi}(\hbar) = \sum_{m \ge 0} \varphi_{m} \hbar^{m + \beta_0} \,.
	\end{equation}
	Assume that its Borel transform $\widehat{\varphi}$ satisfies the following assumptions.
	\begin{itemize}
		\item Polynomial growth at infinity. There exists $\nu \in \R$ such that $\widehat{\varphi}(s) = \bigO(s^{\nu})$ as $s \to \infty$.

		\item Singularity structure. The function $\widehat{\varphi}$ has finitely many logarithmic singularities $A_1, \dots, A_n \in \mathbb{C}^{\times}$ around which it behaves as
		\begin{equation}\label{eqn:behaviour:sing}
			\widehat{\varphi}(s)
			=
			- \frac{S_i}{2\pi} \, \widehat{\varphi}^{(i)}(s - A_i) \, \log(s - A_i)
			+
			\textup{holomorphic at } A_i\,,
		\end{equation}
		for some constants $S_i$ and functions $\widehat{\varphi}^{(i)}$, the minors, which are holomorphic at the origin.

		\item Behaviour of the minors. The functions $\widehat{\varphi}^{(i)}$ have a polynomial growth at infinity and have a finite number of singularities avoiding the ray starting at the origin and passing through $A_i$. Moreover, their Taylor expansion at the origin have critical exponent $\beta_i \in \N$, i.e.~$\widehat{\varphi}^{(i)}(s) = \sum_{k = 0}^{\infty} \varphi^{(i)}_{k} \frac{s^{k + \beta_i}}{(k + \beta_i)!}$.
	\end{itemize}
	Then the large $m$ asymptotics of $\varphi_m$ is given by
	\begin{equation}\label{eqn:large:order}
		\varphi_m
		=
		\sum_{i = 1}^n
			\frac{S_i}{2\pi}
			\frac{\Gamma(m + \beta_0 - \beta_i)}{A_i^{m + \beta_0 - \beta_i}}
			\left(
			\sum_{k = 0}^{K}
				\frac{A_i^{k}}{(m + \beta_0 - \beta_i - 1)^{\underline{k}}} \, \varphi^{(i)}_{k} 
				+
				\bigO \biggl( \frac{1}{m^{K+1}} \biggr)
			\right) \,.
	\end{equation}
	Here $(x)^{\underline{k}} \coloneqq x (x-1) \cdots (x-m+1)$ denotes the falling factorial.
\end{theorem}

Before proceeding with a proof of the above result, it is instructive to write down explicitly the first terms of \cref{eqn:large:order}. Without loss of generality, suppose that the singularities are ordered as $|A_1| \le \cdots \le |A_n|$ from the closest to farthest distance from the origin. Then
\begin{equation*}
\begin{split}
	\varphi_m
	& = \phantom{+}
	\frac{S_1}{2\pi}
	\frac{\Gamma(m + \beta_0 - \beta_1)}{A_1^{m + \beta_0 - \beta_1}}
	\left(
		\varphi^{(1)}_{0}
		+
		\frac{A_1}{(m + \beta_0 - \beta_1 - 1)} \, \varphi^{(1)}_{1}
		+
		\frac{(A_1)^2}{(m + \beta_0 - \beta_1 - 1)^{\underline{2}}} \, \varphi^{(1)}_{2}
		+
		\cdots
	\right) \\
	& \phantom{=} + \cdots \\
	& \phantom{=} +
	\frac{S_n}{2\pi}
	\frac{\Gamma(m + \beta_0 - \beta_n)}{A_n^{m + \beta_0 - \beta_n}}
	\left(
		\varphi^{(n)}_{0}
		+
		\frac{A_n}{(m + \beta_0 - \beta_n - 1)} \, \varphi^{(n)}_{1}
		+
		\frac{(A_n)^2}{(m + \beta_0 - \beta_n - 1)^{\underline{2}}} \, \varphi^{(n)}_{2}
		+
		\cdots
	\right)\,.
\end{split}
\end{equation*}
We can see that the asymptotics presents an overall factorial growth, $\Gamma(m + \beta_0 - \beta_i) = \bigO(m!)$, and it is ``organised'' into different exponential growths corresponding to the rows governed by $1/A_i^{m + \beta_0 - \beta_1} = \bigO(A_i^{-m})$ with overall constant given by $S_i$. Moreover, the $i$-th row contains an asymptotic series in $1/m$ which scales as $A_i$ and whose coefficients are the Taylor coefficients of the corresponding minors.

\begin{proof}[Proof of \cref{thm:large:order}]
	Since $\widehat{\varphi}$ is holomorphic at the origin, we can extract its expansion coefficients through Cauchy's formula:
	\begin{equation*}
		\frac{\varphi_{m}}{(m + \beta_0)!}
		=
		\frac{1}{2\pi\iu} \oint_{\mc{C}_0} \frac{\widehat{\varphi}(s)}{s^{m + \beta_0 + 1}} \, ds \,.
	\end{equation*}
	Here $\mc{C}_0$ is a small contour around the origin oriented counter-clockwise. We can deform the contour, avoiding the logarithmic cuts starting at each singularity $A_i$ with partial Hankel contours $\mc{H}^{(R)}_i$ starting at $A_i$ in the direction $\theta_i \coloneqq \arg(A_i)$, connected by arcs whose union is denoted by $\gamma^{(R)}$ (see \cref{fig:deformation:contour}).

	\begin{figure}
	\centering
	\begin{tikzpicture}[scale=1.2]
		\draw [->] (-2,0) -- (2,0);
		\draw [->] (0,-2) -- (0,2);

		\draw[
			decoration={markings, mark=at position 0.625 with {\arrow{>}}},
			postaction={decorate}
			] (0,0) circle (.3cm);

		\node at (-.4,-.4) {\footnotesize$\mathcal{C}_0$};

		\begin{scope}[xshift=.4cm,yshift=.4cm,rotate=45]
			\draw[BrickRed,decoration = {zigzag,segment length = 1mm, amplitude = .3mm},decorate] (0,0) -- (1.53431457505,0);
			\node at (0,0) {$\bullet$};
			\node at (0,0) [below right] {\footnotesize$A_2$};
		\end{scope}
		\begin{scope}[xshift=0.5cm,yshift=-0.86602540378cm,rotate=-60]
			\draw[BrickRed,decoration = {zigzag,segment length = 1mm, amplitude = .3mm},decorate] (0,0) -- (1.1,0);
			\node at (0,0) {$\bullet$};
			\node at (0,0) [above] {\footnotesize$A_3$};
		\end{scope}
		\begin{scope}[xshift=-1.03923048454cm,yshift=0.6cm,rotate=150]
			\draw[BrickRed,decoration = {zigzag,segment length = 1mm, amplitude = .3mm},decorate] (0,0) -- (0.9,0);
			\node at (0,0) {$\bullet$};
			\node at (0,0) [below left] {\footnotesize$A_1$};
		\end{scope}

		\draw[->,decorate,decoration=snake] (3,0) -- (4,0);

		\begin{scope}[xshift=7cm]

			\draw[opacity=.5] (0,0) -- (-135:1.7);
			\node at (-145:1) {\footnotesize$R$};

			\draw [->] (-2,0) -- (2,0);
			\draw [->] (0,-2) -- (0,2);

			\draw[decoration={markings, mark=at position 0.3 with {\arrow{>}}},
				postaction={decorate}] (-55.1:1.7cm) arc (-55.1:40.1:1.7cm);
			\draw (50.1:1.7cm) arc (50.1:145.1:1.7cm);
			\draw (155.1:1.7cm) arc (155.1:295.1:1.7cm);

			\node at (65:1.2) {\footnotesize$\mathcal{H}_2^{(R)}$};
			\node at (-78:1.5) {\footnotesize$\mathcal{H}_3^{(R)}$};
			\node at (130:1.3) {\footnotesize$\mathcal{H}_1^{(R)}$};
			\node at (-150:2) {\footnotesize$\gamma^{(R)}$};

			\begin{scope}[xshift=.4cm,yshift=.4cm,rotate=45]
				\draw[BrickRed,decoration = {zigzag,segment length = 1mm, amplitude = .3mm},decorate] (0,0) -- (1.53431457505,0);
				\node at (0,0) {$\bullet$};
				\node at (0,0) [below right] {\footnotesize$A_2$};
			\end{scope}
			\begin{scope}[xshift=0.5cm,yshift=-0.86602540378cm,rotate=-60]
				\draw[BrickRed,decoration = {zigzag,segment length = 1mm, amplitude = .3mm},decorate] (0,0) -- (1.1,0);
				\node at (0,0) {$\bullet$};
				\node at (0,.2) [above] {\footnotesize$A_3$};
			\end{scope}
			\begin{scope}[xshift=-1.03923048454cm,yshift=0.6cm,rotate=150]
				\draw[BrickRed,decoration = {zigzag,segment length = 1mm, amplitude = .3mm},decorate] (0,0) -- (.9,0);
				\node at (0,0) {$\bullet$};
				\node at (0,0) [below left] {\footnotesize$A_1$};
			\end{scope}

			\clip (0,0) circle (1.7cm);

			\begin{scope}[xshift=.4cm,yshift=.4cm,rotate=45]
				\draw (2,.15) -- (0,.15) arc (90:270:.15) -- (0,-.15) -- (2,-.15);
			\end{scope}
			\begin{scope}[xshift=0.5cm,yshift=-0.86602540378cm,rotate=-60]
				\draw (2,.15) -- (0,.15) arc (90:270:.15) -- (0,-.15) -- (2,-.15);
			\end{scope}
			\begin{scope}[xshift=-1.03923048454cm,yshift=0.6cm,rotate=150]
				\draw (2,.15) -- (0,.15) arc (90:270:.15) -- (0,-.15) -- (2,-.15);
			\end{scope}
		\end{scope}
	\end{tikzpicture}
	\caption{
		Deformation of the contour $\mc{C}_0$ to the union of $\gamma^{(R)}$ and the partial Hankel contours $\mc{H}_i^{(R)}$ avoiding the logarithmic cuts.
	}
	\label{fig:deformation:contour}
	\end{figure}
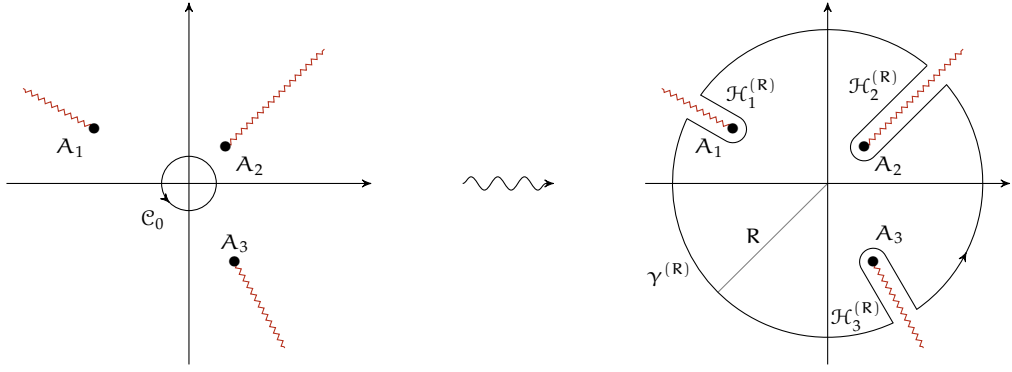

	Here $R$ is the radius of the arcs, and is assumed to be greater than $\max_{i=1,\dots,n}{|A_i|}$. Taking the limit $R \to +\infty$ gives zero contribution from the contours $\gamma^{(R)}$. Indeed, assuming $R$ large enough so that $|\widehat{\varphi}(s)| < C s^{\nu}$, we find
	\[
		\left| \frac{1}{2\pi \iu} \int_{\gamma_R} \frac{\widehat{\varphi}(s)}{s^{m + \beta_0 + 1}} \, ds \right|
		\le
		\frac{C}{R^{m + \beta_0 - \nu}}
	\]
	which tends to zero for $R \to +\infty$ for $m \gg 0$. Thus we only have to consider the Hankel contours $\mc{H}_i$ obtained in the limit $R \to +\infty$. A simple manipulation of the integral shows that
	\begin{multline*}
		\frac{\varphi_{m}}{(m + \beta_0)!}
		=
		\frac{1}{2\pi\iu} \sum_{i=1}^n \int_{\mc{H}_i} \frac{\widehat{\varphi}(s)}{s^{m + \beta_0 + 1}} \, ds 
		=
		-\frac{1}{2\pi\iu} \sum_{i=1}^n \frac{S_i}{2\pi}
		\int_{A_i}^{e^{\iu \theta_i }\infty}
			\Disc\left[ \widehat{\varphi}^{(i)}(s - A_i) \, \log(s - A_i)  \right] \frac{ds}{s^{m + \beta_0 + 1}}
		\\ 
		=
		\sum_{i=1}^n \frac{S_i}{2\pi}
		\int_{A_i}^{e^{\iu \theta_i }\infty}
			\frac{\widehat{\varphi}^{(i)}(s - A_i)}{s^{m + \beta_0 + 1}} \, ds
		=
		\sum_{i=1}^n \frac{S_i}{2\pi} e^{-\iu\theta_i (m+\beta_0)}
		\int_{0}^{+\infty}
			\frac{\widehat{\varphi}^{(i)}(e^{\iu \theta_i} s)}{(s + |A_i|)^{m + \beta_0 + 1}} \, ds \,.
	\end{multline*}
	In the second equality, we used the behaviour of $\widehat{\varphi}$ at $A_i$ given in \cref{eqn:behaviour:sing}. In the third equality, we used the fact that the discontinuity is completely governed by the logarithm, and gives a contribution of $-2\pi\iu$. The last equality is simply a change of variable in the integral. The goal now is to approximate $\widehat{\varphi}^{(i)}$ with the first $K$ terms of its Taylor expansion, and estimate the error using Watson's lemma. Define the tail of $\widehat{\varphi}^{(i)}$ by removing the first $K$ terms of its Taylor expansion around the origin: $\widehat{\varphi}^{(i)}(s; K) \coloneqq \widehat{\varphi}^{(i)}(s) - \sum_{k=0}^{K} \varphi^{(i)}_{k} \frac{s^{k + \beta_i}}{(k + \beta_i)!}$. Then
	\begin{multline}\label{eqn:to:estimate}
		\frac{\varphi_{m}}{(m + \beta_0)!}
		-
		\sum_{i=1}^n \frac{S_i}{2\pi} e^{-\iu\theta_i (m + \beta_0)}
			\sum_{k=0}^K
				\frac{\varphi^{(i)}_k e^{\iu \theta_i (k + \beta_i)}}{(k + \beta_i)!}
				\int_{0}^{+\infty}
				\frac{s^{k + \beta_i}}{(s + |A_i|)^{m + \beta_0 + 1}} \, ds \\
		=
		\sum_{i=1}^n \frac{S_i}{2\pi} e^{-\iu\theta_i (m + \beta_0)}
			\int_{0}^{+\infty}
			\frac{\widehat{\varphi}^{(i)}(e^{\iu \theta_i} s;K)}{(s + |A_i|)^{m + \beta_0 + 1}} \, ds \,.
	\end{multline}
	The left-hand side of the equation above can be explicitly computed using the Euler-type integral $\frac{1}{\Gamma(\mu + 1)} \int_{0}^{+\infty} \frac{s^{\mu}}{(s + A)^{\nu + 1}} ds = A^{\mu-\nu} \, \frac{ \Gamma(\nu - \mu) }{ \Gamma(\nu + 1) }$ valid for $A > 0$ and $\nu > \mu > 0$. We thus find
	\begin{equation*}
		\frac{\varphi_{m}}{(m + \beta_0)!}
		-
		\frac{1}{(m + \beta_0)!}
		\sum_{i=1}^n
			\frac{S_i}{2\pi} 
			\sum_{k=0}^K
			\varphi^{(i)}_k
			\frac{\Gamma(m + \beta_0 - \beta_i - k)}{A_i^{m + \beta_0 - \beta_i - k}} \,.
	\end{equation*}
	We are now left with an estimate of the right-hand side of \eqref{eqn:to:estimate}. In the $i$-th term of the sum, let us set $s = |A_i|(e^t - 1)$. We obtain
	\begin{equation*}
		\sum_{i=1}^n \frac{S_i}{2\pi} e^{-\iu\theta_i (m+\beta_0)}
			\int_{0}^{+\infty}
			\frac{\widehat{\varphi}^{(i)}(e^{\iu \theta_i} s;K)}{(s + |A_i|)^{m + \beta_0 + 1}} \, ds
		=
		\sum_{i=1}^n \frac{S_i}{2\pi}
			\frac{1}{A_i^{m + \beta_0}}
			\int_{0}^{+\infty}
				e^{-t(m + \beta_0)} \widehat{\varphi}^{(i)}( A_i (e^t - 1) ;K) \, dt \,.
	\end{equation*}
	We can now apply Watson's \cref{Watson:lemma} to estimate the integral: since $\widehat{\varphi}^{(i)}( A_i(e^t - 1) ;K) = \bigO(t^{K+\beta_i+1})$ as $t \to 0^+$ (recall that $\widehat{\varphi}^{(i)}(s;K)$ is $\widehat{\varphi}^{(i)}(s)$ pruned of its first $K$ expansion coefficients at the origin) and $\widehat{\varphi}^{(i)}( A_i(e^t - 1) ;K)$ behaves at most exponentially at infinity, we find 
	\begin{equation*}
		\int_{0}^{+\infty}
			e^{-t(m + \beta_0)} \widehat{\varphi}^{(i)}( A_i(e^t - 1) ;K) \, dt
		=
		\bigO\biggl( \frac{1}{(m + \beta_0)^{K+\beta_i+2}} \biggr)
		=
		\bigO\biggl( \frac{1}{m^{K+\beta_i+2}} \biggr) \,.
	\end{equation*}
	After multiplication by $(m+\beta_0)!$ and some algebraic manipulation, we find the thesis.
\end{proof}

The estimate for the error term in the asymptotics was a consequence of Watson's lemma for Laplace-type integrals, which we report here for the reader's convenience. See \cite{Olv97} for more details.

\begin{lemma}[Watson]\label{Watson:lemma}
	Let $f \colon \R_{+} \to \C$ be a continuous function satisfying the following assumptions.
	\begin{itemize}
		\item Polynomial growth at the origin. There exists an integer $K$ such that $f(t) = \bigO( t^{d} )$ as $t \to 0^+$.

		\item Exponential growth at infinity. There exists $\nu \in \R$ such that $f(t) = \bigO( e^{t\nu} )$ as $t \to + \infty$.
	\end{itemize}
	Then, as $x \to +\infty$:
	\begin{equation}
		\int_{0}^{+\infty} e^{-tx} \, f(t) \, dt
		=
		\bigO\biggl( \frac{1}{x^{d+1}} \biggr) \,.
	\end{equation}
\end{lemma}

\begin{remark}[{Gevrey-2 series}]\label{rem:Gevrey:2}
	In the next sections, we are going to apply the Borel transform method to formal power series with vanishing even or odd coefficients. In other words, up to relabelling of the coefficients, we have
	\begin{equation}
		\widetilde{\varphi}(\hbar)
		=
		\sum_{g \ge 0} \varphi_{g} \hbar^{2g + \beta_0} \,.
	\end{equation}
	Formal power series with $\varphi_{g} = \bigO((2g)! a^{-g})$ are also known as \textit{Gevrey-2} series. In this case, the large $g$ asymptotics, that is \cref{eqn:large:order}, reads
	\begin{equation}
		\varphi_{g}
		=
		\sum_{i = 1}^n
			\frac{S_i}{2\pi}
			\frac{\Gamma(2g + \beta_0 - \beta_i)}{A_i^{2g + \beta_0 - \beta_i}}
			\left(
			\sum_{k = 0}^{K}
				\frac{A_i^{k}}{(2g + \beta_0 - \beta_i - 1)^{\underline{k}}} \, \varphi^{(i)}_{k} 
				+
				\bigO \biggl( \frac{1}{g^{K+1}} \biggr)
			\right) \,.
	\end{equation}
	We also remark that the above theorem can be generalised to the case of $\widehat{\varphi}$ admitting simple poles. Moreover, one can allow for the critical exponents $\beta_0$ and $\beta_i$ to be any real numbers by suitably adapting the definition of the Borel transform.
\end{remark}

\subsection{Algebraic properties of resurgent series}\label{subsec:Borel:properties}

So far we have established that the large-order asymptotics of the coefficients of a simple resurgent series is entirely determined by its Borel singularity structure. One advantage of working within the class of simple resurgent functions is that they exhibit favourable properties under addition and multiplication. As will become evident in the subsequent sections, this aspect is of particular significance to us, since our goal is to deduce the singularity structure of specific power series. These series are derived as sums of products of simple resurgent ``building blocks'', whose singularity structure is completely under control. For all these reasons, let us discuss some algebraic properties of simple resurgent series.

The Borel transform is by definition a linear operation. On the other hand, the Borel transform does not respect the Cauchy product of formal power series, but it rather sends it to the convolution product.

\begin{definition}
	Given $\widehat{\phi}, \widehat{\psi} \in \C\bbraket{s}$, their \textit{convolution product} is
	\begin{equation}
		\widehat{\phi} \ast \widehat{\psi}
		\coloneqq
		\mf{B} \bigl[ \widetilde{\phi} \cdot \widetilde{\psi} \bigr] \,,
	\end{equation}
	with $\widehat{\phi} = \mf{B}[ \widetilde{\phi} ]$ and $\widehat{\psi} = \mf{B}[ \widetilde{\psi} ]$.
\end{definition}

It is easy to check that, if both $\widehat{\phi}$ and $\widehat{\psi}$ converge in a disc of radius $R$ around the origin, then for all $s$ in such a disc we find $(\widehat{\phi} \ast \widehat{\psi})(s) = \frac{d}{ds} \int_{0}^{s} \widehat{\phi}(\sigma) \widehat{\psi}(s-\sigma) \, d\sigma$. In other words, $\widehat{\phi} \ast \widehat{\psi}$ is (the derivative of) the usual convolution product of functions, hence the name. The next two theorems precisely determine how the analytic properties of simple resurgent functions transform under convolution.

\begin{lemma}[{Analyticity of the convolution product \cite[lemma~5.54]{MS16}}]\label{lemma:conv}
	Let $\Omega$ be a star-shaped open subset of $\C$ (that is, for every $s\in\Omega$, the line segment $[0,s]$ is contained in $\Omega$), and $\widehat{\phi},\widehat{\psi}$ be holomorphic in $\Omega$. Then their convolution product is also holomorphic in $\Omega$.
\end{lemma}

\begin{theorem}[{Product of simple resurgent series \cite[theorem~6.83]{MS16}}]\label{theorem:prod}
	Let $\mc{A}$ and $\mc{B}$ be non-empty closed discrete subsets of $\C^{\times}$, and denote $\mc{C} \coloneqq \mc{A} \cup \mc{B} \cup \left( \mc{A} + \mc{B} \right)$, with $\mc{A} + \mc{B} \coloneqq \set{ A + B | A \in \mc{A}, \, B \in \mc{B}  }$. Let $\widetilde{\phi}$ and $\widetilde{\psi}$ be simple resurgent series with Borel plane singularities in $\mc{A}$ and $\mc{B}$ respectively, such that
	\begin{equation}
	\begin{aligned}
		\widehat{\phi}(s)
		& =
		- \frac{S_{A}}{2\pi} \widehat{\phi}_{A}(s - A) \, \log(s - A)
		+
		\textup{holomorphic at } A
		&& \qquad
		\forall A \in \mc{A} \,, \\
		\widehat{\psi}(s)
		& =
		- \frac{S_{B}}{2\pi} \widehat{\psi}_{B}(s - B) \, \log(s - B)
		+
		\textup{holomorphic at } B
		&&\qquad \forall B \in \mc{B} \,.
	\end{aligned}
	\end{equation}
	If $\mc{C}$ is closed and discrete, then $\widetilde{\phi} \cdot \widetilde{\psi}$ is also a simple resurgent series with Borel plane singularities in $\mc{C}$ and
	\begin{equation}
	\begin{aligned}
		\bigl( \widehat{\phi} \ast \widehat{\psi} \bigr)(s)
		& =
		- \frac{S_{A}}{2\pi}
		\bigl( \widehat{\phi}_{A} \ast \widehat{\psi} \bigr)(s - A) \,
		\log(s - A)
		+
		\textup{holomorphic at } A
		&& \qquad
		\forall A \in \mc{A} \,, \\
		\bigl( \widehat{\phi} \ast \widehat{\psi} \bigr)(s)
		& =
		- \frac{S_{B}}{2\pi}
		\bigl( \widehat{\phi} \ast \widehat{\psi}_{B} \bigr)(s - B) \,
		\log(s - B)
		+
		\textup{holomorphic at } B
		&& \qquad
		\forall B \in \mc{B} \,.
	\end{aligned}
	\end{equation}
\end{theorem}

\begin{remark}[Singularities on the principal sheet]\label{rem:princ:sheet}
	A priori, \cref{theorem:prod} implies that the set of Borel plane singularities of $\widetilde{\phi} \cdot \widetilde{\psi}$ not only contain elements of $\mc{A}$ and $\mc{B}$, but also of $\mc{A} + \mc{B}$. Nonetheless, \cref{lemma:conv} ensures that, if all elements of $\mc{A} \cup \mc{B}$ lay on distinct rays issuing from the origin, then all singularities from $\mc{A} + \mc{B}$ will not be on the principal sheet of the Borel plane. Thus, according to \cref{thm:large:order}, they do not contribute to the large-order asymptotic behaviour of the coefficients of $\widetilde{\phi} \cdot \widetilde{\psi}$.
\end{remark}

\section{Large genus asymptotics of \texorpdfstring{$\psi$}{psi}-class intersection numbers}\label{sec:WK}
The goal of this section is to prove the large genus asymptotics of the $\psi$-class intersection numbers, including all subleading corrections. To this end, we apply the Borel transform method to the $n$-point function, which in turn is built out of formal solutions of the Airy differential equations through the determinantal formulae. Thus, all terms appearing in Aggarwal's asymptotic formula can be explained in terms of data appearing in such formal solutions of the Airy ODE.

We start by recalling some facts about the Airy functions and their Borel transform, as well as the determinantal formulae for the $\psi$-class intersection numbers. We then proceed with the resurgent analysis of $n$-point functions, i.e.~their singularity structure on the Borel plane.

\subsection{Airy functions}
Consider the ($\hbar$-dependent) \textit{Airy ODE}:
\begin{equation}
	\biggl( \biggl( \hbar \frac{d}{dx} \biggr)^2 - x \biggr) \psi(x;\hbar) = 0 \,.
\end{equation}
A general solution is given by the Airy integral: $\int_{\gamma} e^{-\frac{1}{\hbar}V(t,x)} dt$, with $V(t,x) \coloneqq \frac{t^3}{3} - xt$ and $\gamma$ a properly chosen integration contour. In this section, we are interested in a basis of formal (or WKB) solutions, i.e.~the asymptotic expansion of the (properly normalised) Airy integrals with $\gamma$ being a Lefschetz thimble. In the context of Lax operator formalism these are called the \textit{wave functions} or Baker--Akhiezer functions. Following the prescription outlined in \cref{subsec:exp:int}, such formal solutions and their derivatives are given by
\begin{equation}
\begin{split}
	\psi_{\pm}(x;\hbar)
	&= 
	\frac{e^{\mp \frac{V(x)}{\hbar}}}{\sqrt{2}} x^{-1/4}
	\sum_{k = 0}^{\infty}
		\frac{1}{864^k}\frac{(6k)!}{(2k)!(3k)!} \left(\mp \frac{\hbar}{V(x)} \right)^{k} \,, \\
	\psi_{\pm}'(x;\hbar)
	&= 
	\pm \frac{e^{\mp \frac{V(x)}{\hbar}}}{\sqrt{2}} x^{1/4}
	\sum_{k = 0}^{\infty}
		\frac{1}{864^k}\frac{(6k)!}{(2k)!(3k)!} \frac{1+6k}{1-6k} \left(\mp \frac{\hbar}{V(x)} \right)^{k} \,, \\
\end{split}
\end{equation}
where $\pm V(x) \coloneqq \mp \frac{2}{3} x^{3/2}$ are the critical values of the potential at the critical point $t = \pm x^{1/2}$. The reader can recognise $\psi_{-}$ and $\psi_{+}$ as the asymptotic expansions in the appropriate regions of the Airy and Bairy functions respectively. The normalisation is fixed so that the Wronskian is worth $\psi_- \psi_+' - \psi_-' \psi_+ = 1$.

Consider the Borel transform of the Airy functions. To start with, let us separate the exponential part by setting $\psi_{\pm}(x;\hbar) \eqqcolon \exp(\mp \frac{1}{\hbar} V(x)) \, \widetilde{\psi}(x;\hbar)$, so that $\widetilde{\psi}_{\pm}$ is a formal power series in $\hbar$. Similarly for the derivatives. Then the Borel transforms of $\widetilde{\psi}_{\pm}$ and $\widetilde{\psi}_{\pm}'$ are expressed in terms of Gauss hypergeometric functions as:
\begin{equation}
	\widehat{\psi}_{\pm}(x;s)
	=
	\frac{x^{-1/4}}{\sqrt{2}} \, \pFq{2}{1} \left( \tfrac{5}{6},\tfrac{1}{6} ; 1 ; \pm \frac{s}{A(x)} \right) \,,
	\qquad
	\widehat{\psi}_{\pm}'(x;s)
	=
	\pm \frac{x^{1/4}}{\sqrt{2}} \, \pFq{2}{1} \left( \tfrac{7}{6},-\tfrac{1}{6} ; 1 ; \pm \frac{s}{A(x)} \right) \,,
\end{equation}
where $\pm A(x) \coloneqq \pm \frac{4}{3} x^{3/2}$ are the instanton actions.

From the above explicit expression, we deduce that the Borel transforms $\widehat{\psi}_{\pm}$ and $\widehat{\psi}_{\pm}'$ converge in a disc of radius $|A(x)|$ centred at the origin, and they can be extended analytically to functions with logarithmic singularities at $s = \pm A(x)$. From a direct analysis of the integral representation of the Gauss hypergeometric functions, one can prove the following behaviour at the singularities:
\begin{equation}\label{eqn:Borel:Airy}
\begin{aligned}
	\widehat{\psi}_{\pm}(x;s)
	& =
		- \frac{S}{2\pi} \,
		\widehat{\psi}_{\mp}\bigl( x;s \mp A(x) \bigr)
		\log\bigl( s \mp A(x) \bigr)
		+
		\text{holomorphic at} \pm A(x)\,, \\
	\widehat{\psi}_{\pm}'(x;s)
	& =
		- \frac{S}{2\pi} \,
		\widehat{\psi}_{\mp}'\bigl( x;s \mp A(x) \bigr)
		\log\bigl( s \mp A(x) \bigr)
		+
		\text{holomorphic at} \pm A(x) \,,
\end{aligned}
\end{equation}
where $S = 1$ is the Stokes constant. In other words, $\widetilde{\psi}_{\pm}$ and $\widetilde{\psi}_{\pm}'$ are simple resurgent functions with finitely many logarithmic singularities, with Stokes constants and minors that are fully under control. Notice that the above behaviour also follows from \cref{thm:sing:exp:int} and the (properly normalised) integral representation of $\psi_{\pm}$ and $\psi_{\pm}'$, for which the explicit expressions of the Borel transforms is not needed.

In the remaining part of this section, we often keep the generic notation $S$ and $A(x)$ to emphasise how their actual values are not crucial in the following analysis. This fact will allow us to compute the large genus asymptotics of $\Theta$-class intersection numbers in \cref{sec:Norbury} by simply changing the values of $S$ and $A(x)$ to the appropriate ones.

\subsection{Airy correlators}
The Airy ODE can be re-written as a first order $2 \times 2$ system in terms of the companion matrix $\mathcal{D}$:
\begin{equation}
	\hbar \frac{d}{dx} \Psi(x;\hbar) = \mathcal{D}(x) \Psi(x;\hbar) \,,
	\qquad
	\mathcal{D}(x)
	\coloneqq
	\begin{pmatrix}
		0 & 1 \\
		x & 0
	\end{pmatrix} \,,
	\qquad
	\Psi(x;\hbar)
	=
	\begin{pmatrix}
		\psi_-(x;\hbar) & \psi_+(x;\hbar) \\
		\psi_-'(x;\hbar) & \psi_+'(x;\hbar)
	\end{pmatrix} \,.
\end{equation}
We refer to $\Psi$ as the \textit{wave matrix}. Notice that with the chosen normalisation of the Wronskian, we have $\det{\Psi} = 1$. We are now interested in computing the correlators associated to the above differential system. To this end, define the matrix
\begin{equation}\label{eqn:M:matrix}
	M(x;\hbar)
	\coloneqq
	\Psi(x;\hbar)
	E
	\Psi^{-1}(x;\hbar) \,,
	\qquad\qquad
	E
	\coloneqq
	\frac{1}{2} \begin{pmatrix}
		1 & 0 \\
		0 & -1
	\end{pmatrix} .
\end{equation}
Notice that $E$ is a generator of the Cartan subalgebra of $\mf{sl}_2(\C)$. More explicitly $M$ is given by
\begin{equation}
	M
	=
	\begin{pmatrix}
		\frac{1}{2}(\psi_+' \psi_- + \psi_+ \psi_-') & \psi_+ \psi_- \\
		\psi_+' \psi_-' & - \frac{1}{2}(\psi_+' \psi_- + \psi_+ \psi_-')
	\end{pmatrix}
	=
	\begin{pmatrix}
		\frac{1}{2}(\widetilde{\psi}_+' \widetilde{\psi}_- + \widetilde{\psi}_+ \widetilde{\psi}_-') &
		\widetilde{\psi}_+ \widetilde{\psi}_- \\
		\widetilde{\psi}_+' \widetilde{\psi}_-' &
		- \frac{1}{2}(\widetilde{\psi}_+' \widetilde{\psi}_- + \widetilde{\psi}_+ \widetilde{\psi}_-')
	\end{pmatrix} .
\end{equation}
In the second equation, we used the simple but crucial fact that the exponential factors $e^{\pm \frac{1}{\hbar}V(x)}$ cancel out in quadratic expressions involving functions labelled by `$+$' and `$-$'. In other words, we can substitute the $\psi$'s with their respective $\widetilde{\psi}$'s. Thus, $M$ is a matrix-valued formal power series in $\hbar$.

We remark that, from a more geometric point-of-view, $\Psi$ can be identified with a formal flat section for the $\mf{sl}_2(\C)$-connection $\nabla = \hbar d - \mathcal{D} dx$ on the trivial bundle over the projective line. Besides, $M$ can be seen as a flat section of the adjoint bundle, i.e.~it satisfies the differential system $\hbar dM - [\mathcal{D},M] dx = 0$.

We are now ready to define the correlators associated to the Airy differential system. To the best of our knowledge, correlators in the context of ODEs appeared for the first time in \cite{BE}, inspired by the analogy with random matrix models (see for instance \cite{Dys70,Met04}). Later, they were re-discovered in the context of tau-functions of the KdV hierarchy in \cite{BDY16}.

\begin{definition}[Airy correlators]
	Define the $n$-point \textit{Airy correlators} $W_n$ as\footnote{In the literature, a different convention is used for $n = 2$. Since the the difference is of order $\hbar^0$ and we are interested in the $\hbar$-asymptotic expansion, we can ignore the correction term. Another difference in the literature is the convention for the matrix $E$, which is sometimes shifted by a matrix proportional to the identity $\lambda \Id$. Thus, $M$ is also shifted by the same matrix $\lambda \Id$. One can prove that such a shift leaves $W_n$ invariant, see \cite[appendix~A]{EM}.}
	\begin{equation}\label{eqn:n:pnt:Airy}
	\begin{aligned}
		& W_1(x_1;\hbar)
		\coloneqq
		- \frac{1}{\hbar} \Tr{\bigl( \mathcal{D}(x_1) M(x_1;\hbar) \bigr)} \,, \\
		& W_n(x_1,\dots,x_n;\hbar)
		\coloneqq
		(-1)^{n-1} \sum_{\sigma \in S_n^{\textup{cyc}}}
		\frac{\Tr{ \bigl( \prod_{i=1}^n M(x_{\sigma^i(1)};\hbar) } \bigr) }{\prod_{i=1}^n ( x_{i} - x_{\sigma(i)} )}
		\qquad\quad
		\text{for }n \ge 2 \,.
	\end{aligned}
	\end{equation}
	We refer to the above formula as the \textit{determinantal formula}. Here $S_n^{\textup{cyc}}$ denotes the set of cyclic permutations, that is the subset of $S_n$ consisting of all permutations with a single cycle of length $n$.
\end{definition}

It can be shown that $W_n$ is regular along $x_i = x_j$ (see for instance \cite{BBE15}). Moreover, as $M$ is a formal power series in $\hbar$, so is $W_{n}$. The next theorem makes the $\hbar$-dependence more precise, and most importantly it establishes an enumerative-geometric interpretation of the expansion coefficients: they correspond to $\psi$-class intersection numbers. For this reason, the $\hbar$-expansion of $W_{n}$ is also called a \textit{genus expansion}. A proof of the theorem can be deduced from the original work of Kontsevich \cite{Kon92}, together with the fact that correlators of matrix models are computed by determinantal formulae (fact already hinted in \cite[section~10.5]{EO07}). The theorem was later re-discovered in \cite{BDY16} in the context of solutions of the KdV hierarchy.

\begin{theorem}[{Genus expansion of the Airy correlators}]\label{thm:genus:expns:Airy}
	The $n$-point Airy correlators $W_{n}$ admit the following $\hbar$-expansion:
	\begin{equation}\label{eqn:genus:expns:Airy}
		W_{n}(\bm{x};\hbar)
		=
		\sum_{g = 0}^{\infty}
			\hbar^{2g-2+n} \, W_{g,n}(\bm{x}) \,.
	\end{equation}
	Moreover, $W_{g,n}$ stores $\psi$-class intersection numbers: if $2g-2+n>0$, 
	\begin{equation}
		W_{g,n}(\bm{x})
		=
		(-2)^{-(2g-2+n)}
		\sum_{\substack{ d_{1},\dots,d_n \ge 0 \\ |d| = 3g-3+n }}
			\braket{ \tau_{d_1} \cdots \tau_{d_n} }
			\prod_{i=1}^n \frac{(2d_i+1)!!}{2 \, x_i^{d_i+3/2}} \,.
	\end{equation}
\end{theorem}

\begin{example}[{$1$-point correlators}]
	A direct computation shows that
	\begin{equation*}
	\begin{split}
		W_1(x;\hbar)
		&=
		\frac{x \, \widetilde{\psi}_- \widetilde{\psi}_+ - \widetilde{\psi}_-' \widetilde{\psi}_+'}{\hbar}
		=
		\frac{1}{\hbar} x^{1/2}
		- \hbar \tfrac{1}{32} x^{-5/2}
		- \hbar^3 \tfrac{105}{2048} x^{-11/2}
		- \hbar^5 \tfrac{25025}{65536} x^{-17/2}
		+ \bigO(\hbar^7) \,.
	\end{split}
	\end{equation*}
	In particular, the above expansion agrees with the intersection numbers:
	\begin{alignat*}{4}
		W_{0,1}(x) &= x^{1/2} \,,
		\qquad
		& W_{1,1}(x) &= - \tfrac{1}{32} x^{-5/2} \,,
		\qquad
		& W_{2,1}(x) &= - \tfrac{105}{2048} x^{-11/2} \,,
		\qquad
		& W_{3,1}(x) &= - \tfrac{25025}{65536} x^{-17/2} \,, \\
		&
		& \braket{\tau_1} &= \tfrac{1}{24} \,,
		& \braket{\tau_4} &= \tfrac{1}{1152} \,,
		& \braket{\tau_7} &= \tfrac{1}{82944} \,.
	\end{alignat*}
	More generally, an explicit computation of the expansion coefficients of $W_1$ from those of the Airy and Bairy functions shows the well-known closed formula $\braket{\tau_{3g-2}} = \frac{1}{24^g g!}$ for $1$-point intersection numbers.
\end{example}

\begin{remark}[Connection with topological recursion and the spectral curve]
	The reader familiar with topological recursion can recognise $W_{g,n}$ as essentially the topological recursion correlators computed from the Airy spectral curve $(\P^1, x(z) = z^2, y(z) = z, \omega_{0,2}(z_1,z_2) = \frac{dz_1 dz_2}{(z_1 - z_2)^2})$. More precisely:
	\begin{equation}
		\omega_{g,n}(\bm{z}) = W_{g,n}(\bm{x}) \, dx_1 \cdots dx_n \big|_{x_i = z_i^2} \,.
	\end{equation}
	This can be explained by the fact that both correlators defined through determinantal formulae and topological recursion satisfy the same loop equations \cite{BE,BEO15}. We also emphasise that the wave functions, and consequently the correlators, are defined up to a choice of square root of $x$. In other words, they are well-defined on the spectral curve $\set{x = y^2}$. Throughout the rest of the section, we will work with a fixed choice of square root of $x$, often denoted as $\sqrt{x} = z$.
\end{remark}

The next result gives an alternative expression for the $n$-point Airy correlators in terms of the Airy kernel, which plays an important role in the Tracy--Widom law and in extreme values statistics \cite{TW94a}.

\begin{lemma}[Determinantal formula in kernel form]\label{lemma:kernel}
	For $n \ge 2$, the $n$-point Airy correlators are given by the two equivalent expressions
	\begin{equation}
		W_n(\bm{x};\hbar)
		=
		(-1)^{n-1} \sum_{\sigma \in S_n^{\textup{cyc}}}
			\prod_{i = 1}^n K_{+,-}(x_i,x_{\sigma(i)};\hbar)
		=
		(-1)^{n-1} \sum_{\sigma \in S_n^{\textup{cyc}}}
			\prod_{i = 1}^n K_{-,+}(x_i,x_{\sigma(i)};\hbar)\,,
	\end{equation}
	where $K_{\pm,\mp}(x,y;\hbar)$, known as the Airy kernels, are defined as
	\begin{equation}
		K_{\pm,\mp}(x,y;\hbar)
		\coloneqq
		\frac{\widetilde{\psi}_{\pm}'(x;\hbar) \widetilde{\psi}_{\mp}(y;\hbar) - \widetilde{\psi}_{\pm}(x;\hbar) \widetilde{\psi}_{\mp}'(y;\hbar)}{x-y} \,. \\
	\end{equation}
	Moreover, the kernels are related by the parity relation $K_{-,+}(x,y;\hbar) = - K_{+,-}(x,y;-\hbar)$.
\end{lemma}

\begin{proof}
	The two expression for $W_n$ are equivalent, due to the cyclicity of the sum and the symmetry property $K_{+,-}(x,y;\hbar) = K_{-,+}(y,x;\hbar)$. Thus, we only need to prove one equality. Let us start from the definition of the $n$-point correlators, \cref{eqn:n:pnt:Airy}. Since the formula is invariant under shifts of $M$ by matrices proportional to the identity \cite[appendix~A]{EM}, we can substitute $M$ by
	\[
		M - \tfrac{1}{2} \, \Id
		=
		\begin{pmatrix}
			\widetilde{\psi}_+ \widetilde{\psi}_-' & - \widetilde{\psi}_+ \widetilde{\psi}_- \\
			\widetilde{\psi}_+' \widetilde{\psi}_-' & - \widetilde{\psi}_+' \widetilde{\psi}_-
		\end{pmatrix}
		=
		\begin{pmatrix}
			\widetilde{\psi}_+ \\
			\widetilde{\psi}_+'
		\end{pmatrix}
		\otimes
		\begin{pmatrix}
			\widetilde{\psi}_-' \\
			- \widetilde{\psi}_-
		\end{pmatrix} \,.
	\]
	Here $u \otimes v$ denotes the outer product of two vectors. Inserting the above expression in the determinantal formula and using the identity $\Tr( \prod_{k=1}^{n} u_{k} \otimes v_k ) = \prod_{k=1}^{n} \braket{v_k,u_{k+1}}$ (with the identification $u_{n+1} = u_1$) yields the thesis. To conclude, the relations $\psi_{-}(x;\hbar) = \psi_{+}(x;-\hbar)$ and $\psi_{-}'(x;\hbar) = - \psi_{+}'(x;-\hbar)$ imply the parity property.
\end{proof}

\subsection{Singularity structure and large genus asymptotics}
The goal of this section is to understand the large genus behaviour of the coefficients in the $\hbar$-expansion of the correlators via the Borel transform method. Since the singularity structure of the wave functions is well understood (see \cref{eqn:Borel:Airy}), one can deduce a similar statement for the correlator $W_n$ using the algebraic properties of the Borel transform, as explained in \cref{subsec:Borel:properties}. In the following analysis we assume $n \ge 2$. The case $n = 1$ can be considered separately and leads to the same result. Moreover, we assume $(x_1,\dots,x_n)$ to be in a generic position, so that the instanton actions lay on distinct rays issuing from the origin, therefore realising the conditions of \cref{rem:princ:sheet}.

\begin{proposition}[Singularity structure of the Airy correlators]
	The Borel transform $\widehat{W}_{n}(\bm{x};s)$ of the $n$-point Airy correlator is simple resurgent. More precisely, on the principal sheet $\widehat{W}_{n}$ has $2n$ logarithmic singularities located at $s = \pm A(x_i)$ and such that
	\begin{equation}
		\widehat{W}_{n}(\bm{x};s)
		=
		- \frac{S}{2\pi} \,
		\widehat{W}^{(\pm,i)}_{n}\bigl( \bm{x}; s \mp A(x_i) \bigr) \,
		\log\bigl( s \mp A(x_i) \bigr)
		+
		\textup{holomorphic at } \pm A(x_i) \,,
	\end{equation}
	where:
	\begin{itemize}
		\item $A(x) = \frac{4}{3} x^{3/2}$ and $S = 1$ are the instanton action and the Stokes constant associated to the Airy functions,

		\item the minor $\widehat{W}^{(\pm,i)}_{n}$ is the Borel transform of the formal power series
		\begin{equation}\label{eqn:1inst:corr}
			W^{(\pm,i)}_{n}( \bm{x}; \hbar )
			\coloneqq
			(-1)^{n-1} \sum_{\sigma \in S_n^{\textup{cyc}}}
				K_{\mp,\mp}(x_i,x_{\sigma(i)};\hbar) \prod_{j \neq i} K_{\pm,\mp}(x_j,x_{\sigma(j)};\hbar) \,,
		\end{equation}
		where the kernels $K_{\pm,\pm}$ are given by
		\begin{equation}
			K_{\pm,\pm}(x,y;\hbar)
			\coloneqq
			\frac{\widetilde{\psi}_{\pm}'(x;\hbar) \widetilde{\psi}_{\pm}(y;\hbar) - \widetilde{\psi}_{\pm}(x;\hbar) \widetilde{\psi}_{\pm}'(y;\hbar)}{x-y} \,.
		\end{equation}
	\end{itemize}
	Moreover, the kernels are related by the parity relation $K_{-,-}(x,y;\hbar) = - K_{+,+}(x,y;-\hbar)$. Hence, the analogous relation holds for the correlators: $W^{(-,i)}_{n}( \bm{x}; \hbar ) = (-1)^n \, W^{(+,i)}_{n}( \bm{x}; -\hbar )$.
\end{proposition}

\begin{proof}
	The proof simply follows from the behaviour of the Borel transform of products, \cref{lemma:conv} and \cref{rem:princ:sheet}, together with the definition of the $n$-point function through the determinantal formula. The parity relations follow again from those satisfied by the Airy functions and their derivatives: $\psi_{-}(x;\hbar) = \psi_{+}(x;-\hbar)$ and $\psi_{-}'(x;\hbar) = - \psi_{+}'(x;-\hbar)$.
\end{proof}

\begin{remark}[{Minors in matrix form and $\Z/2\Z$-symmetry}]\label{rem:minors:Z2symmetry}
	In matrix formulation, the minors are expressed as
	\begin{equation}
		W^{(\pm,i)}_{n}( \bm{x}; \hbar )
		=
		(-1)^{n-1} \sum_{\sigma \in S_n^{\textup{cyc}}}
		\frac{\Tr{ \bigl( M_{\mp}(x_i;\hbar) \prod_{j=1}^{n-1} M(x_{\sigma^j(i)};\hbar) } \bigr) }{\prod_{j=1}^n ( x_{j} - x_{\sigma(j)} )} \,,
	\end{equation}
	where $M$ and $M_{\pm}$ are given by
	\begin{equation}
		M
		=
		\begin{pmatrix}
			\frac{1}{2}(\widetilde{\psi}_+' \widetilde{\psi}_- + \widetilde{\psi}_+ \widetilde{\psi}_-') &
			\widetilde{\psi}_+ \widetilde{\psi}_- \\
			\widetilde{\psi}_+' \widetilde{\psi}_-' &
			- \frac{1}{2}(\widetilde{\psi}_+' \widetilde{\psi}_- + \widetilde{\psi}_+ \widetilde{\psi}_-')
		\end{pmatrix} ,
		\qquad
		M_{\pm}
		=
		\begin{pmatrix}
			\widetilde{\psi}'_{\pm}\widetilde{\psi}_{\pm}
			&
			- ( \widetilde{\psi}_{\pm} )^2
			\\
			( \widetilde{\psi}'_{\pm} )^2
			&
			- \widetilde{\psi}'_{\pm}\widetilde{\psi}_{\pm}
		\end{pmatrix} .
	\end{equation}
	From these expression, it is easy to show that $x_j^{1/2} W^{(\pm,i)}_{n}$ is an even function of $x_j^{1/2}$ for all $j \ne i$. This follows from the fact that $x^{1/4} \psi_{\pm}$ is mapped to $x^{1/4} \psi_{\mp}$ under $x^{1/2} \mapsto -x^{1/2}$, and similarly for the derivatives. Thus $x^{1/2} M$ is even in $x^{1/2}$. With a similar argument, $x_i^{1/2} W^{(\pm,i)}_{n}$ is mapped to $x_i^{1/2} W^{(\mp,i)}_{n}$ under $x_i^{1/2} \mapsto -x_i^{1/2}$.
\end{remark}

Since $W^{(\pm,i)}_{n}$ are related by the parity relation, we simply consider $W_n^{(+,i)}$ in what follows. To simplify the notation, we drop the `$+$' symbol from the superscript. We can now apply the Borel transform method \labelcref{thm:large:order} to deduce the large genus asymptotics of the expansion coefficients of $W_n$ from the knowledge of its Borel singularity structure.

\begin{proposition}[{Large genus asymptotics of the Airy correlators}]\label{prop:large:g:corr:WK}
	For every $K$, the large genus asymptotics of the expansion coefficients of the $n$-point Airy correlators is given by
	\begin{equation}\label{eqn:large:g:corr:WK}
		W_{g,n}(\bm{x})
		=
		\frac{S}{\pi} \sum_{i=1}^n
			\frac{\Gamma(2g-2+n)}{A(x_i)^{2g-2+n}}
			\left(
				\sum_{k = 0}^K
				\frac{A(x_i)^{k}}{(2g-3+n)^{\underline{k}}} W^{(i)}_{k,n}(\bm{x})
				+
				\bigO\biggl( \frac{1}{g^{K+1}} \biggr)
			\right) \,,
	\end{equation}
	where:
	\begin{itemize}
		\item $A(x) = \frac{4}{3} x^{3/2}$ and $S = 1$ are the instanton action and the Stokes constant associated to the Airy functions,

		\item $W^{(i)}_{k,n}$ is the coefficient of $\hbar^k$ in the expansion of $W_n^{(i)}$.
	\end{itemize}
	Moreover, the leading term is explicitly given by
	\begin{equation}\label{eqn:large:g:corr:leading:WK}
		W_{g,n}(\bm{x})
		=
		\frac{(-1)^n}{4\pi} \,
		\frac{\Gamma(2g-2+n)}{\bigl( \frac{4}{3} \bigr)^{2g-2+n}}
		\Biggl(
			\sum_{\substack{ d_1,\dots,d_n \ge 0 \\ |d| = 3g-3+n}}
			\prod_{i=1}^n \frac{1}{x_i^{d_i+3/2}}
			+
			\bigO\bigl( g^{-1} \bigr)
		\Biggr) \,.
	\end{equation}
\end{proposition}

\begin{proof}
	As a consequence of the Borel transform method \cref{thm:large:order} (see also \cref{rem:Gevrey:2}) and the proposition above, we find
	\[
		W_{g,n}(\bm{x})
		=
		\frac{S}{2\pi} \sum_{i=1}^n
			\frac{\Gamma(2g-2+n)}{A(x_i)^{2g-2+n}}
			\left( \sum_{k = 0}^K
				\frac{A(x_i)^k}{(2g-3+n)^{\underline{k}}} \left(
					W^{(+,i)}_{k,n}(\bm{x})
					+
					(-1)^k
					W^{(-,i)}_{k,n}(\bm{x})
			\right)
			+
			\bigO\biggl( \frac{1}{g^{K+1}} \biggr)
			\right) .
	\]
	The parity relation implies $W^{(-,i)}_{k,n}(\bm{x}) = (-1)^{n+k} \, W^{(+,i)}_{k,n}(\bm{x})$, hence the first part of the proposition.

	As for the leading term, it is determined by the coefficient of $\hbar^0$ in the kernels. From the definition of the kernels and the asymptotic expansion of the Airy functions, we see that the first term in the $\hbar$-expansion is given by
	\begin{equation*}
		K_{\pm,-}(x,y;\hbar)
		=
		\pm \frac{1}{2 (x y)^{1/4}}
		\frac{1}{x^{1/2} \mp y^{1/2}}
		+ \bigO(\hbar) \,.
	\end{equation*}
	Inserting the above expressions in the formula for the $W_{n}^{(i)}$, we find
	\begin{equation*}
	\begin{split}
		W_{0,n}^{(i)}(\bm{x})
		&=
		\frac{(-1)^n}{2^n \, (x_1 \cdots x_n)^{1/2}}
		\sum_{\sigma \in S_n^{\textup{cyc}}}
			\frac{1}{x_i^{1/2} + x_{\sigma(i)}^{1/2}} \prod_{j \ne i} \frac{1}{x_j^{1/2} - x_{\sigma(j)}^{1/2}}
		=
		- \frac{1}{4}
		\frac{1}{(x_1 \cdots x_n)^{1/2}}
		\frac{x_i^{n/2 - 1}}{ \prod_{i \ne j} (x_i - x_j) } \,.
	\end{split}
	\end{equation*}
	In the last equation, we applied \cref{tech:lemma}. As a consequence, recalling that $S = 1$ and $A(x_i) = \frac{4}{3} x^{3/2}$, we find
	\[
		W_{g,n}(\bm{x})
		=
		-\frac{1}{4\pi} \, \frac{\Gamma(2g-2+n)}{\bigl( \frac{4}{3} \bigr)^{2g-2+n}}
		\Biggl(
			\frac{1}{(x_1 \cdots x_n)^{1/2}}
			\sum_{i=1}^n \frac{x_i^{-(3g-3+n)-1}}{ \prod_{i \ne j} (x_i - x_j) }
			+
			\bigO\bigl( g^{-1} \bigr)
		\Biggr) \,.
	\]
	Using the relation $\sum_{i=1}^n \frac{x_{i}^{-D-1}}{\prod_{i \ne j} (x_i - x_j)} = \frac{(-1)^{n-1}}{x_1 \cdots x_n} h_D(\bm{x}^{-1})$, with $h_D$ being the complete homogeneous symmetric polynomial of degree $D$, we find the thesis. This relation is a special case of \cref{lemma:poly}, but it can be easily proved by taking the generating series on both sides.
\end{proof}

As a consequence, we obtain the leading large genus asymptotics of $\psi$-class intersection numbers.

\begin{corollary}[{Leading large genus asymptotics for $\psi$-class intersection numbers}] \label{cor:leading:large:g:WK}
	For any given $n \ge 1$, uniformly in $d_1,\dots,d_n$ as $g \to \infty$:
	\begin{equation}
		\braket{\tau_{d_1} \cdots \tau_{d_n}} \prod_{i=1}^n (2d_i + 1)!! \\
		=
		\frac{2^n}{4\pi}  \frac{\Gamma(2g-2+n)}{( \frac{2}{3} )^{2g-2+n}}
		\Bigl(
			1
			+
			\bigO\bigl( g^{-1} \bigr)
		\Bigr) \,.
	\end{equation}
\end{corollary}

The most challenging aspect of proving the corollary lies in the uniformity statement. We postpone the proof to the more general \cref{thm:large:g:WK}, which accounts for all subleading corrections. In fact, the next subsection is devoted to the study of the subleading terms $W_{k,n}^{(i)}$ for arbitrary $k$.

\subsubsection{Subleading contributions}
From \cref{eqn:large:g:corr:leading:WK}, it is clear that both $W_{g,n}$ and its leading large genus behaviour are polynomials in $x_1^{-1},\dots,x_n^{-1}$ of degree $3g-3+n$ (up to a prefactor of $(x_1 \cdots x_n)^{-3/2}$). It is natural to ask whether the same property holds for the subleading corrections. More precisely, from the analysis of the kernel we see that the leading term $W_{0,n}^{(i)}$ has, a priori, simple poles along $\sqrt{x_i} = - \sqrt{x_j}$ for all $j \ne i$ and along $\sqrt{x_k} = \sqrt{x_l}$ for all $k \ne l$. However, the only poles appearing are those at $\sqrt{x_i} = \pm \sqrt{x_j}$. Furthermore, after summing over all contributions, the poles remarkably disappear, yielding a polynomial of the expected form. The next lemmas show that a similar structure holds for all subleading terms.

\begin{lemma}[Pole structure of the minors] \label{lemma:poles:minors}
	The minor $W_n^{(i)}$ is regular along $\sqrt{x_k} = \sqrt{x_l}$ for all $k \ne l$ with $k,l \ne i$. Moreover, it has simple poles at $\sqrt{x_i} = \pm \sqrt{x_j}$ for all $j \neq i$, with residue
	\begin{equation}
		\Res_{\sqrt{x_j} \to \pm \sqrt{x_i}} W_n^{(i)}(x_1,\dots,x_n;\hbar)
		=
		- \frac{1}{2 \sqrt{x_i}} \, W_{n-1}^{(i)}(x_1,\dots,\widehat{\,x_j\,},\dots,x_{n};\hbar) \, .
	\end{equation}
	On the right-hand side, the superscript $(i)$ means that the ``special variable'' is $x_i$.
\end{lemma}

\begin{proof}
	Throughout the proof, we denote $z_m = \sqrt{x_m}$. Consider the poles at $z_k = z_l$ for all $k \ne l$ with $k,l \ne i$. The only terms in the definition of $W_n^{(i)}$ contributing to the residue are those permutations $\sigma \in S_n^{\textup{cyc}}$ such that $\sigma(k) = l$ or $\sigma(l) = k$. If both $k$ and $l$ are different from $i$, we find
	\begin{multline*}
		\Res_{z_k \to z_l} W_n^{(i)}(\bm{x};\hbar)
		=
		(-1)^{n-1} \Res_{z_k \to z_l} 
		\Biggl(
			\sum_{\substack{\sigma \in S_n^{\textup{cyc}} \\ \sigma(k) = l}}
				K_{-,-}(x_i,x_{\sigma(i)};\hbar)
				K_{+,-}(x_k,x_l;\hbar)
				\prod_{j \ne i,k} K_{+,-}(x_j,x_{\sigma(j)};\hbar) \\
			+
			\sum_{\substack{\sigma \in S_n^{\textup{cyc}} \\ \sigma(l) = k}}
				K_{-,-}(x_i,x_{\sigma(i)};\hbar)
				K_{+,-}(x_l,x_k;\hbar)
				\prod_{j \ne i,l} K_{+,-}(x_j,x_{\sigma(j)};\hbar)
		\Biggr) .
	\end{multline*}
	The pole at $z_k \to z_l$ is given by $K_{+,-}(x_k,x_l;\hbar)$ in the first sum and by $K_{+,-}(x_l,x_k;\hbar)$ in the second one. Moreover, the residues are
	\[
		\Res_{z_k \to z_l} K_{+,-}(x_k,x_l;\hbar) = \frac{1}{2z_l} \,,
		\qquad\qquad
		\Res_{z_k \to z_l} K_{+,-}(x_l,x_k;\hbar) = - \frac{1}{2z_l} \,,
	\]
	since they coincide (up to a sign) with the Wronskian of the differential system. After taking the residue, the two terms cancel out. This proves the first part of the statement.

	For the second part, assume without loss of generality that $j = n$ (and $i < n$). Consider the pole at $z_n = z_i$. With the same computation as before, we find
	\begin{multline*}
		\Res_{z_n \to z_i} W_n^{(i)}(\bm{x};\hbar)
		=
		(-1)^{n-1} \Res_{z_n \to z_i} 
		\Biggl(
			\sum_{\substack{\sigma \in S_n^{\textup{cyc}} \\ \sigma(n) = i}}
				K_{-,-}(x_i,x_{\sigma(i)};\hbar)
				K_{+,-}(x_n,x_i;\hbar)
				\prod_{j \ne i,n} K_{+,-}(x_j,x_{\sigma(j)};\hbar) \\
			+
			\sum_{\substack{\sigma \in S_n^{\textup{cyc}} \\ \sigma(i) = n}}
				K_{-,-}(x_i,x_n;\hbar)
				\prod_{j \ne i} K_{+,-}(x_j,x_{\sigma(j)};\hbar)
		\Biggr) .
	\end{multline*}
	Now the residues from the two sums give different contributions:
	\[
		\Res_{z_n \to z_i} K_{+,-}(x_n,x_i;\hbar) = \frac{1}{2z_i} \,,
		\qquad\qquad
		\Res_{z_n \to z_i} K_{-,-}(x_i,x_n;\hbar) = 0 \,.
	\]
	The first residue is again the Wronskian, while the second one is simply a cancellation of the numerator of $K_{-,-}$ in the limit. Thus, the second sum gives no contribution, while the first sum can be re-written as a sum over $S_{n-1}^{\textup{cyc}}$ recovering $- W_{n-1}^{(i)}$.

	Finally, consider the pole at $z_n = - z_i$. With the same computation as before, but with the residue contributions
	\[
		\Res_{z_n \to -z_i} K_{+,-}(x_n,x_i;\hbar) = 0 \,,
		\qquad\qquad
		\Res_{z_n \to -z_i} K_{-,-}(x_i,x_n;\hbar) = \frac{\iu}{2z_i} \,.
	\]
	and the property $K_{+,-}(x_n,x_{\sigma(n)};\hbar) \big|_{z_n \mapsto -z_i} = - \iu \, K_{-,-}(x_i,x_{\sigma(n)};\hbar)$, we conclude the proof.
\end{proof}

We now turn our attention to the $k$-th subleading term in the asymptotic expansion of the $n$-point Airy correlators. For simplicity, denote
\begin{equation}
	U_{k,g,n}(\bm{x})
	\coloneqq
	(-1)^n \, 2^{k+2} \, (x_1 \cdots x_n)^{3/2}
	\sum_{i=1}^n
		x_i^{-\frac{3}{2}(2g-2+n-k)} \, W_{k,n}^{(i)}(\bm{x}) \,.
\end{equation}
With this notation, \eqref{eqn:large:g:corr:WK} reads
\begin{equation}\label{eqn:large:g:corr:Y}
	W_{g,n}(\bm{x})
		=
		\frac{(-1)^n}{4\pi}
		\frac{\Gamma(2g-2+n)}{(\frac{4}{3})^{2g-2+n}}   
			\left(
			\sum_{k = 0}^K
				\frac{(\frac{2}{3})^{k}}{(2g-3+n)^{\underline{k}}}
				\frac{U_{k,g,n}(\bm{x})}{(x_1 \cdots x_n)^{3/2}}
			+
			\bigO\biggl( \frac{1}{g^{K+1}} \biggr)
			\right) \,.
\end{equation}
\Cref{eqn:large:g:corr:leading:WK} is equivalent to $U_{0,g,n}(\bm{x}) = h_{3g-3+n}(\bm{x}^{-1})$, i.e.~the leading term is the complete homogeneous symmetric polynomial of degree $3g-3+n$ in the variables $x_i^{-1}$. The next lemma extends the result to subleading corrections, showing that the genus dependence is completely captured by a complete homogeneous polynomial of degree $3g-3+n$ minus a small defect which depends on $k$. We refer to \cref{app:symmetric:fncts} and \cite{Mac98} for more details and notations on symmetric functions.

\begin{lemma}[Polynomial properties of subleading corrections]\label{lemma:poly:subleading}
	For every $k \ge 0$, $n \ge 1$ and $g$ large enough, $U_{k,g,n}$ is a homogeneous symmetric polynomial in $x_i^{-1}$ of degree $3g-3+n$ of the form
	\begin{equation}\label{eqn:subleading:poly}
		U_{k,g,n}(\bm{x})
		=
		\sum_{d=0}^{\min\{ \floor{\frac{3k+n-1}{2}},\, 3k-1 \}}
			P_{k,n}^{(d)}(\bm{x}^{-1/2}) \, h_{3g-3+n-d}(\bm{x}^{-1}) \,,
	\end{equation}
	where $P_{k,n}^{(d)}$ is a homogeneous symmetric polynomial of degree $2d$, both even and of degree $\le 3k$ in each individual variable. The sequence of polynomials $(P_{k,n})_{n \ge 1}$, with\footnote{
		Experimentally, we can see $P_{k,n}$ has degree $2 \min\{ \floor{\frac{3k+n-1}{2}}, 2k \}$, so that the sum in \cref{eqn:subleading:poly} terminates earlier than expected. This would slightly improve the algorithm for computing the subleading corrections, since the improved bound would lower the number of quantities to compute. Moreover, we notice that the expansion of $P_{k,n}$ in the basis of elementary symmetric does not contain all possible partitions of $d$. This underlying structure of intersection numbers was noticed in \cite{EL,EM} as well.
	}
	\begin{equation}
		P_{k,n} \coloneqq \sum_{d=0}^{\min\{ \floor{\frac{3k+n-1}{2}},\, 3k-1 \}} P_{k,n}^{(d)} \,,
	\end{equation}
	satisfies the following two specialisation properties:
	\begin{equation}\label{eqn:spec}
		P_{k,n+1}(u_1,\dots, u_n, u_{n+1}) \big|_{u_{n+1} = \pm 1}  = P_{k,n}(u_1,\dots,u_n) \,.
	\end{equation}
	In turn, the degree condition and the specialisation properties uniquely determine $P_{k,n}$ for all $n$ by an explicit algorithm from the first terms of the sequence, namely $P_{k,1},\dots, P_{k,3k-1}$. Moreover, the coefficient of $m_{2\nu}$ in the expansion of $P_{k,n}$ in the monomial basis is a polynomial of $n$ of degree $\le 3k-1 - |\nu|$.
\end{lemma}

\begin{table}[t]
\centering
{\renewcommand{\arraystretch}{1.15}
\begin{tabularx}{\textwidth} { 
	c
	| >{\raggedright\arraybackslash}X 
	| >{\raggedright\arraybackslash}X }
	\toprule
	$k$ & $P_{k,n}$ & $\alpha_k(n,\bm{p})$ \\
	\midrule
	$0$ & $1$ & $1$
	\\
	\midrule
	$1$ & $
		- \tfrac{17 - 15n + 3n^2}{12} m_{\varnothing}
		- \tfrac{3 - n}{2} \, m_{(1)}
		- \tfrac{1}{2} \, m_{(1^2)}$
	& $
		- \tfrac{17 - 15n + 3n^2}{12}
		- \tfrac{(3 - n)(n - p_0)}{2}
		- \tfrac{(n - p_0)^{\underline{2}}}{4} $
	\\
	\midrule
	$2$
	& $
		\tfrac{1225 - 1632 n + 741 n^2 - 138 n^3 + 9 n^4}{288} m_{\varnothing}
		+
		\tfrac{105 - 98 n + 30 n^2 - 3 n^3}{24} m_{(1)}
		+
		\tfrac{3(10 - 7 n + n^2)}{8} m_{(2)}
		+
		\tfrac{59 - 51 n + 9 n^2}{24} m_{(1^2)}
		+
		\tfrac{5}{8} m_{(3)}
		+
		\tfrac{3(4 - n)}{4} m_{(2,1)}
		+
		\tfrac{7 - 3 n}{4} m_{(1^3)}
		+
		\tfrac{3}{4} m_{(2,1^2)}
		+
		\tfrac{3}{4} m_{(1^4)}$
	&  $
		\tfrac{1225 - 1632 n + 741 n^2 - 138 n^3 + 9 n^4}{288}
		+
		\tfrac{(105 - 98 n + 30 n^2 - 3 n^3)(n - p_0)}{24}
		+
		\tfrac{3(10 - 7 n + n^2)(n - p_0 - p_1)}{8}
		+
		\tfrac{(59 - 51 n + 9 n^2)(n - p_0)^{\underline{2}}}{48}
		+
		\tfrac{5(n - p_0 - p_1 - p_2)}{8}
		+
		\tfrac{3(4 - n)(n - p_0 - 1)(n - p_0 - p_1)}{4}
		+
		\tfrac{(7 - 3 n)(n - p_0)^{\underline{3}}}{24}
		+
		\tfrac{3(n - p_0 - 1)^{\underline{2}}(n - p_0 - p_1)}{48}
		+
		\tfrac{3(n - p_0)^{\underline{4}}}{96}$
	\\
	\bottomrule
\end{tabularx}
}
\caption{
	The polynomials $P_{k,n}(u_1,\dots,u_n)$ and the coefficients $\alpha_k$ for $k = 0,1,2$. Here $m_{\lambda}$ denotes the monomial symmetric polynomial in the variables $u_1^2,\dots,u_n^2$.
}
\label{table:Pkn:ck}
\end{table}

As consequence of the above lemma, we get a polynomiality statement for the coefficients of $U_{k,g,n}$ in terms of $n$ and the multiplicities of the exponents.

\begin{corollary}[Polynomial properties of subleading corrections: the coefficients]\label{cor:poly:subleading}
	For every $k \ge 0$, $n \ge 1$, $g$ large enough, and $d_1, \dots, d_n \ge 0$ satisfying $d_1 + \dots + d_n = 3g-3+n$, the coefficients of $U_{k,g,n}$ are of the form
	\begin{equation}\label{eqn:subleading:coeffs}
		\bigl[ x_1^{-d_1} \cdots x_n^{-d_n} \bigr] \,
		U_{k,g,n}(\bm{x})
		=
		\alpha_k\bigl( n, p_0, \dots, p_{\floor{\frac{3k}{2}}-1} \bigr) \,,
	\end{equation}
	where $\alpha_k$ is a polynomial in $n$ and $p_m \coloneqq \#\set{ d_i = m}$ for $m = 1, \dots, \floor{\frac{3k}{2}}-1$ that can be effectively computed. Moreover, the degree is $\le 3k-1$, under the assignment $\deg{n} = 1$ and $\deg{p_m} = m+1$.
\end{corollary}

The first few polynomials $P_{k,n}$ and coefficients $\alpha_k$ are given in \cref{table:Pkn:ck}. The intuitive explanation of the above statement can be given as follows. \Cref{prop:large:g:corr:WK} shows that the $k$-th subleading correction in the large genus asymptotics of the correlators is proportional to
\begin{equation}
	\sum_{i=1}^n x_i^{-\frac{3}{2}(2g-2+n-k)} \, W_{k,n}^{(i)} \,.
\end{equation}
In the large genus limit, most of the degree is then coming from the factors $x_i^{-\frac{3}{2}(2g-2+n)}$, which in turn give rise to complete homogeneous polynomials of large degree times some polynomials of small degree. Thus, the extraction of coefficients depends only on the number of variables with a small degree, as all monomials appear in the complete homogeneous polynomial with coefficient $1$. We also emphasise that \cref{lemma:poly:subleading} gives a concrete algorithm to compute the subleading corrections. Indeed the proof shows that, for fixed $k$, the computation of $P_{k,n}$ for $n \le 3k-1$ gives a closed, explicit expression for $P_{k,n}$ for arbitrary $n$. In turn, this gives a closed, explicit expression for the coefficients $\alpha_k$, see \cref{fig:algorithm}.

\begin{figure}
\centering
\begin{tikzpicture}[yscale=.9]

	\node[punkt,style={text width=6.5cm}] (A1) at (0,0) {\small Compute $P_{k,1},\dots,P_{k,3k-1}$ \\ via determinantal formulae};
	\node[punkt,style={text width=6.5cm}] (A2) at (0,-2) {\small Compute $(P_{k,n})_{n \ge 1}$ in the elementary basis} edge[pil,<-] (A1);
	\node[punkt,style={text width=6.5cm}] (A3) at (0,-4) {\small Compute $(P_{k,n})_{n \ge 1}$ \\ in the monomial basis} edge[pil,<-] (A2);
	\node[punkt,style={text width=6.5cm}] (A4) at (0,-6) {\small Compute $\alpha_k$} edge[pil,<-] (A3);

	\node (a23) at (0,-3) {};
	\node (KN) at (3,-3) {\small Kostka numbers} edge[pil,->] (a23);

	\node (a34) at (0,-5) {};
	\node (KN) at (3,-5) {\small $M_{n,\mu,\nu}$ via \cref{prop:polynom:monomial:coeff}} edge[pil,->] (a34);

\end{tikzpicture}
\caption{
	The flowchart representing the algorithm for the computation of the subleading contributions $\alpha_k$.
}
\label{fig:algorithm}
\end{figure}
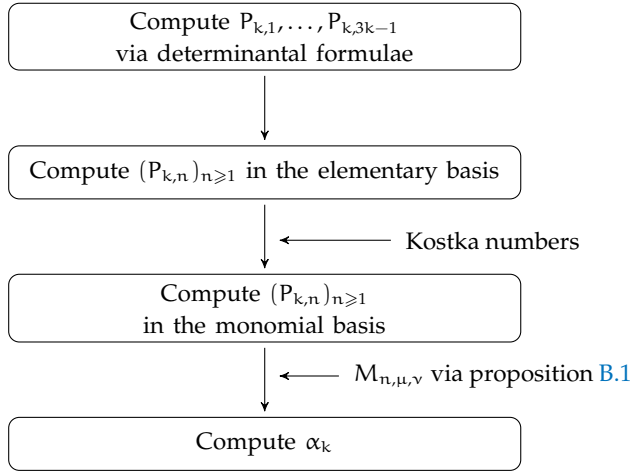

\begin{proof}[{Proof of \cref{lemma:poly:subleading}}]
	We first claim that the functions $W_{k,n}^{(i)} = [\hbar^k]W_{n}^{(i)}$ are of the form
	\begin{equation*}
		W_{k,n}^{(i)}(\bm{x})
		=
		-
		\frac{2^{-k-2}}{(x_1 \cdots x_n)^{1/2}}
		\frac{x_i^{n/2 - 1 - 3k/2}}{\prod_{j \ne i} (x_i - x_j)} \,
		P_{k,n-1} \biggl(
			\Bigl( \frac{x_i}{x_1} \Bigr)^{1/2}, \dots,
			\widehat{ \Bigl( \frac{x_i}{x_i} \Bigr)^{1/2} }, \dots,
			\Bigl( \frac{x_i}{x_n} \Bigr)^{1/2}
		\biggr) \,,
	\end{equation*}
	where $P_{k,n-1}$ is a symmetric polynomial in $n-1$ variables of total degree $\le \min\set{ 3k + n - 2, 2(3k - 1) }$, both even and of degree $\le 3k$ in each individual variable, and satisfying the two specialisation properties $P_{k,n+1}( \,\cdot\,, \pm 1) = P_{k,n}(\,\cdot\,)$. Indeed, from the homogeneity property and the parity relation satisfied by the Airy functions, i.e.~$\psi_{\pm}(x;\hbar) = x^{-1/4} \, \psi_{\pm}(1;\frac{\hbar}{2A(x)})$, $\psi_{+}(x;\hbar) = \psi_{-}(x;-\hbar)$, and similarly for the derivatives, we deduce that
	\[
		K_{\pm,-}(x,y;\hbar)
		=
		\pm \frac{1}{(x y)^{1/4}} \frac{1}{x^{1/2} \mp y^{1/2}}
		\sum_{m \ge 0} A_m\bigl( \mp x^{-1/2},y^{-1/2} \bigr) \hbar^m \,,
	\]
	where $A_m$ is a homogeneous symmetric polynomial of degree $3m$. From \cref{lemma:poles:minors} and a degree consideration, we conclude that
	\[
		(x_1 \cdots x_n)^{1/2} \, x_i^{3k/2 - n/2 + 1}
		\Biggl( \prod_{j \ne i} (x_i - x_j) \Biggr)
		W_{k,n}^{(i)}(\bm{x})
	\]
	is a symmetric polynomial in $t_j = \sqrt{x_i}/\sqrt{x_j}$ for $j \neq i$ of degree $\le 3k + n - 2$. In the $t_j$ variables the above polynomial is given by
	\[
		(-1)^n \prod_{j \neq i} (1 - t_j^2)
		\sum_{m \neq i}
		\sum_{\substack{k_1 + \dots + k_{n} = k \\ \sigma \in S_n^{\textup{cyc}}, \, \sigma(m) = i}}
			\frac{A_{k_i}(1, t_{\sigma(i)})}{1 + t_{\sigma(i)}}
			\frac{A_{k_m}(-t_{m},1)}{t_{m} - 1}
			\prod_{\substack{j \ne i, \, m}}
				\frac{A_{k_j}(-t_j, t_{\sigma(j)})}{t_j - t_{\sigma(j)}} \,.
	\]
	The polynomiality and the specialisation properties, which are not obvious from the above expression, are guaranteed by \cref{lemma:poles:minors}. It is clear that the degree in each individual variable is $\le 3k$. Moreover, from \cref{rem:minors:Z2symmetry} we deduce that the above expression is even in $t_j$. Thus, after normalising by $-2^{-k-2}$, we have that
	\[
		W_{k,n}^{(i)}(\bm{x})
		=
		-
		\frac{2^{-k-2}}{(x_1 \cdots x_n)^{1/2}}
		\frac{x_i^{n/2 - 1 - 3k/2}}{\prod_{j \ne i} (x_i - x_j)} \,
		P_{k,n-1} \biggl(
			\Bigl( \frac{x_i}{x_1} \Bigr)^{1/2}, \dots,
			\widehat{ \Bigl( \frac{x_i}{x_i} \Bigr)^{1/2} }, \dots,
			\Bigl( \frac{x_i}{x_n} \Bigr)^{1/2}
		\biggr) \,,
	\]
	where $P_{k,n-1}$ is a symmetric polynomial of total degree $\le 3k + n - 2$, both even and of degree $\le 3k$ in each individual variable. Our next goal is to prove that there exists a value $N_k$ such that
	\begin{equation*}
		\deg(P_{k,n+1}) = \deg(P_{k,n})
		\qquad\quad
		\text{ for all } n \geq N_k \,.
	\end{equation*}
	To this end, consider for $P_{k,n}(u_1, \dots, u_n)$ the two changes of variables (which naturally follows from the specialisation properties)
	\begin{equation*}
		\hat{u}_j \coloneqq u_j - 1 \,,
		\quad
		\hat{e}_s \coloneqq e_s(\hat{u}_1, \dots, \hat{u}_n) \,,
		\qquad\qquad
		\check{u}_j \coloneqq u_j + 1 \,,
		\quad
		\check{e}_s \coloneqq e_s(\check{u}_1, \dots, \check{u}_n) \,.
	\end{equation*}
	We also refer to $\hat{P}_{k,n}$ (resp. $\check{P}_{k,n}$) as to $P_{k,n}$ expanded in the basis of elementary symmetric polynomials $\hat{e}_s$ (resp. $\check{e}_s$). From the specialisation properties \labelcref{eqn:spec} and the fact that setting any of the $\hat{u}_j$ (resp. $\check{u}_j$) to zero is the same as setting $\hat{e}_{n+1}$ (resp. $\check{e}_{n+1}$) to zero, we obtain the two decompositions
	\begin{equation}\label{eqn:decomp}
		\hat{P}_{k,n+1} = \hat{P}_{k,n} + \hat{e}_{n+1}\hat{Q}_{k,n} \,,
		\qquad \qquad
		\check{P}_{k,n+1} = \check{P}_{k,n} + \check{e}_{n+1}\check{Q}_{k,n} \,,
	\end{equation}
	for some $\hat{Q}_{k,n}$ and $\check{Q}_{k,n}$, polynomials in $\hat{e}_s$ and $\check{e}_s$ respectively. With the natural degree assignment $\deg{\hat{e}_s} = \deg{\check{e}_s} = s$, we find $\deg(\hat{Q}_{k,n}) = \deg(\check{Q}_{k,n}) \leq 3k - 1$. Define the linear unitriangular ring isomorphism $\ms{T}_n$ performing the change of basis from $\hat{e}$ to $\check{e}$:
	\begin{align*}
		\ms{T}_n \colon \mb{Q}[\hat{e}] \longrightarrow \mb{Q}[\check{e}] \,,
		\qquad
		\hat{e}_s \longmapsto \sum_{m=0}^s \binom{n-s+m}{m} (-2)^m \check{e}_{s-m} \,.
	\end{align*}
	The factor $-2$ arises from the difference $\hat{u}_j - \check{u}_j$. Applying $\ms{T}_{n+1}$ to the first decomposition of \cref{eqn:decomp} we obtain
	\begin{equation*}
		\check{P}_{k,n+1} - \ms{T}_{n+1}(\hat{P}_{k,n})
		=
		\ms{T}_{n+1}(\hat{e}_{n+1} \hat{Q}_{k,n})\,.
	\end{equation*}
	Here we used the fact that $\ms{T}_{n+1}(\hat{P}_{k,n+1}) = \check{P}_{k,n+1}$. However, $\ms{T}_{n+1}(\hat{P}_{k,n})$ does not necessarily coincide with $\check{P}_{k,n}$, as the change of basis depends on the number of variables. Adding and subtracting $\check{P}_{k,n}$ to the above equation and applying the second decomposition of \cref{eqn:decomp} yields
	\begin{equation*}
		\check{P}_{k,n} - \ms{T}_{n+1}(\hat{P}_{k,n}) - \ms{T}_{n+1}(\hat{e}_{n+1}) \ms{T}_{n+1}(\hat{Q}_{k,n})
		=
		\check{e}_{n+1}\check{Q}_{k,n} \,.
	\end{equation*}
	Notice that, as the right-hand side is divisible by $\check{e}_{n+1}$, so does the left-hand side. The only factors that (might) contain $\check{e}_{n+1}$ are $\ms{T}_{n+1}(\hat{e}_{n+1})$ and $\ms{T}_{n+1}(\hat{Q}_{k,n})$. However, for $n \geq 3k - 1$ there are no such terms in $\ms{T}_{n+1}(\hat{Q}_{k,n})$ for degree reasons. Therefore, after removing the leading term from $\ms{T}_{n+1}(\hat{e}_{n+1})$, we find
	\begin{equation*}
		\check{P}_{k,n} - \ms{T}_{n+1}(\hat{P}_{k,n})
		=
		\left(\ms{T}_{n+1}(\hat{e}_{n+1}) - \check{e}_{n+1}\right) \ms{T}_{n+1}(\hat{Q}_{k,n}) \,.
	\end{equation*}
	By the unitriagularity of the operator $\ms{T}_n$, the degree of $\check{P}_{k,n} - \ms{T}_{n+1}(\hat{P}_{k,n})$ is lower by one compared to the degree of $\check{P}_{k,n}$, as the leading terms cancel each other out. Explicitly, we obtain that $\deg(\check{P}_{k,n} - \ms{T}_{n+1}(\hat{P}_{k,n})) \leq 3k + n - 2,$ so that $\deg(\hat{Q}_{k,n}) \leq 3k - 2$, which in turns forces 
	\begin{equation*}
		\deg(P_{k,n+1}) = \deg(P_{k,n})
		\qquad
		\text{for } n \geq 3k-1 \,.
	\end{equation*}
	In other words, we proved that for $n \ge 3k-1$ the degree of $P_{k,n}$ is independent on $n$, which in turn implies that $\deg{( P_{k,n} )} \le \min\{ 3k + n - 1, 2(3k - 1) \}$. Being even in all variables, we find that $\deg{( P_{k,n} )} \le 2 \min\{ \floor{\frac{3k + n - 1}{2}}, 3k - 1 \}$\footnote{
		We also remark that there is also a shortcut proof that applies for the Airy and Bessel cases, though it does not extend to the $r$-spin setting. Consider the homogeneous component of maximal degree in $P_{k,n}(u_1,\dots,u_n)$, denoted by $P_{k,n}^{\textup{max}}(u_1,\dots,u_n)$. Suppose that $u_n$ divides $P_{k,n}^{\textup{max}}$. By symmetry and parity, $u_1^2 \cdots u_n^2$ must divide $P_{k,n}^{\textup{max}}$. Hence, $2n \le \deg{( P_{k,n}^{\textup{max}} )} = \deg{( P_{k,n} )} \le 3k+n-1$. This gives a contradiction as soon as $n \ge 3k$. Thus, $u_n$ cannot divide $P_{k,n}^{\textup{max}}$ for $n \ge 3k$. In this case, we obtain that $\deg{( P_{k,n}^{\textup{max}} )} = \deg{( P_{k,n}^{\textup{max}}\big|_{u_n = 1} )} = \deg{( P_{k,n-1}^{\textup{max}} )}$ from the recursive relation $P_{k,n}( \,\cdot\,, 1) = P_{k,n-1}(\,\cdot\,)$, from which we conclude that
		\(
			\deg(P_{k,n+1})
			=
			\deg(P_{k,n})
			\text{ for any } n \ge 3k-1 \,.
		\)
	}.

	From the above proof it follows that, for any fixed $k$, the first few values of $n$ determine the whole sequence $(P_{k,n})_{n \ge 1}$. Moreover, the coefficients of $P_{k,n}$ in the monomial basis are polynomials of $n$. Indeed, as the degree of $P_{k,n}$ is constant in $n$ for $n$ large enough, we deduce from \cref{eqn:decomp} that eventually $\hat{P}_{k,n+1} = \hat{P}_{k,n}$. The change of basis from $\hat{e}_s = e_s(u_1-1,\dots,u_n-1)$ to $e_s = e_s(u_1,\dots,u_n)$, which is analogous to the transformation $\ms{T}_n$, depends polynomially in $n$. Furthermore, the change of basis from elementary to complete homogeneous polynomials is independent of $n$ (it involves Kostka numbers, see \cite{Mac98}). Hence, the coefficients of $P_{k,n}$ in the monomial basis are polynomials of $n$. As for the number of polynomials $P_{k,n}$ that determine the whole sequence, one can show that $P_{k,1}, \dots, P_{k,3k - 1}$ are sufficient by degree considerations and the parity property. Moreover, as polynomials of $n$, it can be shown that the degree of $m_{2\nu}$ is bounded by $3k-1-|\nu|$ by the explicit change of basis from $\hat{e}$ to $e$.

	We now proceed with the proof of \cref{eqn:subleading:poly} for $U_{k,g,n}$. From the definition of $U_{k,g,n}$ and the definition of $P_{k,n-1}$, we find
	\begin{equation*}
		U_{k,g,n}(\bm{x})
		=
		(-1)^{n-1} \, x_1 \cdots x_n
		\sum_{i=1}^n
			\frac{ x_i^{-(3g-3+n)-1} }{ \prod_{j \ne i} (x_i - x_j) } \,
			P_{k,n-1} \biggl(
				\Bigl( \frac{x_i}{x_1} \Bigr)^{1/2}, \dots,
				\widehat{ \Bigl( \frac{x_i}{x_i} \Bigr)^{1/2} }, \dots,
				\Bigl( \frac{x_i}{x_n} \Bigr)^{1/2}
			\biggr) \,.
	\end{equation*}
	We claim that, for $g$ large enough,
	\begin{multline*}
		\sum_{i=1}^n
			\frac{ x_i^{-(3g-3+n)-1} }{ \prod_{j \ne i} (x_i - x_j) } \,
			P_{k,n-1} \biggl(
				\Bigl( \frac{x_i}{x_1} \Bigr)^{1/2}, \dots,
				\widehat{ \Bigl( \frac{x_i}{x_i} \Bigr)^{1/2} }, \dots,
				\Bigl( \frac{x_i}{x_n} \Bigr)^{1/2}
			\biggr) \\
		=
		\frac{(-1)^{n-1}}{x_1 \cdots x_n}
		\sum_{d=0}^{\min\{ \floor{\frac{3k+n-1}{2}},\, 3k-1 \}}
			P_{k,n}^{(d)} ( \bm{x}^{-1/2} ) \,
			h_{3g-3+n-d} ( \bm{x}^{-1} )\,,
	\end{multline*}
	where $P_{k,n}^{(d)}$ is the homogeneous component of degree $2d$ in $P_{k,n}$ (notice the shift in $n$). This would complete the proof of the lemma. In order to prove the claimed formula, consider the function
	\[
		f(q)
		\coloneqq
		\frac{P_{k,n}(\frac{q}{\sqrt{x_1}}, \dots, \frac{q}{\sqrt{x_n}})}{\prod_{i=1}^n (q^2 - x_i)} \,.
	\]
	We omit its dependence on $x_1^{1/2},\dots,x_n^{1/2}$ for simplicity. As a function of $q$, it has simple poles at $q = \pm \sqrt{x_i}$ with residues
	\begin{equation*}
		\Res_{q = \pm\sqrt{x_i}} f(q)
		=
		\frac{P_{k,n}\bigl(
			\pm \frac{\sqrt{x_i}}{\sqrt{x_1}}, \dots,
			\pm \frac{\sqrt{x_i}}{\sqrt{x_i}}, \dots,
			\pm \frac{\sqrt{x_i}}{\sqrt{x_n}}
		\bigr)}{\pm 2\sqrt{x_i} \prod_{j \ne i} (x_i - x_j)}
		=
		\pm \frac{P_{k,n-1}\bigl(
			\frac{\sqrt{x_i}}{\sqrt{x_1}}, \dots,
			\widehat{ \frac{\sqrt{x_i}}{\sqrt{x_i}} }, \dots,
			\frac{\sqrt{x_i}}{\sqrt{x_n}}
		\bigr)}{2\sqrt{x_i} \prod_{j \ne i} (x_i - x_j)} \,.
	\end{equation*}
	Here we used the fact that $P_{k,n}$ is even and it satisfies $P_{k,n}( \,\cdot\,, \pm 1) = P_{k,n-1}(\,\cdot\,)$. As a consequence, we obtain
	\[
	\begin{split}
		\bigl[ q^{6g-6+2n} \bigr] f(q)
		& =
		\bigl[ q^{6g-6+2n} \bigr] \sum_{i=1}^n \left(
			\frac{P_{k,n-1}\bigl(
				\frac{\sqrt{x_i}}{\sqrt{x_1}}, \dots,
				\widehat{ \frac{\sqrt{x_i}}{\sqrt{x_i}} }, \dots,
				\frac{\sqrt{x_i}}{\sqrt{x_n}}
			\bigr)}{(q - \sqrt{x_i}) \, 2\sqrt{x_i} \prod_{j \ne i} (x_i - x_j)}
			-
			\frac{P_{k,n-1}\bigl(
				\frac{\sqrt{x_i}}{\sqrt{x_1}}, \dots,
				\widehat{ \frac{\sqrt{x_i}}{\sqrt{x_i}} }, \dots,
				\frac{\sqrt{x_i}}{\sqrt{x_n}}
			\bigr)}{(q + \sqrt{x_i}) \, 2\sqrt{x_i} \prod_{j \ne i} (x_i - x_j)}
		\right) \\
		& =
		\bigl[ q^{6g-6+2n} \bigr] \sum_{i=1}^n
			\frac{1}{q^2 - x_i}
			\frac{P_{k,n-1}\bigl(
				\frac{\sqrt{x_i}}{\sqrt{x_1}}, \dots,
				\widehat{ \frac{\sqrt{x_i}}{\sqrt{x_i}} }, \dots,
				\frac{\sqrt{x_i}}{\sqrt{x_n}}
			\bigr)}{\prod_{j \ne i} (x_i - x_j)} \\
		& =
		- \sum_{i=1}^n
			\frac{x_i^{-(3g-3+n)-1}}{\prod_{j \ne i} (x_i - x_j)} \,
			P_{k,n-1} \biggl(
				\Bigl( \frac{x_i}{x_1} \Bigr)^{1/2}, \dots,
				\widehat{ \Bigl( \frac{x_i}{x_i} \Bigr)^{1/2} }, \dots,
				\Bigl( \frac{x_i}{x_n} \Bigr)^{1/2}
			\biggr) \, .
	\end{split}
	\]
	On the other hand, extracting the coefficient of $q^{6g-6+2n}$ from the definition of $f$ gives
	\[
		\bigl[ q^{6g-6+2n} \bigr] f(q)
		=
		\frac{(-1)^{n}}{x_1 \cdots x_n}
		\sum_{d=0}^{\min\{ \floor{\frac{3k+n-1}{2}},\, 3k-1 \}}
			P_{k,n}^{(d)} ( \bm{x}^{-1/2} ) \,
			h_{3g-3+n-d} ( \bm{x}^{-1} ) \,.
	\]
	This completes the proof.
\end{proof}
	
\begin{proof}[{Proof of \cref{cor:poly:subleading}}]
	From \cref{lemma:poly:subleading}, we see that $P_{k,n}$ in the monomial basis reads
	\[
		P_{k,n}(u_1,\dots,u_n)
		=
		\sum_{\substack{
			|\nu| \le \min\{ \floor{\frac{3k+n-1}{2}},\, 3k-1 \} \\
			\nu_1 \le \floor{\frac{3k}{2}},\, \ell(\nu) \le n
		}}
			C_{k,n,\nu} m_{\mu}(\bm{u}^2) \,.
	\]
	The condition $|\nu| \le \min\{ \floor{\frac{3k+n-1}{2}},\, 3k-1 \}$ follows from the bound on the total degree, while $\nu_1 \le \floor{\frac{3k}{2}}$ follows from the bound on the degrees in each individual variable. Notice that we also employed the parity condition, together with the property $m_{2\nu}(\bm{u}) = m_{\nu}(\bm{u}^2)$, so that only monomial symmetric polynomials in the squared variables appear. Moreover, we know that $C_{k,n,\nu}$ is an (explicit) polynomial of $n$ of degree $\le 3k-1-|\nu|$. Inserting the above expansion in the formula for $U_{k,g,n}$, we obtain:
	\[
		U_{k,g,n}(\bm{x})
		=
		\sum_{\substack{
			|\nu| \le \min\{ \floor{\frac{3k+n-1}{2}},\, 3k-1 \} \\
			\nu_1 \le \floor{\frac{3k}{2}} ,\, \ell(\nu) \le n
		}}
			C_{k,n,\nu} \, m_{\nu}(\bm{x}^{-1}) \, h_{3g-3+n-|\nu|}(\bm{x}^{-1}) \,.
	\]
	Since we are interested in extracting coefficients, it is natural to decompose $U_{k,g,n}$ in the basis of monomial symmetric polynomials. For products of the form $m_{\nu} h_{D - |\nu|}$, this is discussed in \cref{app:symmetric:fncts}:
	\[
		m_{\nu} \, h_{D - |\nu|}
		=
		\sum_{|\mu| = D} M_{n,\mu,\nu} \, m_{\mu} \,,
	\]
	where $M_{n,\mu,\nu}$ are (explicit) polynomials of $n$ and $p_m = \#\set{ \mu_i = m}$ for $m = 0,\dots, \nu_1 - 1$. For a fixed tuple $d_1, \dots, d_n \ge 0$ satisfying $|d| = 3g-3+n$, consider the associated partition $\mu$ (since $U_{k,g,n}$ is symmetric, the order does not matter). Then
	\[
		\alpha_k\bigl( n, p_0, \dots, p_{\floor{\frac{3k}{2}}-1} \bigr)
		=
		\bigl[ x_1^{-d_1} \cdots x_n^{-d_n} \bigr] \, U_{k,g,n}(\bm{x})
		=
		\sum_{\substack{
			|\nu| \le \min\{ \floor{\frac{3k+n-1}{2}},\, 3k-1 \} \\
			\nu_1 \le \floor{\frac{3k}{2}} ,\, \ell(\nu) \le n}
		}
			C_{k,n,\nu} \, M_{n,\mu,\nu}
	\]
	is a polynomial of $n$ and $p_m = \#\set{ d_i = m}$ for $m = 1, \dots, \floor{\frac{3k}{2}}-1$. This concludes the proof.
\end{proof}

The three main ingredients that entered in the proof of the above results are:
\begin{enumerate}
	\item the Wronskian equation: $\psi'_+ \psi_- - \psi_+ \psi_-' = 1$;

	\item the homogeneity property and the parity relation satisfied by the formal Airy functions: $\psi_{\pm}(x;\hbar) = x^{-1/4} \, \psi_{\pm}(1;\frac{\hbar}{2A(x)})$, $\psi_{+}(x;\hbar) = \psi_{-}(x;-\hbar)$ and similarly for the derivatives;

	\item the exponent $3/2$ in the instanton action $A(x) = \frac{4}{3} x^{3/2}$, giving rise to the value $\floor{3k/2}$.
\end{enumerate}
This observation will allow us to obtain a similar result for $\Theta$-class and $r$-spin intersection numbers for which, mutatis mutandis, the same properties hold.

We can now extract coefficients from both sides of \cref{eqn:large:g:corr:Y} to get the large genus asymptotics of Witten--Kontsevich intersection numbers. The leading term, already computed in \cref{cor:leading:large:g:WK}, recovers Aggarwal's result. The polynomiality structure of the subleading corrections proved in the previous proposition confirms a conjecture of Guo--Yang \cite[conjecture~1]{GY22}.

\begin{theorem}[{Large genus asymptotics for $\psi$-class intersection numbers}] \label{thm:large:g:WK}
	For any given $n \ge 1$ and $K \ge 0$, uniformly in $d_1,\dots,d_n$ as $g \to \infty$:
	\begin{multline}
		\braket{\tau_{d_1} \cdots \tau_{d_n}} \prod_{i=1}^n (2d_i + 1)!! \\
		=
		\frac{2^n}{4\pi}  \frac{\Gamma(2g-2+n)}{( \frac{2}{3} )^{2g-2+n}}
		\left(
			\sum_{k = 0}^K
				\frac{( \frac{2}{3} )^{k}}{(2g-3+n)^{\underline{k}}} \,
				\alpha_k\bigl( n, p_0, \dots, p_{\floor{\frac{3k}{2}}-1} \bigr)
			+
			\bigO\biggl( \frac{1}{g^{K+1}} \biggr)
			\right) \,.
	\end{multline}
	Each term $\alpha_k$ in the asymptotic expansion is a polynomial in $n$ and $p_m = \#\set{ d_i = m}$ for $m = 1, \dots, \floor{\frac{3k}{2}}-1$ of degree $\le 3k-1$, with the assignment $\deg{n} = 1$ and $\deg{p_m} = m+1$. Moreover, the term $\alpha_k$ can be effectively computed from \cref{eqn:subleading:coeffs}.
\end{theorem}

\begin{proof}
	Our starting point is the application of the Borel transform method, \cref{thm:large:order}, to the $n$-point function. From the proof of \cref{thm:large:order}, after the application of Cauchy's theorem and a deformation of the contour around the origin into several Hankel contours along the logarithmic branch-cuts, we find
	\begin{multline}\label{eqn:asympt:WK:start}
		\frac{W_{g,n}(\bm{x})}{\Gamma(2g-2+n)}
		-
		\sum_{i=1}^n \frac{S}{\pi} \frac{1}{A(x_i)^{2g-2+n}}
		\sum_{k=0}^K
			\frac{A(x_i)^{2g-2+n-k}}{(2g-3+n)^{\underline{k}}} \,
			W_{k,n}^{(i)}(\bm{x}) \\
		=
		(2g - 2 + n) \sum_{i=1}^n \frac{S}{\pi} \, \frac{1}{A(x_i)^{2g-2+n}}
			\int_{0}^{+\infty} e^{-t(2g-2+n)} \, \widehat{W}_{n}^{(i)} \bigl(\bm{x};A(x_i)(e^t-1);K \bigr) \, dt \,.
	\end{multline}
	Here we recall that $S=1$ and $A(x) = \frac{4}{3} x^{3/2}$ are the Stokes constant and the instanton action of the Airy functions, and $\widehat{W}_{n}^{(i)}(\bm{x};s;K)$ denotes the correlator $\widehat{W}_{n}^{(i)}(\bm{x};s)$ minus its Taylor expansion at $s = 0$ up to order $K$. We also notice that the application of the above theorem requires all singularities to be distinct. Since the singularities are located at $\pm A(x_i)$, we require $x_i \neq \pm x_j$ for all $i \ne j$. From \cref{eqn:asympt:WK:start}, it is clear that the dominant contribution to the asymptotics comes from the closest singularity to the origin (see \cref{app:sings} for a visualisation of the phenomenon). Without loss of generality, we can suppose that $R_1 < \cdots < R_n$, so that $i = 1$ is the dominant contribution.

	On the other hand, after multiplication by $(x_1 \cdots x_n)^{3/2}$, the left-hand side is a homogeneous symmetric polynomial in $\bm{x}^{-1}$ of degree $3g-3+n$ (thanks to \cref{thm:genus:expns:Airy,lemma:poly:subleading}). Our goal is to extract the coefficient of a fixed monomial $x_1^{-d_1} \cdots x_{n}^{-d_n}$ and show that the resulting value has the appropriate asymptotics as $g \to \infty$, uniform in $d_1,\dots,d_n$. As the dominant contribution comes from $i = 1$, it is natural to normalise the left-hand side of \cref{eqn:asympt:WK:start} as:
	\begin{multline*}
		\Biggl( \prod_{i=1}^{n} 2x_i^{3/2} \Biggr) (-2A(x_1))^{2g-2+n} \, x_1^{n/2}
		\Bigg[
			\frac{W_{g,n}(\bm{x})}{\Gamma(2g-2+n)} \\
			-
			\sum_{i=1}^n \frac{S}{\pi} \frac{1}{A(x_i)^{2g-2+n}}
			\sum_{k=0}^K
				\frac{A(x_i)^{2g-2+n-k}}{(2g-3+n)^{\underline{k}}} \,
				W_{k,n}^{(i)}(\bm{x})
		\Bigg] .
	\end{multline*}
	Indeed, a simple degree counting shows that the left-hand side is a homogeneous polynomial of degree $3g-3+n$ in $x_1,x_2^{-1},\dots,x_n^{-1}$. Extracting the coefficient of $x_1^{d_1} x_2^{-d_2} \cdots x_n^{-d_n}$ yields
	\begin{equation*}
		\mc{C}_{g,d,K}
		\coloneqq
		\frac{( \frac{2}{3} )^{2g-2+n}}{\Gamma(2g-2+n)}
		\braket{\tau_{d_1} \cdots \tau_{d_n}} \prod_{i=1}^n (2d_i + 1)!!
		-
		\frac{2^n}{4\pi}
		\sum_{k = 0}^K
			\frac{( \frac{2}{3} )^{k}}{(2g-3+n)^{\underline{k}}} \,
			\alpha_k\bigl( n, p_0, \dots, p_{\floor{\frac{3k}{2}}-1} \bigr)
		\,.
	\end{equation*}
	
	The thesis is equivalent to showing that $\mc{C}_{g,d,K} = \bigO( g^{-K-1} )$, uniformly in $d$. In other words, there must exist $g_0 \ge 0$ and $C > 0$, depending on $n$ and $K$ but not on $d$, such that for $g \ge g_0$ we have $|\mc{C}_{g,d,K}| \le C \, g^{-K-1}$. In order to prove such a claim, let us consider the appropriately normalised right-hand of \cref{eqn:asympt:WK:start}. The coefficient extraction can be achieved by applying the operator
	\begin{equation*}
		\mc{O}_{d}
		\colon
		Z \longmapsto
		\left(
			\prod_{i=1}^n \oint_{|x_i| = R_i} \frac{dx_i}{2\pi\iu} \, x_i^{\widetilde{d}_i-1}
		\right) Z(\bm{x}) \,.
	\end{equation*}
	Here $Z$ is any continuous function on $\set{ \bm{x} \in \C^n | |x_i| = R_i }$, $\widetilde{d}_1 = -d_1$ and $\widetilde{d}_j = d_j$ for all $j = 2,\dots,n$. Thus, looking at the right-hand side of \cref{eqn:asympt:WK:start}, we find
	\begin{equation*}
		\mc{C}_{g,d,K}
		=
		(2g-2+n) \, \mc{O}_{d} \biggl[
			\sum_{i=1}^n \left( \frac{x_1}{x_i} \right)^{3g-3+n}
			\int_{0}^{+\infty} e^{-t(2g-2+n)} \, \widehat{Z}_{n}^{(i)}(\bm{x};t;K) \, dt
		\biggr] \,,
	\end{equation*}
	where we have set
	\begin{equation*}
		\widehat{Z}_{n}^{(i)}(\bm{x};t;K)
		\coloneqq
		(-1)^n \frac{S}{\pi} \,
		\Biggl( \prod_{i=1}^{n} 2x_i^{3/2} \Biggr) \,
		\widehat{W}_{n}^{(i)} \bigl(\bm{x};A(x_i)(e^t-1);K \bigr) \,.
	\end{equation*}
	Notice that summing over $n$ we find a polynomial in $x_1,x_2^{-1},\dots,x_n^{-1}$, though the $j$-th summand has simple poles at $x_i = x_j$. However, since we are integrating along circles of different radii $R_1 < \cdots < R_n$, there is no issue with the application of $\mc{O}_{d}$. Furthermore, it should be noted that $\widehat{Z}_{n}^{(i)}$ satisfies the following properties as functions of $t$ for $|x_i| = R_i$ fixed:
	\begin{itemize}
		\item {Polynomial growth at the origin.} $\widehat{Z}_{n}^{(i)}(\bm{x};t;K) = \bigO(t^{K+1})$ as $t \to 0^+$, since the left hand-side is obtained from the truncation of the Taylor series of $\widehat{W}_{n}^{(i)}$ at the origin.

		\item {Exponential growth at infinity.} $\widehat{Z}_{n}^{(i)}(\bm{x};t;K) = \bigO(e^{t\nu})$ as $t \to +\infty$ for some exponent $\nu \in \R_+$, since $\widehat{W}_{n}^{(i)}(\bm{x};A(x_i)s)$ is a sum of convolutions of wave functions with polynomial growth at $s = \infty$. The latter property can be easily checked from the integral representation of the Airy functions.
	\end{itemize}
	From the homogeneity properties of the Airy function, it can be shown that both estimates are uniform in $R_i$, as long as all $R_i$ are fixed in a compact set, say $R_i \in [1,2]$. In order to prove the claim, let us analyse separately the term corresponding to $i = 1$ (the leading term), and the one corresponding to $i > 1$ (the exponentially subleading terms).

	{\sc Leading term.} Consider
	\begin{equation*}
		\mc{C}_{g,d,K}^{\textup{lead}}
		\coloneqq
		(2g-2+n) \,
		\mc{O}_{d} \biggl[
			\int_{0}^{+\infty} e^{-t(2g-2+n)} \, \widehat{Z}_{n}^{(1)}(\bm{x};t;K) \, dt
		\biggr] \,.
	\end{equation*}
	As the integrand is analytic in the integration domain, we can exchange the integral in $t$ with the one in $\bm{x}$ to get
	\begin{equation*}
		| \mc{C}_{g,d,K}^{\textup{lead}} |
		\le
		(2g-2+n)
		\int_{0}^{+\infty} e^{-t(2g-2+n)} \left|
			\mc{O}_{d} \Bigl[ \widehat{Z}_{n}^{(1)}(\bm{x};t;K) \Bigr]
		\right| dt \,.
	\end{equation*}
	Notice that the integrand $| \mc{O}_{d} [ \widehat{Z}_{n}^{(1)}(\bm{x};t;K) ] |$ can be bounded by
	\[
		\left| \mc{O}_{d} \Bigl[ \widehat{Z}_{n}^{(1)}(\bm{x};t;K) \Bigr] \right|
		\le
		\frac{R_{2}^{d_2} \cdots R_{n}^{d_n}}{R_1^{d_1}} \, f_{n,K}(t) \,,
	\]
	where $f_{n,K}$ satisfies the hypothesis of Watson's \cref{Watson:lemma}: it behaves like $\bigO(t^{K+1})$ as $t \to 0^+$ and it has exponential growth as $t \to + \infty$. Both estimates depend on $n$ and $K$, but not on $R_i$ and $d_i$. Applying Watson's \cref{Watson:lemma}, we find
	\[
		\mc{C}_{g,d,K}^{\textup{lead}}
		=
		(2g - 2 + n)
		\frac{R_{2}^{d_2} \cdots R_{n}^{d_n}}{R_1^{d_1}} \,
		\bigO\biggl( \frac{1}{(2g - 2 + n)^{K+2}} \biggr)
		=
		\frac{R_{2}^{d_2} \cdots R_{n}^{d_n}}{R_1^{d_1}} \,
		\bigO\biggl( \frac{1}{g^{K+1}} \biggr) \,.
	\]
	Again, $\bigO( g^{-K-1} )$ depends on $n$ and $K$ but not on $d$.

	{\sc Exponentially subleading terms.} With the same argument as before, we deduce that
	\begin{equation*}
	\begin{split}
		\mc{C}_{g,d,K}^{\textup{sub}}
		& \coloneqq
		(2g-2+n) \,
		\mc{O}_{d} \Biggl[
			\sum_{i=2}^n \left( \frac{x_1}{x_i} \right)^{3g-3+n}
			\int_{0}^{+\infty} e^{-t(2g-2+n)} \, \widehat{Z}_{n}^{(i)}(\bm{x};t;K) \, dt
		\Biggr] \\
		& =
		\frac{R_{2}^{d_2} \cdots R_{n}^{d_n}}{R_1^{d_1}} \,
		\sum_{i=2}^n \left( \frac{R_1}{R_i} \right)^{3g-3+n}
		\bigO\biggl( \frac{1}{g^{K+1}} \biggr) \,.
	\end{split}
	\end{equation*}
	Notice that $(R_1/R_i)^{3g-3+n}$ are exponentially suppressed terms.

	{\sc Conclusion.} To conclude, let us choose the radii $R_i = 1 + \frac{i-1}{n-1} \frac{1}{g}$. In the limit $g \to + \infty$ all radii coincide. It is then natural to expect that all terms would give a contribution behaving as $\bigO(g^{-K-1})$. Indeed, since $|d| = 3g-3+n$, we find
	\[
	\begin{split}
		\mc{C}_{g,d,K}^{\textup{lead}}
		& =
		\left(1 + \frac{1}{g} \right)^{3g-3+n}
		\bigO\biggl( \frac{1}{g^{K+1}} \biggr)
		=
		\bigO\biggl( \frac{1}{g^{K+1}} \biggr) \,, \\
		\mc{C}_{g,d,K}^{\textup{sub}}
		& =
		\left(1 + \frac{n-1}{g} \right)^{3g-3+n} \,
		(n-1) \left( \frac{(n-1)g}{1 + (n-1)g} \right)^{3g-3+n}
		\bigO\biggl( \frac{1}{g^{K+1}} \biggr)
		=
		\bigO\biggl( \frac{1}{g^{K+1}} \biggr) \,.
	\end{split}
	\]
	Here we used the fact that $(1 + \frac{1}{g} )^{3g-3+n} = (1 + \frac{n-1}{g} )^{3g-3+n} = (n-1) ( \frac{(n-1)g}{1 + (n-1)g} )^{3g-3+n} = \bigO(1)$. All estimates depend on $n$ and $K$, but not on $d$. All together, we find the thesis: $\mc{C}_{g,d,K} = \bigO(g^{-K-1})$ uniformly in $d$.
\end{proof}

\subsection{Trans-series structure}\label{subsec:transseries}
For the reader familiar with resurgence and trans-series analysis, it can be shown that the correlators have a full trans-series expansion that can be constructed as follows. Consider the matrix
\begin{equation}
	\ms{M}(x;\hbar)
	=
	\Psi(x;\hbar)
	\,\ms{E}\,
	\Psi^{-1}(x;\hbar) \,,
	\qquad\qquad
	\ms{E}
	=
	\begin{pmatrix}
		\frac{\sigma}{2} & -\sigma_- \\
		\sigma_+ & -\frac{\sigma}{2}
	\end{pmatrix} ,
\end{equation}
which depends on trans-series parameters $\sigma, \sigma_+, \sigma_-$ (but we omit its dependence for the sake of notation simplicity). Notice that $\ms{E}$ is a generic element of $\mf{sl}_2(\C)$, rather than a Cartan element as in \cref{eqn:M:matrix}. We decompose $\ms{M}$ as
\begin{equation}
	\ms{M}
	=
	\sigma \, M
	+
	\sigma_{+} \, e^{\frac{A(x)}{\hbar}} M_+
	+
	\sigma_{-} \, e^{-\frac{A(x)}{\hbar}} M_-
	\,,
\end{equation}
where $M$ and $M_{\pm}$ are the formal power series in $\hbar$ previously encountered. 
We can then define the $n$-point correlators with the exact same determinantal formula as in \cref{eqn:n:pnt:Airy}:
\begin{equation}
	\ms{W}_1(x_1;\hbar)
	\coloneqq
	- \frac{1}{\hbar} \Tr{\bigl( \mathcal{D}(x_1) \ms{M}(x_1;\hbar) \bigr)} \,,
	\qquad
	\ms{W}_n(\bm{x};\hbar)
	\coloneqq
	(-1)^{n-1} \sum_{\sigma \in S_n^{\textup{cyc}}}
	\frac{\Tr{ \bigl( \prod_{i=1}^n \ms{M}(x_{\sigma^i(1)};\hbar) } \bigr) }{\prod_{i=1}^n ( x_{i} - x_{\sigma(i)} )}
	\,.
\end{equation}
In the above definition we have chosen different constants $\sigma_{i}$, $\sigma_{i,+}$, $\sigma_{i,-}$ for different $i=1,\dots,n$. Thus, we have the trans-series expansion
\begin{equation}\label{eqn:transseries}
	\ms{W}_n(\bm{x};\hbar)
	=
	\sum_{I_{+} \sqcup I_{-} \subseteq [n]}
		\prod_{i \not \in I_{+} \sqcup I_{-} } \sigma_{i}
		\prod_{i \in I_{+} } \sigma_{i,+}
		\prod_{i \in I_{-} } \sigma_{i,-}
		\,
		e^{\frac{1}{\hbar} \bigl( \sum_{i \in I_{+}} A(x_i) - \sum_{i \in I_{-}} A(x_i) \bigr)}
		\,
		\ms{W}^{(I_{+},I_{-})}_n(\bm{x};\hbar) \,.
\end{equation}
Notice that, by definition, the perturbative sector $\ms{W}^{(\varnothing,\varnothing)}_{n}$ coincides with the $W_n$ defined in \cref{eqn:n:pnt:Airy} and stores $\psi$-class intersection numbers. Moreover, the $(1,0)$- and $(0,1)$-instanton sectors coincide with the correlators defined in \cref{eqn:1inst:corr}: $\ms{W}^{(\{i\},\varnothing)}_n = W^{(+,i)}_n$ and $\ms{W}^{(\varnothing,\{i\})}_n = W^{(-,i)}_n$. In particular, they govern the large genus asymptotics of the intersection numbers, but they do not have (to the best of our knowledge) an enumerative-geometric interpretation. It would be interesting to explore whether such a geometric interpretation actually exists for all multi-instanton sectors.

Finally, we remark that the trans-series expansion is in accordance with the fact that $\ms{W}_n$ satisfies a linear ODE of order $3^n$ (see \cite{EMO} for a proof). For instance, the $1$-point correlator satisfies
\begin{equation}
	\ms{W}^{'''}_1 - 4x \, \ms{W}'_1 + 2 \hbar \, \ms{W}_1 = 0 \,,
\end{equation}
where $f'$ stands for $\hbar \frac{d}{dx}f$. In particular, a perturbative ansatz of the form $W_1 = \sum_{g \ge 0} w_g \frac{\hbar^{2g-1}}{x^{3g-1/2}}$ would give rise to a recursive relation for the coefficients $w_g$ that can be explicitly solved, and is equivalent to $\braket{\tau_{3g-2}}_{g} = \frac{1}{24^g g!}$ after setting $w_0 = 1$.

\section{Large genus asymptotics of \texorpdfstring{$\Theta$}{Theta}-class intersection numbers}\label{sec:Norbury}

The goal of this section is to prove the large genus asymptotics of $\Theta$-class intersection numbers. To the best of our knowledge, their asymptotics is not known in the literature. The strategy is, mutatis mutandis, the same as the one employed in the previous section: we apply the Borel transform method to the $n$-point function, which in turn is built out of formal solutions of the Bessel differential equations through the determinantal formulae. We are going to skip most of the details, and highlight only the substantial changes compared to the previous section.

Before proceeding with the details of the determinantal formula, let us briefly recall the definition of these intersection numbers. The \textit{$\Theta$-class}, introduced by Norbury in \cite{Nor23}, is a collection of cohomology classes (more precisely, a cohomological field theory) of pure degree:
\begin{equation}
	\Theta_{g,n} \in H^{2(2g-2+n)}(\Mbar_{g,n},\Q) \,.
\end{equation}
Their construction goes through the moduli space of spin curves: let $\Mbar_{g,n}^{\textup{spin}}$ be the moduli space parametrising roots (also known as spin structures) of the form
\begin{equation}
	L^{\otimes 2} \cong \omega_C \left( \sum_{i=1}^n p_i \right) \,,
\end{equation}
where $[C,p_1,\dots,p_n] \in \Mbar_{g,n}$ and $\omega_C$ denotes the canonical line bundle on $C$. There is a forgetful map $\pi \colon \Mbar_{g,n}^{\textup{spin}} \to \Mbar_{g,n}$ that forgets the spin structure. Consider the vector bundle $\mc{E}_{g,n} \to \Mbar_{v,n}^{\textup{spin}}$ whose fibre over $[C,p_1,\dots,p_n,L]$ is $H^1(C,L^{\ast})^{\ast}$. The $\Theta$-class is then defined as the push-forward of the (normalised) top Chern class of $\mc{E}_{g,n}$:
\begin{equation}
	\Theta_{g,n}
	\coloneqq
	2^{g-1+n} \, \pi_{\ast} c_{\textup{top}}\left( \mc{E}_{g,n} \right) .
\end{equation}
Define the $\Theta$-class intersection numbers as
\begin{equation}
	\braket{\tau_{d_1} \cdots \tau_{d_n}}^{\Theta}
	\coloneqq
	\int_{\Mbar_{g,n}}
		\Theta_{g,n} \prod_{i=1}^n \psi_i^{d_i} \,.
\end{equation}
They are non-zero only for $|d| = g-1$. In \cite{Nor23}, Norbury conjectures that the partition function associated the $\Theta$-class coincides with Brézin--Gross--Witten solution of the KdV hierarchy, discovered in the '80s in the context of unitary matrix models \cite{BG80,GW80}. Since then, a proof of Norbury's conjecture has been proposed in \cite{CGG} through Givental formalism and its connection with topological recursion.

In this section we are interested in yet another method for computing $\Theta$-class intersection numbers, namely the determinantal formula with building blocks being the asymptotic solutions of the ($\hbar$-de\-pen\-dent) \textit{Bessel ODE}:
\begin{equation}
	\left( \hbar^2 \frac{d}{dx} x \frac{d}{dx} - 1 \right) \psi(x;\hbar) = 0 \,.
\end{equation}
The general solution is given by the Bessel integral: $\int_{\gamma} e^{-\frac{1}{\hbar}V(t,x)} \frac{dt}{t}$ with $V(t,x) \coloneqq -\frac{1}{t} - xt$ and $\gamma$ a properly chosen integration contour. The asymptotic solutions constructed through Lefschetz thimbles are given by
\begin{equation}
	\psi_{\pm}(x;\hbar)
	\coloneqq 
	\frac{e^{\mp \frac{V(x)}{\hbar}}}{\sqrt{2}} x^{-1/4}
	\sum_{k = 0}^{\infty}
		\frac{(\tfrac{1}{2})^{\overline{k}} (\tfrac{1}{2})^{\overline{k}}}{2^k k!} \left(\mp \frac{\hbar}{V(x)} \right)^{k} \,,
\end{equation}
where $\pm V(x) \coloneqq \mp 2 x^{1/2}$ are the critical values of the potential. Here $(x)^{\overline{k}} \coloneqq x (x+1) \cdots (x+k-1)$ denotes the rising factorial. Notice that $\psi_{\pm}$ correspond to the (properly renormalised) asymptotic expansions of the modified Bessel functions $\mathrm{K}_0$ and $\mathrm{I}_0$. The Bessel ODE can be re-written as a $2 \times 2$ system:
\begin{equation}
	\hbar \frac{d}{dx} \Psi_0(x;\hbar) = \mathcal{D}_0(x;\hbar) \Psi_0(x;\hbar) \,,
	\qquad
	\mathcal{D}_0(x;\hbar)
	=
	\begin{pmatrix}
		0 & 1 \\
		\tfrac{1}{x} & -\tfrac{\hbar}{x}
	\end{pmatrix} \,.
\end{equation}
However, we notice that $\det(\Psi_0(x;\hbar)) = x^{-1}$ is not constant. We can gauge-transform the above system so that the determinant of the wave matrix is normalised to $1$. Setting $\Psi(x;\hbar) \coloneqq \left(\begin{smallmatrix} 1 & 0 \\ 0 & x \end{smallmatrix}\right) \Psi_0(x;\hbar)$, the system becomes
\begin{equation}
	\hbar \frac{d}{dx} \Psi(x;\hbar) = \mathcal{D}(x) \Psi(x;\hbar) \,,
	\qquad\qquad
	\mathcal{D}(x)
	=
	\begin{pmatrix}
		0 & \tfrac{1}{x} \\
		1 & 0
	\end{pmatrix} \,.
\end{equation}
Again, we denote the rows of $\Psi$ as $\psi_{\pm}$ and $\psi_{\pm}'$. Because of the gauge transformation, we have $\psi_{\pm}' = x \hbar \frac{d}{dx}\psi_{\pm}$. After separating the exponential part, denoting the resulting formal power series with a tilde, and taking the Borel transforms, we find
\begin{equation}
	\widehat{\psi}_{\pm}(x;s)
	=
	\frac{x^{-1/4}}{\sqrt{2}} \, \pFq{2}{1} \left( \tfrac{1}{2},\tfrac{1}{2} ; 1 ; \pm \frac{s}{A(x)} \right) \,,
	\qquad
	\widehat{\psi}_{\pm}'(x;s)
	=
	\pm \frac{x^{1/4}}{\sqrt{2}} \, \pFq{2}{1} \left( \tfrac{3}{2},-\tfrac{1}{2} ; 1 ; \pm \frac{s}{A(x)} \right) \,,
\end{equation}
where $\pm A(x) \coloneqq \pm 4 x^{1/2}$ are the instanton actions for the Bessel ODE. In particular, they converge in a disc of positive radius centred at the origin, and can be extended analytically to functions with logarithmic singularities at $s = \pm A(x)$. One can check the following behaviour at the singularities:
\begin{equation}
\begin{aligned}
	\widehat{\psi}_{\pm}(x;s)
	& =
		- \frac{S}{2\pi} \,
		\widehat{\psi}_{\mp}\bigl( x;s \mp A(x) \bigr)
		\log\bigl( s \mp A(x) \bigr)
		+
		\text{holomorphic at } \pm A(x) \,, \\
	\widehat{\psi}_{\pm}'(x;s)
	& =
		- \frac{S}{2\pi} \,
		\widehat{\psi}_{\mp}'\bigl( x;s \mp A(x) \bigr)
		\log\bigl( s \mp A(x) \bigr)
		+
		\text{holomorphic at } \pm A(x) \,,
\end{aligned}
\end{equation}
where the Stokes constant is $S = 2$. Again, $\widehat{\psi}_{\pm}$ and $\widehat{\psi}_{\pm}'$ are simple resurgent functions whose singularity structure is fully under control.

Consider now the matrix $M \coloneqq \frac{1}{2} \, \Psi \left(\begin{smallmatrix} 1 & 0 \\ 0 & -1 \end{smallmatrix} \right) \Psi^{-1}$, and define the $n$-point \textit{Bessel correlators} via determinantal formulae:
\begin{equation}\label{eqn:n:pnt:Bessel}
	W_1(x_1;\hbar)
	\coloneqq
	- \frac{1}{\hbar} \Tr{\bigl( \mathcal{D}(x_1) M(x_1;\hbar) \bigr)} \,,
	\qquad
	W_n(\bm{x};\hbar)
	\coloneqq
	(-1)^{n-1} \sum_{\sigma \in S_n^{\textup{cyc}}}
	\frac{\Tr{ \bigl( \prod_{i=1}^n M(x_{\sigma^i(1)};\hbar) } \bigr) }{\prod_{i=1}^n ( x_{i} - x_{\sigma(i)} )} \,.
\end{equation}
Again, one finds a genus expansion of the correlators with the expansion coefficients storing interesting enumerative invariants, namely $\Theta$-class intersection numbers. The determinantal formula for such intersection numbers was deduced in \cite{DYZ21} assuming the validity of Norbury's conjecture \cite{Nor23}, later proved in \cite{CGG}.

\begin{theorem}[{Genus expansion of the Bessel correlators \cite{DYZ21}}]
	The $n$-point Bessel correlators $W_{n}$ admit the following $\hbar$-expansion:
	\begin{equation}\label{eqn:genus:expns:Bessel}
		W_{n}(\bm{x};\hbar)
		=
		\sum_{g = 0}^{\infty}
			\hbar^{2g-2+n} \, W_{g,n}(\bm{x}) \,.
	\end{equation}
	Moreover, $W_{g,n}$ stores $\Theta$-class intersection numbers: if $2g-2+n>0$, 
	\begin{equation}
		W_{g,n}(\bm{x})
		=
		(-2)^{-(2g-2+n)}
		\sum_{\substack{ d_{1},\dots,d_n \ge 0 \\ d_{1}+\cdots+d_n = g-1 }}
			\braket{ \tau_{d_1} \cdots \tau_{d_n} }^{\Theta}
			\prod_{i=1}^n \frac{(2d_i+1)!!}{2 \, x_i^{d_i+3/2}} \,.
	\end{equation}
\end{theorem}

\begin{remark}[Connection with topological recursion]
	For the reader familiar with topological recursion, $W_{g,n}$ coincides with topological recursion differentials computed from the Bessel spectral curve $( \P^1, \, x(z) = z^2, \, y(z) = z^{-1}, \, B(z_1,z_2) = \frac{dz_1 dz_2}{(z_1 - z_2)^2} )$ -- see \cite{DN19}:
	\begin{equation}
		\omega_{g,n}(\bm{z})
		=
		W_{g,n}(\bm{x}) \, dx_1 \cdots dx_n \big|_{x_i = z_i^2} \,.
	\end{equation}
\end{remark}

Using the exact same strategy as before, we find the large genus behaviour of $\Theta$-class intersection numbers. In a nutshell:
\begin{itemize}
	\item Write the correlators in terms of Bessel kernels (see also \cite{TW94b}).

	\item Study the Borel plane singularity structure of the Bessel correlators. One finds that, on the principal sheet, the Borel transform of $W_n$ has $2n$ logarithmic singularities located at $s = \pm A(x_i)$ and such that
	\begin{equation}
		\widehat{W}_{n}(\bm{x};s)
		=
		- \frac{S}{2\pi} \,
		\widehat{W}^{(\pm,i)}_{n}\bigl( \bm{x}; s \mp A(x_i) \bigr) \,
		\log\bigl( s \mp A(x_i) \bigr)
		+\text{holomorphic at } \pm A(x_i) \,.
	\end{equation}
	Here $A(x) = 4 x^{1/2}$ and $S = 2$ are the instanton action and the Stokes constant associated to the Bessel functions respectively, and $\widehat{W}^{(\pm,i)}_{n}$ are defined through determinantal formulae as in \cref{eqn:1inst:corr}.

	\item Deduce the large genus asymptotics of the correlators through the Borel transform method, and get the large genus asymptotics of the intersection numbers by extracting coefficients.
\end{itemize}

The final result is the following asymptotic formula. We omit the proof, which is completely analogous to that of \cref{thm:large:g:WK}.

\begin{theorem}[{Large genus asymptotics for $\Theta$-class intersection numbers}] \label{thm:large:g:Nor}
	For any given $n \ge 1$ and $K \ge 0$, uniformly in $d_1,\dots,d_n$ as $g \to \infty$:
	\begin{multline}
		\braket{\tau_{d_1} \cdots \tau_{d_n}}^{\Theta} \prod_{i=1}^n (2d_i + 1)!! \\
		=
		\frac{2^n}{2\pi}  \frac{\Gamma(2g-2+n)}{2^{2g-2+n}}
		\left(
			\sum_{k = 0}^K
				\frac{2^{k}}{(2g-3+n)^{\underline{k}}} \,
				\beta_k \bigl( n, p_0, \dots, p_{\floor{\frac{k}{2}}-1} \bigr)
			+
			\bigO\biggl( \frac{1}{g^{K+1}} \biggr)
			\right) \,.
	\end{multline}
	Each term $\beta_k$ in the asymptotic expansion is a polynomial in $n$ and $p_m = \#\set{ d_i = m}$ for $m = 1, \dots, \floor{\frac{k}{2}}-1$ of degree $\le k-1$, with the assignment $\deg{n} = 1$ and $\deg{p_m} = m+1$, that can be effectively computed.
\end{theorem}

The polynomials $Q_{k,n}$ (defined in complete analogy to the Witten--Kontsevich case) and the coefficients $\beta_k$ for the first few values of $k$ are given in \cref{table:Pkn:ck:Theta}. Notice that the complexity of $\beta_k$ is much lower than the corresponding one for the Witten--Kontsevich case. This is due to the fact that $\Theta$-class intersection numbers are computationally less involved that $\psi$-class intersection numbers. Indeed, the $\Theta$-class fills $2g-2+n$ cohomological degrees, which in turn are spared to $\psi$-classes.

\begin{table}
\centering
{\renewcommand{\arraystretch}{1.15}
\begin{tabularx}{\textwidth} { 
	c
	| >{\raggedright\arraybackslash}X 
	| >{\raggedright\arraybackslash}X }
	\toprule
	$k$ & $Q_{k,n}$ & $\beta_k(n,\bm{p})$ \\
	\midrule
	$0$ & $1$ & $1$
	\\
	\midrule
	$1$ & $- \frac{1}{4}$ & $- \frac{1}{4}$
	\\
	\midrule
	$2$
	& $
		\tfrac{9 - 4 n}{8} m_{\varnothing} + \tfrac{1}{8} m_{(1)}$
	&  $
		\tfrac{9 - 4 n}{8} + \tfrac{n - p_0}{8}$
	\\
	\midrule
	$3$
	& $
		- \tfrac{57 - 44 n + 8 n^2}{128} m_{\varnothing}
		- \tfrac{13 - 4 n}{32} m_{(1)}
		- \tfrac{1}{8} m_{(1^2)} $
	& $
		- \tfrac{57 - 44 n + 8 n^2}{128}
		- \tfrac{(13 - 4 n)(n - p_0)}{32}
		- \tfrac{(n - p_0)^{\underline{2}}}{16}
		$ \\
	\bottomrule
\end{tabularx}
}
\caption{
	The polynomials $Q_{k,n}(u_1,\dots,u_n)$ and the coefficients $\beta_k$ for $k = 0,1,2,3$. Here $m_{\lambda}$ denotes the monomial symmetric polynomial in the variables $u_1^2,\dots,u_n^2$.
}
\label{table:Pkn:ck:Theta}
\end{table}

\section{Large genus asymptotics of \texorpdfstring{$r$}{r}-spin intersection numbers}\label{sec:rspin}
The goal of this section is to prove the large genus asymptotics of $r$-spin intersection numbers, introduced for the first time by Witten \cite{Wit92,Wit93} in the context of topological gravity coupled to a Wess--Zumino--Witten theory. Compared to the previous sections, this case contains novel effects due to the underlying differential system being of order $r \times r$ (rather than $2 \times 2$). As a consequence, the number of singularities in the Borel plane for the wave function is $r-1$, which translates into $2(r-1)n$ singularities for the $n$-point function. Thus, the large genus asymptotics of $r$-spin intersection numbers manifests the novel feature of exponentially suppressed contributions coming from the singularities that are further away from the origin. In the literature, an asymptotic analysis of $r$-spin intersection numbers was recently considered only for $n = 1$ by Dubrovin--Yang--Zagier in \cite[theorem~5 (vii)]{DYZ}. We emphasise that their asymptotic formula only considers the exponentially dominant contribution and is valid at first order in $g$. In this section, we generalise their result to arbitrary $n$, and taking into account all subleading and exponentially subleading contributions. Once again, the strategy consists in analysing the determinantal formula through the Borel transform method.

Before proceeding with the large genus asymptotics, let us briefly recall the definition of these intersection numbers. Fix $r \ge 2$. The \textit{Witten $r$-spin class} is a collection of pure-dimensional cohomology classes (to be precise, a cohomological field theory) depending on some parameters $a = (a_1,\dots,a_n) \in \set{1,\dots,r-1}^n$ called \textit{primary fields}\footnote{
	In the literature, a different convention is often found, with all primary fields shifted by a unit: $a_i \mapsto a_i -1$.
}:
\begin{equation}
	W_{g,n}^{r}(a) \in H^{2 D_{g,n}^{r}(a)}(\Mbar_{g,n},\Q) \,.
\end{equation}
Here $D_{g,n}^{r}(a) = \frac{(r-2)(g-1)-n+|a|}{r}$, and the class is set to be zero if $(r-2)(g-1)-n+|a| \not\equiv 0 \pmod{r}$. In genus $0$, the construction was first carried out by Witten using $r$-spin structures, and we briefly recall it here. Let $\Mbar_{0,n}^{r}(a)$ be the moduli space parametrising $r$-th roots of the form
\begin{equation}
	L^{\otimes r} \cong \omega_{C} \biggl( - \sum_{i=1}^n a_i p_i \biggr) ,
\end{equation}
where $[C,p_1,\dots,p_n] \in \Mbar_{0,n}$ and $\omega_C$ denotes the canonical line bundle on $C$. There is a forgetful map $\pi \colon \Mbar_{0,n}^{r}(a) \to \Mbar_{0,n}$ that forgets the $r$-th root. Consider the vector bundle $\mc{V}_{0,n}^{r}(a) \to \Mbar_{0,n}^{r}(a)$ whose fibre over $[C,p_1,\dots,p_n,L]$ is $H^1(C,L)^{\ast}$. The Witten $r$-spin class\footnote{
	The reader might notice the similarity to the definition of the $\Theta$-class when $r$ is set to $2$. The main difference is in the definition of the fibres of the vector bundle: $H^1(C,L)^{\ast}$ for Witten, $H^1(C,L^{\ast})^{\ast}$ for Norbury. This similarity has been extensively studied in \cite{CGG}, where ``higher spin'' $\Theta$-classes have been defined.
} is defined as the push-forward of the (normalised) top Chern class of $\mc{V}_{0,n}^{r}(a)$:
\begin{equation}
	W_{0,n}^{r}(a)
	\coloneqq
	r \cdot \pi_{\ast} c_{\textup{top}}\left( \mc{V}_{0,n}^{r}(a) \right) .
\end{equation}
The definition of the Witten class in higher genera is much more complicated, since the fibres $H^1(C,L)^{\ast}$ no longer glue together to form a vector bundle. The construction was first carried out by Polishchuk--Vaintrob \cite{PV00}, and was later simplified by Chiodo \cite{Chi06}. We refer the reader to \cite{PPZ15} for more details on the Witten classes, and its connection to (the conjecturally full set of) relations on the tautological ring of the moduli space of curves.

The $r$-spin intersection numbers are defined as
\begin{equation}
	\braket{\tau_{d_1,a_1} \cdots \tau_{d_n,a_n}}^{r\textup{-spin}}
	\coloneqq
	\int_{\Mbar_{g,n}}
		W_{g,n}^{r}(a_1,\dots,a_n)
		\prod_{i=1}^n \psi_i^{d_i} \,.
\end{equation}
Notice that the numbers $\braket{\tau_{d_1,a_1} \cdots \tau_{d_n,a_n}}^{r\textup{-spin}}$ are set to be zero unless $D_{g,n}^{r}(a) + |d| = 3g-3+n$. Similar to the cases of $\psi$- and $\Theta$-class intersection numbers, $r$-spin intersection numbers are also connected to integrable hierarchies. Specifically, the associated partition function is the unique tau function of the $r$-KdV hierarchy satisfying a particular initial condition, called the string equation, as demonstrated by Faber--Shadrin--Zvonkine in \cite{FSZ10}. Another approach to computing $r$-spin intersection numbers is through the determinantal formula \cite{BDY18,BDY21}, which relies on the asymptotic solution of the $r$-Airy ODE as its fundamental component.

\subsection{Higher Airy functions}
Consider the ($\hbar$-dependent) \textit{$r$-Airy ODE}:
\begin{equation}
	\left( \left( \hbar \frac{d}{dx} \right)^r - x \right) \psi(x;\hbar) = 0 \,.
\end{equation}
The general solution is given by the $r$-Airy integral: $\int_{\gamma} e^{-\frac{1}{\hbar}V(t,x)} dt$ with $V(t,x) \coloneqq \frac{t^{r+1}}{r+1} - xt$ and $\gamma$ a properly chosen integration contour. Again, we are interested in asymptotic solutions constructed through Lefschetz thimbles: $\psi_{\alpha}(x;\hbar)$ for $\alpha = 1, \dots, r$. Their expression, together with that of their derivatives, is given by
\begin{equation}
	\psi_{\alpha}^{(m)}(x;\hbar)
	\coloneqq 
	(-1)^{\frac{r-\alpha+2}{2}}
	\zeta^{\alpha(m + \frac{1}{2})}
	\frac{e^{-\frac{V_{\alpha}(x)}{\hbar}}}{\sqrt{r}}
	x^{-\frac{r-1-2m}{2r}}
	\sum_{k = 0}^{\infty}
		a_k^{(m)} \left( - \frac{\hbar}{V_{\alpha}(x)} \right)^{k} \,,
\end{equation}
where $V_{\alpha}(x) \coloneqq - \zeta^{\alpha} \frac{r}{r+1} x^{(r+1)/r}$ and $\zeta \coloneqq e^{\frac{2\pi\iu}{r}}$. The coefficients $a_{k}^{(m)}$ are independent of $\alpha$ and are computed recursively from the differential equation (see for instance \cite{CCGG}) as
\begin{equation}
\begin{cases}
	a_{k}^{(m)} = a_{k}^{(m-1)} - \left( k - \frac{1}{2} - \frac{m}{r+1} \right) a_{k-1}^{(m-1)}
	& \text{if $m = 1,\dots, r$} \\
	a_{k}^{(0)} = a_{k}^{(r)}
\end{cases}
\end{equation}
with $a_{0}^{(m)} = 1$. With the above normalisation, we can see that the Wronskian is constantly $1$. Indeed, the ODE implies that the Wronskian is constant, and we can compute its value at $x \to \infty$ as
\begin{equation}
	\det\left(
		(-1)^{\frac{r-\alpha+2}{2}}
		\zeta^{\alpha(m + \frac{1}{2})}
		\frac{ e^{-\frac{V_{\alpha}(x)}{\hbar}} }{ \sqrt{r} }
		x^{-\frac{r-1-2m}{2r}}
	\right)_{\substack{ m = 0,\dots,r-1 \\ \alpha = 1,\dots,r}}
	=
	\frac{(-1)^{\frac{(r-1)(r-2)}{4}}}{r^{r/2}}
	\det\bigl(
		\zeta^{\alpha m}
	\bigr)_{\substack{ m = 0,\dots,r-1 \\ \alpha = 1,\dots,r}} \,.
\end{equation}
The Vandermonde determinant on roots of unity was computed by Schur \cite{Sch21} and it reads $\det( \zeta^{\alpha m} ) = r^{r/2} (-1)^{\frac{(r-1)(3r+2)}{4}}$, so that the Wronskian simplifies to $1$.

Following the prescription of \cref{subsec:exp:int} we find that, after separating the exponential part as $\psi_{\alpha}^{(m)} \eqqcolon e^{-V_{\alpha}/\hbar} \, \widetilde{\psi}_{\alpha}^{(m)}$, the Borel transform of $\widetilde{\psi}_{\alpha}^{(m)}$ is simple resurgent with logarithmic singularities at $s = A_{\alpha,\beta}(x)$ for $\beta \neq \alpha$, where
\begin{equation}
	A_{\alpha,\beta}(x) \coloneqq V_{\beta}(x) - V_{\alpha}(x)
	=
	\frac{r}{r+1} \bigl( \zeta^{\alpha} - \zeta^{\beta} \bigr) x^{(r+1)/r} \,.
\end{equation}
The behaviour at the singularities is given by
\begin{equation}
\begin{aligned}
	\widehat{\psi}_{\alpha}(x;s)
	& =
	- \frac{S_{\alpha,\beta}}{2\pi} \,
	\widehat{\psi}_{\beta}\bigl( x;s - A_{\alpha,\beta}(x) \bigr)
	\log\bigl( s - A_{\alpha,\beta}(x) \bigr)
	+ \text{holomorphic at } A_{\alpha,\beta}(x)
	 \,,
\end{aligned}
\end{equation}
where the Stokes constants $S_{\alpha,\beta}$ are given in terms of constants $\ms{S}_{\alpha,\beta}$ computed as Lefschetz thimble intersection numbers (see \cref{eqn:Stokes:rAiry}) and are worth
\begin{equation}
	S_{\alpha,\beta}
	=
	(-1)^{\frac{\beta-\alpha+1}{2}} \, \ms{S}_{\alpha,\beta}
	= 
	\begin{cases}
		+(-1)^{\frac{\beta-\alpha+1}{2}} & \text{if } \alpha > \beta \,,\\
		-(-1)^{\frac{\beta-\alpha+1}{2}} & \text{if } \alpha < \beta \,.
	\end{cases}
\end{equation}
The above ODE can be re-written as a first order $r \times r$ system:
\begin{equation}
	\hbar \frac{d}{dx} \Psi(x;\hbar) = \mathcal{D}(x) \Psi(x;\hbar) \,,
	\qquad\qquad
	\mathcal{D}(x;\hbar)
	=
	\begin{pmatrix}
		0 & 1 & & \\
		& \ddots & \ddots & \\
		& & \ddots & 1 \\
		x & & & 0
	\end{pmatrix} \,.
\end{equation}
A simple computation shows that the matrix $\Phi \coloneqq \Psi^{-t}$, called the dual wave matrix, satisfies the differential system $-\hbar \frac{d}{dx}\Phi = \mc{D}^{t} \Phi$. In other words, $\Phi$ is a (permutation of the) companion matrix associated to the higher Airy ODE, with $\hbar \mapsto - \hbar$. In order to fix the normalisation, it is sufficient to compute $\Phi$ at $x \to \infty$. We find that
\begin{equation}
	\Phi(x;\hbar)
	=
	\left(
		(-1)^{r-\alpha} \phi_{\alpha}^{(r-m-1)}
	\right)_{\substack{ m = 0,\dots,r-1 \\ \alpha = 1,\dots,r}} ,
	\qquad
	\phi_{\alpha}^{(k)}(x;\hbar)
	\coloneqq
	\psi_{\alpha}^{(k)}(x;-\hbar) \,.
\end{equation}
Since the $\phi$'s are obtained from the $\psi$'s simply by substituting $\hbar$ with $-\hbar$, they enjoy similar properties. In particular we can deduce that, after separating the exponential part as $\phi_{\alpha}^{(m)} \eqqcolon e^{+V_{\alpha}/\hbar} \, \widetilde{\phi}_{\alpha}^{(m)}$, the Borel transform of $\widetilde{\psi}_{\alpha}^{(m)}$ is simple resurgent with logarithmic singularities at $s = - A_{\alpha,\beta}$ for $\beta \neq \alpha$. Its behaviour is given by
\begin{equation}
\begin{aligned}
	\widehat{\phi}_{\alpha}(x;s)
	& =
	- \frac{S_{\alpha,\beta}}{2\pi} \,
	\widehat{\phi}_{\beta}\bigl( x;s + A_{\alpha,\beta}(x) \bigr)
	\log\bigl( s + A_{\alpha,\beta}(x) \bigr)
	+ \text{holomorphic at } -A_{\alpha,\beta}(x)
	\,,
\end{aligned}
\end{equation}
where the Stokes constants are the same as before.

\subsection{Higher Airy correlators}
The Lie algebra $\mf{sl}_r(\C)$ admits the root space decomposition into traceless diagonal matrices, upper-diagonal matrices, and lower-diagonal matrices: $\mf{sl}_r(\C) = \mf{h} \oplus \mf{n}_+ \oplus \mf{n}_-$. A basis of $\mf{h}$ is given in terms of elementary matrices $e_{i,j} = (\delta_{i,k} \delta_{j,l})$ by 
\begin{equation}
	E_a \coloneqq e_{a,a} - e_{a+1,a+1} \in \mf{h} \,,
	\qquad\qquad
	a = 1,\dots,r-1 \,.
\end{equation}
Define the matrix
\begin{equation}
	M(x;\hbar)
	\coloneqq
	\Psi(x;\hbar) E \Psi^{-1}(x;\hbar) \,,
	\qquad\qquad
	E
	\coloneqq
	\frac{1}{r} \sum_{a=1}^{r-1} a E_a \,,
\end{equation}
which is an $\mf{h}$-valued formal power series in $\hbar$. Indeed, as in the Airy case, $M$ contains only quadratic expressions in the functions $\psi$'s and $\phi$'s with opposite exponential parts. The result is then a formal power series in $\hbar$. As before, one can consider the $n$-point \textit{$r$-Airy correlators} defined through determinantal formulae:
\begin{equation}\label{eqn:n:pnt:rspin}
\begin{aligned}
	& W_1(x_1;\hbar)
	\coloneqq
	- \frac{1}{\hbar} \Tr{\bigl( \mathcal{D}(x_1) M(x_1;\hbar) \bigr)} \,, \\
	& W_n(x_1,\dots,x_n;\hbar)
	\coloneqq
	(-1)^{n-1} \sum_{\sigma \in S_n^{\textup{cyc}}}
	\frac{\Tr{ \bigl( \prod_{i=1}^n M(x_{\sigma^i(1)};\hbar) } \bigr) }{\prod_{i=1}^n ( x_{i} - x_{\sigma(i)} )}
	\qquad\quad
	\text{for }n \ge 2 \,.
\end{aligned}
\end{equation}
Again, the genus expansion of the correlators stores interesting enumerative invariants, the $r$-spin intersection numbers. The determinantal formula for such intersection numbers was proved by Bertola--Dubrovin--Yang \cite{BDY18,BDY21} as a consequence of the Faber--Shadrin--Zvonkine result \cite{FSZ10}. Adjusted to our chosen normalisation, and with the insertion of $\hbar$ by homogeneity, the result of Bertola--Dubrovin--Yang reads as follows.

\begin{theorem}[{Genus expansion of the higher Airy correlators \cite{BDY18,BDY21}}]
	The higher Airy correlators $W_{n}$ admit the following $\hbar$-expansion:
	\begin{equation}\label{eqn:genus:expns:rAiry}
		W_{n}(\bm{x};\hbar)
		=
		\sum_{g = 0}^{\infty}
			\hbar^{2g-2+n} \, W_{g,n}(\bm{x}) \,.
	\end{equation}
	Moreover, $W_{g,n}$ stores $r$-spin intersection numbers: if $2g-2+n>0$, 
	\begin{equation}
		W_{g,n}(\bm{x})
		=
		\!\!\!
		\sum_{\substack{ d_{1},\dots,d_n \ge 0 \\ a_1,\dots,a_n \in \{1,\dots,r-1\} \\ r|d|+|a| = (r+1)(2g-2+n) }}
		\!\!\!
			(-r)^{g-1-|d|}
			\braket{ \tau_{d_1,a_1} \cdots \tau_{d_n,a_n} }^{r\textup{-spin}}
			\prod_{i=1}^n \frac{(rd_i+a_i)!_{(r)}}{r \, x_i^{d_i+\frac{a_i}{r}+1}} \,.
	\end{equation}
	Here $m!_{(r)}$ denotes the $r$-factorial, see \cref{eqn:r:fact}. The degree condition is equivalent to $D_{g,n}^r(a) + |d| = 3g-3+n$, where $D_{g,n}^r(a)$ is the complex cohomological degree of the Witten class.
\end{theorem}

\begin{remark}[Connection with topological recursion]
	For the reader familiar with topological recursion, $W_{g,n}$ coincides with the topological recursion differentials computed from the $r$-Airy spectral curve $( \P^1, \, x(z) = z^r, \, y(z) = z, \, B(z_1,z_2) = \frac{dz_1 dz_2}{(z_1 - z_2)^2} )$ -- see \cite{BE17}:
	\begin{equation}
		\omega_{g,n}(\bm{z})
		=
		W_{g,n}(\bm{x}) \, dx_1 \cdots dx_n \big|_{x_i = z_i^r} \,.
	\end{equation}
	Throughout this section, we will work with a fixed choice of $r$-th root of $x$.
\end{remark}

The $n$-point correlators can be expressed through determinantal formulae in the kernel form. We omit the proof, which is completely analogous to \cref{lemma:kernel}.

\begin{lemma}
	For $n \ge 2$, the $n$-point correlators are given by the two equivalent expressions
	\begin{equation}
		W_n(\bm{x};\hbar)
		=
		(-1)^{n-1} \sum_{\sigma \in S_n^{\textup{cyc}}} \prod_{i=1}^n K_{+,-}(x_i,x_{\sigma(i)};\hbar)
		= 
		(-1)^{n-1} \sum_{\sigma \in S_n^{\textup{cyc}}} \prod_{i=1}^n K_{-,+}(x_i,x_{\sigma(i)};\hbar) \,,
	\end{equation}
	where the kernels $K_{\pm}$ are defined by
	\begin{equation}
	\begin{aligned}
		K_{+,-}(x,y;\hbar)
		& \coloneqq
		\frac{
			\sum_{m=0}^{r-1} \widetilde{\psi}_{r}^{(m)}(x;\hbar) \, \widetilde{\phi}_{r}^{(r-1-m)}(y;\hbar)
		}{x-y} \,, \\
		K_{-,+}(x,y;\hbar)
		& \coloneqq
		- \frac{
			\sum_{m=0}^{r-1} \widetilde{\phi}_{r}^{(m)}(x;\hbar) \, \widetilde{\psi}_{r}^{(r-1-m)}(y;\hbar)
		}{x-y} \,.
	\end{aligned}
	\end{equation}
	Moreover, the kernels are related by the parity relation $K_{-,+}(x,y;\hbar) = - K_{+,-}(x,y;-\hbar)$.
\end{lemma}


\subsection{Singularity structure and large genus asymptotics}
From the usual properties of the Borel transform, we immediately find the singularity structure of $W_n$ in the Borel plane from the knowledge of the singularity structure of the functions $\widehat{\psi}$ and $\widehat{\phi}$. Again, we assume $(x_1,\dots,x_n)$ in a generic position.

\begin{proposition}[{Singularity structure of the higher Airy correlators}]
	The Borel transform $\widehat{W}_n(\bm{x};s)$ of the $n$-point $r$-Airy correlator is simple resurgent. More precisely, on the principal sheet, $\widehat{W}_n$ has $2(r-1)n$ logarithmic singularities located at $s = \pm A_{r,\alpha}(x_i)$ (with $\alpha = 1,\dots,r-1$ and $i = 1,\dots,n$) and such that
	\begin{equation}
		\widehat{W}_{n}(\bm{x};s)
		=
		- \frac{S_{r,\alpha}}{2\pi} \,
		\widehat{W}_{n}^{(\pm\alpha,i)}\bigl( \bm{x}; s \mp A_{r,\alpha}(x_i) \bigr)
		\log\bigl( s \mp A_{r,\alpha}(x_i) \bigr)
		+ \text{holomorphic at } \pm A_{r,\alpha}(x_i) \,,
	\end{equation}
	where:
	\begin{itemize}
		\item $A_{r,\alpha}(x) = \frac{r}{r+1} (1 - \zeta^{\alpha}) x^{(r+1)/r}$ and $S_{r,\alpha} = (-1)^{\frac{\alpha-r+1}{2}}$ are the instanton actions and the Stokes constants associated to the $r$-Airy functions;

		\item the minor $\widehat{W}^{(\pm \alpha,i)}_{n}$ is the Borel transform of the formal power series
		\begin{equation}\label{eqn:1inst:corr:rspin}
			W_{n}^{(\pm \alpha,i)}( \bm{x}; \hbar )
			\coloneqq
			(-1)^{n-1}
			\sum_{\sigma \in S_n^{\textup{cyc}}}
				K_{\pm\alpha,\mp}(x_i,x_{\sigma(i)};\hbar)
				\prod_{j \ne i} K_{\pm,\mp}(x_j,x_{\sigma(j)};\hbar) \,,
		\end{equation}
		where the kernels $K_{\pm \alpha,\mp}$ are given by
		\begin{equation}
		\begin{aligned}
			K_{+\alpha,-}(x,y;\hbar)
			& \coloneqq
			\frac{
				\sum_{m=0}^{r-1} \widetilde{\psi}_{\alpha}^{(m)}(x;\hbar) \, \widetilde{\phi}_{r}^{(r-1-m)}(y;\hbar)
			}{x-y} \,, \\
			K_{-\alpha,+}(x,y;\hbar)
			& \coloneqq
			- \frac{
				\sum_{m=0}^{r-1} \widetilde{\phi}_{\alpha}^{(m)}(x;\hbar) \, \widetilde{\psi}_{r}^{(r-1-m)}(y;\hbar)
			}{x-y} \,.
		\end{aligned}
		\end{equation}
	\end{itemize}
	Moreover, the kernels are related by the parity relation $K_{-\alpha,-}(x,y;\hbar) = - K_{+\alpha,+}(x,y;-\hbar)$. Hence, the analogous relation holds for the correlators: $W^{(-\alpha,i)}_{n}( \bm{x}; \hbar ) = (-1)^n \, W^{(+\alpha,i)}_{n}( \bm{x}; -\hbar )$.
\end{proposition}

Again, since $W^{(\pm\alpha, i)}_{n}$ are related by the parity relation, we can simply use $W_n^{(+\alpha,i)}$ and drop the `$+$' symbol from the superscript. From the analysis of the singularity structure of $W_n$ in the Borel plane, we find the large genus asymptotics of its coefficients through the Borel transform method.

\begin{proposition}[{Large genus asymptotics of the $r$-Airy correlators}]\label{prop:large:g:corr:rspin}
	The large genus asymptotics of the expansion coefficients of the $n$-point $r$-Airy correlators is given by
	\begin{equation}\label{eqn:large:g:corr:rspin}
		W_{g,n}(\bm{x})
		=
		\sum_{\alpha=1}^{r-1} \frac{S_{r,\alpha}}{\pi} \sum_{i=1}^n
			\frac{\Gamma(2g-2+n)}{A_{r,\alpha}(x_i)^{2g-2+n}}
			\left(
				\sum_{k = 0}^K
				\frac{A_{r,\alpha}(x_i)^{k}}{(2g-3+n)^{\underline{k}}} W^{(\alpha,i)}_{k,n}(\bm{x})
				+\bigO\biggl( \frac{1}{g^{K+1}} \biggr)
			\right) \,.
	\end{equation}
	Moreover, the leading term is explicitly given by
	\begin{multline}\label{eqn:large:g:corr:leading:rspin}
		W_{g,n}(\bm{x})
		=
		\frac{(-1)^{g-1+n}}{\pi}
		\frac{ \Gamma(2g-2+n) }{ \left( \frac{r}{r+1} \right)^{2g-2+n} }
		\frac{2^n}{r^n x_1 \cdots x_n} \\
		\times \Bigg[
		\sum_{\alpha = 1}^{\floor{\frac{r-1}{2}}}
			\frac{(-1)^{\alpha n}}{\left( 2\sin(\frac{\alpha}{r} \pi) \right)^{2g-1+n}}
			\Bigl(
				h_{(r+1)(2g-2+n)}^{(r,\alpha)}(\bm{x}^{-1/r}) + \bigO\bigl( g^{-1} \bigr)
			\Bigr) \\
			+
			\frac{\delta_{r}^{\textup{even}}}{2} \,
			\frac{(-1)^{\frac{r}{2}n}}{2^{2g-1+n}}
			\Bigl(
				h_{(r+1)(2g-2+n)}^{(r,\frac{r}{2})}(\bm{x}^{-1/r}) + \bigO\bigl( g^{-1} \bigr)
			\Bigr) \Bigg] \,.
	\end{multline}
	Here $\delta_{r}^{\textup{even}}$ gives one if $r$ is even and zero otherwise, and $h_{D}^{(r,\alpha)}(\bm{u})$ denotes the polynomial
	\begin{equation}\label{eqn:h:rspin}
		h_{D}^{(r,\alpha)}(\bm{u})
		\coloneqq
		\sum_{ \substack{k_1,\dots,k_n \ge 0 \\ |k| = D} }
		\prod_{i=1}^n
			\sin\Bigl( \frac{\alpha k_i}{r}\pi \Bigr) u_i^{k_i}
		=
		\sum_{ \substack{d_1,\dots,d_n \ge 0 \\ a_1,\dots,a_n \in \{1,\dots,r-1\} \\ r|d|+|a| = D} }
		(-1)^{\alpha|d|}
		\prod_{i=1}^n
			\sin\left( \frac{\alpha a_i}{r}\pi \right) u_i^{r d_i+a_i} \,.
	\end{equation}
\end{proposition}

\begin{proof}
	\Cref{eqn:large:g:corr:rspin} follows from \cref{thm:large:order} and the parity relation. For the leading term, notice that $\widetilde{\psi}_r^{(m)}(x;\hbar) = \frac{1}{\sqrt{r}} x^{-\frac{r-1-2m}{2r}} + \bigO(\hbar)$ and $\widetilde{\phi}_r^{(r-1-m)}(y;\hbar) = \frac{1}{\sqrt{r}} y^{\frac{r-1-2m}{2r}} + \bigO(\hbar)$. Thus,
	\[
		K_{+,-}(x,y;\hbar)
		=
		\frac{1}{r} \frac{1}{x-y} \left( \frac{x}{y} \right)^{\frac{1-r}{2m}}
		\sum_{m=0}^{r-1} \left( \frac{x}{y} \right)^{\frac{m}{r}}
		+ \bigO(\hbar)
		=
		\frac{1}{r} \frac{(xy)^{\frac{1-r}{2r}}}{x^{1/r} - y^{1/r}}
		+ \bigO(\hbar) \,.
	\]
	A similar computation shows that $K_{+\alpha,-}(x,y;\hbar)
		=
		- \frac{(-1)^{\frac{r-\alpha}{2}}}{r}
		\frac{(xy)^{\frac{1-r}{2r}}}{\zeta^{\alpha/2} x^{1/r} - \zeta^{-\alpha/2} y^{1/r}}
		+ \bigO(\hbar)$. Inserting the above expressions in the definition of $W^{(\alpha,i)}_{n}$, we find
	\[
	\begin{split}
		W^{(\alpha,i)}_{0,n}(\bm{x})
		&=
		(-1)^{n-1} \frac{- (-1)^{\frac{r-\alpha}{2}}}{r^n (x_1 \cdots x_n)^{\frac{r-1}{r}}}
		\sum_{\sigma \in S_n^{\textup{cyc}}}
			\frac{1}{\zeta^{\alpha/2} x_i^{1/r} - \zeta^{-\alpha/2} x_{\sigma(i)}^{1/r}}
			\prod_{j \ne i} \frac{1}{ x_j^{1/r} - x_{\sigma(j)}^{1/r}} \\
		&=
		(-1)^{n-1} \frac{(-1)^{\frac{r-\alpha}{2}}}{r^n (x_1 \cdots x_n)^{\frac{r-1}{r}}}
		\frac{ \zeta^{\frac{\alpha}{2}} (1-\zeta^{\alpha})^{n-2} \, x_i^{\frac{n-2}{r}} }{ \prod_{j \ne i} (x_i^{1/r} - x_j^{1/r})(\zeta^{\alpha} x_i^{1/r} - x_j^{1/r}) } \,.
	\end{split}
	\]
	In the second line, we applied \cref{tech:lemma} to simplify the sum over permutations. Recalling the values of $S_{r,\alpha} = (-1)^{\frac{\alpha-r+1}{2}}$ and $A_{r,\alpha}(x) = \frac{r}{r+1} (1 - \zeta^{\alpha}) x^{(r+1)/r}$, we find
	{\allowdisplaybreaks
	\begin{align*}   
		W_{g,n}(\bm{x})
		& =
		\frac{(-1)^{n}}{\pi\iu \, r^n (x_1 \cdots x_n)^{\frac{r-1}{r}}}
		\frac{\Gamma(2g-2+n)}{\left( \frac{r}{r+1} \right)^{2g-2+n}}
		\sum_{\substack{\alpha = 1,\dots,r-1 \\ i = 1,\dots,n}}
		\left(
			\frac{
				\zeta^{\frac{\alpha}{2}} (1-\zeta^{\alpha})^{-2g} \,
				x_i^{-\frac{(r+1)(2g-2+n)+n-2}{r}}
				}{
				\prod_{j \ne i} (x_i^{1/r} - x_j^{1/r})(\zeta^{\alpha} x_i^{1/r} - x_j^{1/r})
				}
			+
			\bigO\bigl( g^{-1} \bigr)
		\right) \\
		& =
		\frac{(-1)^n}{\pi \iu \, r^n (x_1 \cdots x_n)^{\frac{r-1}{r}}}
		\frac{\Gamma(2g-2+n)}{\left( \frac{r}{r+1} \right)^{2g-2+n}}
		\sum_{i=1}^n
		\Bigg(
		\sum_{\alpha = 1}^{\floor{\frac{r-1}{2}}}
		\Bigg(
			\frac{
				\zeta^{\frac{\alpha}{2}} (1-\zeta^{\alpha})^{-2g}
				}{
				\prod_{j \ne i} (x_i^{1/r} - x_j^{1/r})(\zeta^{\alpha} x_i^{1/r} - x_j^{1/r})
				} \\
		&\qquad\qquad\qquad\qquad
			-
			\frac{
				\zeta^{-\frac{\alpha}{2}} (1-\zeta^{-\alpha})^{-2g}
				}{
				\prod_{j \ne i} (x_i^{1/r} - x_j^{1/r})(\zeta^{-\alpha} x_i^{1/r} - x_j^{1/r})
				}
		\Bigg) x_i^{-\frac{(r+1)(2g-2+n)+n-2}{r}} \\
		&\qquad\qquad\qquad\qquad
		+
		\delta_{r}^{\textup{even}}
		\frac{
			\iu \, 2^{-2g}
			}{
			\prod_{j \ne i} (x_i^{1/r} - x_j^{1/r})(- x_i^{1/r} - x_j^{1/r})
		} x_i^{-\frac{(r+1)(2g-2+n)+n-2}{r}}
		+
		\bigO\bigl( g^{-1} \bigr)
		\Bigg) \,.
	\end{align*}
	}
	In the second line, we have re-arranged the sum as
	\[
		\sum_{\alpha=1}^{r-1} w_{\alpha}
		=
		\sum_{\alpha=1}^{\floor{\frac{r-1}{2}}} ( w_{\alpha} + w_{r-\alpha})
		+
		\delta_{r}^{\textup{even}} \, w_{\frac{r}{2}} \,.
	\]
	The reason behind it is that the singularities corresponding to $(1-\zeta^{\alpha})$ and $(1-\zeta^{r-\alpha})$ are complex conjugate, hence at the same distance from the origin (cf. \cref{fig:actions:rspin}). It is natural to combine them, since they contribute to the same leading asymptotic. We can now perform the sum over $i$ employing \cref{lemma:poly} to get the thesis.
\end{proof}

As a consequence, we obtain the following asymptotic formula for $r$-spin intersection numbers, expressing all $\bigO(1)$ terms for each exponential contribution.

\begin{corollary}[{Leading large genus asymptotics for $r$-spin intersection numbers}] \label{cor:large:g:rspin}
	For any given $n \ge 1$ and $a_1,\dots,a_n \in \set{1,\dots,r-1}$, uniformly in $d_1,\dots,d_n$ as $g \to \infty$:
	\begin{multline}
		\braket{\tau_{d_1,a_1} \cdots \tau_{d_n,a_n}}^{r\textup{-spin}}
		\prod_{i=1}^n (rd_i + a_i)!_{(r)} \\
		=
		\frac{2^n}{2\pi} \, \frac{\Gamma(2g-2+n)}{r^{g-1-|d|}}
		\Bigg[
			\sum_{\alpha = 1}^{\floor{\frac{r-1}{2}}}
				\frac{
					(-1)^{(\alpha-1)(|d|+n)}
				}{
					\left( \frac{2r}{r+1} \sin(\frac{\alpha}{r}\pi) \right)^{2g-2+n}
				}
				\Biggl(
					\frac{\prod_{i=1}^n
						\sin(\frac{\alpha a_i}{r}\pi)}{\sin(\frac{\alpha}{r}\pi)}
					+
					\bigO\bigl( g^{-1} \bigr)
				\Biggr) \\
			+
			\frac{\delta_{r}^{\textup{even}}}{2} \,
			\frac{
				(-1)^{(\frac{r}{2} - 1)(|d|+n)}
			}{
				\left( \frac{2r}{r+1} \right)^{2g-2+n}
			}
			\Biggl(
				\prod_{i=1}^n \sin\left( \tfrac{a_i}{2}\pi \right)
				+
				\bigO\bigl( g^{-1} \bigr)
			\Biggr)
		\Bigg] \,.
	\end{multline}
\end{corollary}

\begin{remark}
	At a practical level, only the term corresponding to $\alpha = 1$ contributes in the asymptotic formula (for $r > 2$), since all other values of $\alpha$ give exponentially suppressed terms. In this case, the formula simplifies to
	\begin{multline}
		\braket{\tau_{d_1,a_1} \cdots \tau_{d_n,a_n}}^{r\textup{-spin}}
		\prod_{i=1}^n (rd_i + a_i)!_{(r)} \\
		=
		\frac{2^n}{2\pi \sin(\frac{\pi}{r})} \,
		\frac{
			\Gamma(2g-2+n)
		}{
			r^{g-1-|d|} \, \left( \frac{2r}{r+1} \sin(\frac{\pi}{r}) \right)^{2g-2+n}
		}
		\Biggl(
			\prod_{i=1}^n
				\sin\left( \tfrac{a_i}{r}\pi \right)
			+
			\bigO\bigl( g^{-1} \bigr)
		\Biggr) \,.
	\end{multline}
	Specialised to $n = 1$, the above asymptotic formula retrieves the one proved by Dubrovin--Yang--Zagier in \cite[theorem 5 (vii)]{DYZ}. We also emphasise that, despite being exponentially suppressed, the terms corresponding to $\alpha > 1$ can be detected numerically as explained in \cref{app:exp}.
\end{remark}

\subsubsection{Subleading contributions}
Similarly to the $r = 2$ case, the minors exhibit a distinct pole structure and polynomiality structure. We provide a brief outline of the proof, as it resembles the $r = 2$ case with a few minor variations that we outline below.

\begin{proposition}\label{prop:minors:rspin}
	For any $k \ge 0$, $n \ge 2$, and $\alpha = 1,\dots, r-1 $:
	\begin{multline}\label{eqn:minors:rspin}
		W_{k,n}^{(\alpha,i)}(\bm{x})
		=
		\frac{(-1)^{n-1+\frac{r-\alpha}{2}}}{r^n (x_1 \cdots x_n)^{\frac{r-1}{r}}}
		\frac{
			\zeta^{\alpha/2} (1-\zeta^{\alpha})^{n-2} \, x_i^{\frac{n-2-(r+1)k}{r}}
		}{
			\prod_{j \ne i} (x_i^{1/r} - x_j^{1/r})(\zeta^{\alpha}x_i^{1/r} - x_j^{1/r})
		} \\
		\times
		R_{k,n-1}^{(\alpha)}\left(
			\Big( \frac{x_i}{x_1} \Big)^{1/r} , \dots,
			\widehat{ \Big( \frac{x_i}{x_i} \Big)^{1/r} } , \dots,
			\Big( \frac{x_i}{x_n} \Big)^{1/r}
		\right),
	\end{multline}
	where $R_{k,n-1}^{(\alpha)}$ is a symmetric polynomial of degree $\le \min\set{ (r+1)k + n - 2, 2((r+1)k - 1)}$ and degree $\le (r+1)k$ in each individual variable satisfying the $\Z/r\Z$-symmetry
	\begin{equation}
		(-\zeta^{\alpha})^k \cdot R_{k,n}^{(\alpha)}(\zeta^{-\alpha}u_1, \dots, \zeta^{-\alpha}u_n)
		=
		R_{k,n}^{(r-\alpha)}(u_1, \dots, u_n) \,.
	\end{equation}
	The sequence of polynomials $(R_{k,n}^{(\alpha)})_{n \ge 1}$ satisfies the following specialisation properties:
	\begin{equation}\label{eqn:spec:rspin}
	\begin{aligned}
		& R_{k,n+1}^{(\alpha)}(u_1,\dots, u_n, u_{n+1}) \big|_{u_{n+1} = 1}
		=
		R_{k,n}^{(\alpha)}(u_1,\dots,u_n) \,, \\
		& R_{k,n+1}^{(\alpha)}(u_1,\dots, u_n, u_{n+1}) \big|_{u_{n+1}
		=
		\zeta^{-\alpha}} = R_{k,n}^{(\alpha)}(u_1,\dots,u_n) \,.
	\end{aligned}
	\end{equation}
	In turn, the degree condition and the specialisation properties uniquely determine $R_{k,n}^{(\alpha)}$ for all $n$ by an explicit algorithm from the first terms of the sequence\footnote{
		Experimentally, we can see that the degree of $R_{k,n}^{(\alpha)}$ is actually $\le \min\set{ (r+1)k + n - 1, (r+2)k}$. As in the Witten--Kontsevich case, this lower bound would improve the efficiency of the algorithm.
	}. Moreover, the coefficient of $R_{k,n}^{(\alpha)}$ in the monomial basis is a polynomial of $n$.
\end{proposition}

\begin{proof}[{Sketch of the proof}]
	The Wronskian condition appearing in the residue computation of the kernels (cf. \cref{lemma:poles:minors}) is no longer the key fact. This is because, for $\Psi \in \SL(2,\C)$, matrix elements of $\Psi^{-1}$ are linear in the matrix elements of $\Psi$. However, this linearity does not hold for elements of $\SL(r,\C)$ with $r > 2$. Nonetheless, we can employ the relation $\Phi^t \Psi = \Id$, which indeed yields quadratic relations that generalise the Wronskian condition. This, together with homogeneity arguments, yields \cref{eqn:minors:rspin} and the specialisation properties. From the specialisation properties, one can prove that the degree of $R_{k,n}^{(\alpha)}$ is eventually independent of $n$, and more precisely
	\[
		\deg{( R_{k,n}^{(\alpha)} )} \le \min\Set{ (r+1)k + n - 1, 2(r+1)k - 2 } \,.
	\]
	Thus, from the degree bound and the specialisation properties, we get an algorithm for computing $(R_{k,n}^{(\alpha)})_{n \ge 1}$ from the first few values of $n$, together with the polynomiality statement for the expansion coefficients in the monomial basis. As for the $\Z/r\Z$-symmetry, which substitutes evenness for the corresponding polynomials in the $r=2$ case, it follows from analogous relations between the wave functions.
\end{proof}

As a consequence of the above result, one can prove that the subleading corrections to $W_{g,n}$ exhibit the following structure. Firstly, let us group together subleading corrections in accordance with the absolute value of the instanton actions. In other words, let us group together terms corresponding to $\alpha$ and $r-\alpha$:
\begin{multline}
	W_{g,n}(\bm{x})
	=
	\frac{(-1)^n}{\pi} \Gamma(2g-2+n)
	\Bigg[
		\sum_{\alpha=1}^{\floor{\frac{r-1}{2}}}
		\Bigg(
		\sum_{k=0}^K
			\frac{1}{(2g-3+n)^{\underline{k}}}
			\frac{V_{k,g,n}^{(\alpha)}(\bm{x})}{x_1 \cdots x_n}
			+
			\bigO\biggl( \frac{1}{|A_{r,\alpha}(x_i)|^{2g} g^{K+1}} \biggr)
		\Bigg) \\
	+
	\frac{\delta_{r}^{\textup{even}}}{2}
	\Bigg(
		\sum_{k=0}^K
			\frac{1}{(2g-3+n)^{\underline{k}}}
			\frac{V_{k,g,n}^{(\frac{r}{2})}(\bm{x})}{x_1 \cdots x_n}
		+
		\bigO\biggl( \frac{1}{|A_{r,\frac{r}{2}}(x_i)|^{2g} g^{K+1}} \biggr)
	\Bigg)
	\Bigg],
\end{multline}
where we have set
\begin{equation}
	V_{k,g,n}^{(\alpha)}(\bm{x})
	\coloneqq
	(-1)^n \, x_1 \cdots x_n
	\sum_{i=1}^n
	\left(
		S_{r,\alpha} \frac{W_{k,n}^{(\alpha,i)}(\bm{x})}{A_{r,\alpha}(x_i)^{2g-2+n-k}} 
		+
		S_{r,r-\alpha} \frac{W_{k,n}^{(r-\alpha,i)}(\bm{x})}{A_{r,r-\alpha}(x_i)^{2g-2+n-k}} 
	\right) \,.
\end{equation}
With the residue strategy of \cref{lemma:poly:subleading,lemma:poly}, but taking $f_{\alpha}(q) = \frac{1}{\pi\iu} \frac{\zeta^{\alpha/2} (1-\zeta^{\alpha})^{k+1-2g}}{\prod_{i=1}^n (q-z_i)(\zeta^{\alpha}q-z_i)} R_{k,n}^{(\alpha)}(q \bm{z}^{-1})$, we find
\begin{multline}
	V_{k,g,n}^{(\alpha)}(\bm{x})
	=
	(-1)^{g-1+\alpha n}
	\frac{2^n}{2 \sin(\frac{\alpha}{r}\pi)}
	\frac{1}{r^n \left( \frac{2r}{r+1} \sin(\frac{\alpha}{r}\pi) \right)^{2g-2+n-k}} \\
	\sum_{d= 0}^{\min{\set{ (r+1)k + n - 1, 2(r+1)k - 2 }} }
		\bar{R}_{k,n}^{(\alpha,d)}(\bm{x}^{-1/r}) \,
		h^{(r,\alpha)}_{(r+1)(2g-2+n)-d}(\bm{x}^{-1/r})\,,
\end{multline}
where $h^{(r,\alpha)}_{D}$ is the polynomial defined in \cref{eqn:h:rspin}, $\bar{R}_{k,n}^{(\alpha,d)}$ is the homogeneous component of $R_{k,n}^{(\alpha)}$ of degree $d$, normalised according to the following formula:
\begin{equation}
	R_{k,n}^{(\alpha)}(q\bm{u})
	=
	(-1)^{\frac{k}{2}}
	\sum_{d \ge 0} q^d \, \zeta^{\frac{\alpha (d-k)}{2}} \bar{R}_{k,n}^{(\alpha,d)}(\bm{u}) \,.
\end{equation}
Except for having to group together instanton actions according to their distance from the origin, we highlight two more differences between the $r = 2$ case and the general $r$ case. Firstly, the $\Z/r\Z$-symmetry, which substitutes the evenness property, is a much weaker constraint. Thus, some simplifications no longer occur. Secondly, the complete homogeneous polynomial is essentially substituted by the polynomials $h^{(r,\alpha)}_D$. While in the complete homogeneous polynomial $h_D$ all monomials of degree $D$ appear with coefficient $1$, in $h^{(r,\alpha)}_D$ all monomials appear with coefficient $\sin(\frac{\alpha k}{r}\pi)$, where $k$ is the exponent of the corresponding variable. The latter depends on the (parity of the) quotient and on the reminder of the Euclidean division of $k$ by $r$: writing $k = r[k] + \braket{k}$ for $[k] \ge 0$ and $\braket{k} \in \set{ 0,\dots,r-1 }$,
\begin{equation}
	\sin\left( \frac{\alpha k}{r}\pi \right)
	=
	(-1)^{\alpha [k]} \sin\left( \frac{\alpha \! \braket{k}}{r}\pi \right) \,.
\end{equation}
This last difference complicates the extraction of coefficients from products of the form $m_{\nu} h^{(r,\alpha)}_{D}$. Nonetheless, the strategy employed in \cref{app:symmetric:fncts} can still be adapted to obtain the following formula. For all partitions $\mu = (\mu_1,\dots,\mu_n)$ and $\nu = (\nu_1,\dots,\nu_n)$ (zero parts are allowed):
\begin{equation}
	\bigl[ u_1^{\mu_1} \cdots u_n^{\mu_n} \bigr] m_{\nu} h^{(r,\alpha)}_{|\mu| - |\nu|}
	=
	(-1)^{\alpha ([\mu] - [\nu])}
	\sum_{ A \in \{-(r-1),\dots,r-1\}^n }
		M_{n,\mu,\nu}^{(r,\alpha)}(A) \,
		\prod_{i=1}^n \sin\left( \frac{\alpha A_i}{r}\pi \right) \,.
\end{equation}
Here $[\mu] = \sum_{i=1}^n [\mu_i]$ is the sum of the quotients of the Euclidean division of the parts $\mu_i$ by $r$, the quantity $M_{n,\mu,\nu}^{(r,\alpha)}(A)$ is the explicit polynomial in the multiplicities $p_{m}(\mu)$ given by
\begin{equation}\label{eqn:M:mu:nu:rspin}
	M_{n,\mu,\nu}^{(r,\alpha)}
	\coloneqq
	\frac{1}{z_{\nu}}
	\prod_{k=0}^{\nu_1} \prod_{s = -(r-1)}^{r-1}
	\left(
		n -
		\sum_{i=0}^{k-1} p_i(\mu)
		-
		\sum_{\substack{ i=k, \dots, \mu_1 \\ i \not\equiv s-k \pmod{r}}} p_{i}(\mu)
		-
		\sum_{j = k+1}^{\nu_1} p_{j,s-k-j}(\nu,A)
	\right)^{\underline{p_{k,s}(\nu,A)}} \,,
\end{equation}
and $p_{k,s}(\nu,A) \coloneqq \#\set{ i | \nu_i = k \,\wedge\, A_i = s }$ is the joint multiplicity of $(k,s)$ in $(\nu,A)$.

As a consequence, we obtain the following asymptotic formula for the $r$-spin intersection numbers. We omit the proof, which is parallel to that of \cref{thm:large:g:WK}.

\begin{theorem}[{Large genus asymptotics for $r$-spin intersection numbers}] \label{thm:large:g:rspin}
	For any given $n \ge 1$ and $a_1,\dots,a_n \in \set{1,\dots,r-1}$, uniformly in $d_1,\dots,d_n$ as $g \to \infty$:
	\begin{multline}
		\braket{\tau_{d_1,a_1} \cdots \tau_{d_n,a_n}}^{r\textup{-spin}}
		\prod_{i=1}^n (rd_i + a_i)!_{(r)} \\
		=
		\frac{2^n}{2\pi} \, \frac{\Gamma(2g-2+n)}{r^{g-1-|d|}}
		\Bigg[
			\sum_{\alpha = 1}^{\floor{\frac{r-1}{2}}}
				\frac{
					1
				}{
					\left( \frac{2r}{r+1} \sin(\frac{\alpha}{r}\pi) \right)^{2g-2+n}
				}
				\Biggl(
					\sum_{k=0}^K
						\frac{ ( \frac{2r}{r+1} \sin(\frac{\alpha}{r}\pi) )^{k} }{(2g-3+n)^{\underline{k}}} \,
						\gamma^{(r,\alpha)}_{k}
					+
					\bigO\biggl( \frac{1}{g^{K+1}} \biggr)
				\Biggr) \\
			+
			\frac{\delta_{r}^{\textup{even}}}{2} \,
			\frac{
				1
			}{
				\left( \frac{2r}{r+1} \right)^{2g-2+n}
			}
			\Biggl(
				\sum_{k=0}^K
					\frac{ ( \frac{2r}{r+1} )^{k} }{(2g-3+n)^{\underline{k}}} \,
						\gamma^{(r,\frac{r}{2})}_{k}
					+
					\bigO\biggl( \frac{1}{g^{K+1}} \biggr)
			\Biggr)
		\Bigg] \,.
	\end{multline}
	Each $\gamma^{(r,\alpha)}_{k}$ is a function of $n$ and the multiplicities $p_{m} = \#\set{i | rd_i + a_i = m}$. Moreover, it depends polynomially in the multiplicities, up to an overall factor of $(-1)^{\alpha |d|}$.
\end{theorem}

\begin{remark}
	Contrary to the case $r = 2$, it is not clear whether the subleading corrections $\gamma^{(r,\alpha)}_{k}$ have a polynomiality behaviour in $n$. Moreover, from \cref{eqn:M:mu:nu:rspin}, there seems to be no bound on the number of multiplicities appearing in $\gamma^{(r,\alpha)}_{k}$ (unless some vanishing condition would hold). Nonetheless, our approach provides an algorithm to compute the subleading corrections for any value of $k$.
\end{remark}


\appendix
\section*{Appendices}
\renewcommand{\thesubsection}{\Alph{subsection}}
\counterwithin{theorem}{subsection}     
\numberwithin{equation}{subsection}     

\subsection{Some technical lemmas}
\begin{lemma}\label{tech:lemma}
	Let $n$ be a positive integer and let $z_1, \dots, z_n$ be formal variables. Let $\tau$ be any function. Then:
	\begin{equation}
		\sum_{\sigma \in S_n^{\textup{cyc}}}
		\frac{1}{z_{\sigma(1)} - \tau(z_{1})} \prod_{i=2}^n \frac{1}{z_{i} - z_{\sigma(i)}}
		=
	 \frac{(z_1 - \tau(z_1))^{n-2} }{\prod_{i=2}^{n}(z_1 - z_i)(\tau(z_1) - z_i)} \,.
	\end{equation}
\end{lemma}

\begin{proof}
	We prove this by induction on $n$. The base case is easily checked. Call the left-hand side of the statement $P_n(\bm{z})$. The statement is equivalent to proving that
	\begin{equation*}
		P_{n+1}(\bm{z}) = P_n(\bm{z}) \frac{z_1 - \tau(z_1)}{(z_1 - z_{n+1})(\tau(z_1) - z_{n+1})}.
	\end{equation*}
	Observe that $P_{n+1}(\bm{z})$ can also be written as
	\begin{multline*}
		\sum_{\sigma \in S_n^{\textup{cyc}}} \Biggl(
			\frac{1}{(z_{n+1} - \tau(z_1)) (z_{n+1} - z_{\sigma(1)})}
			\prod_{i=2}^{n} \frac{1}{z_{i} - z_{\sigma(i)}} \\
		+
		\sum_{j=2}^n
			\frac{1}{(z_{\sigma(1)} - \tau(z_1)) (z_j - z_{n+1}) (z_{n+1} - z_{\sigma(j)})}
			\prod_{\substack{i=2,\dots,n \\ i \ne j}} \frac{1}{z_i - z_{\sigma(i)}} \Biggr),
	\end{multline*}
	since every cyclic permutation of $n+1$ elements can be obtained from a cyclic permutation of $n$ elements by inserting the $(n+1)$-st element in all possible positions. Factoring out $\frac{1}{z_{\sigma(1)} - \tau(z_1)} \prod_{i=2}^{n} \frac{1}{z_{i} - z_{\sigma(i)}}$, we find
	\begin{multline*}
		\sum_{\sigma \in S_n^{\textup{cyc}}} 
		\frac{1}{z_{\sigma(1)} - \tau(z_1)} 
		\prod_{i=2}^{n} \frac{1}{z_{i} - z_{\sigma(i)}}
			\left(
				\frac{z_{\sigma(1)} - \tau(z_1)}{(z_{n+1} - z_{\sigma(1)}) (z_{n+1} - \tau(z_1))}
				+ \sum_{j=2}^{n}
					\frac{z_j - z_{\sigma(j)}}{(z_{n+1} - z_{\sigma(j)}) (z_j - z_{n+1})}
		\right) \\
		=
		\sum_{\sigma \in S_n^{\textup{cyc}}} 
		\frac{1}{z_{\sigma(1)} - \tau(z_1)} 
		\prod_{i=2}^{n} \frac{1}{z_{i} - z_{\sigma(i)}}
			\left(
				\frac{1}{z_{n+1} - z_{\sigma(1)}} - \frac{1}{z_{n+1} - \tau(z_1)}
				+ \sum_{j=2}^{n}
				\left(
					\frac{1}{z_{n+1} - z_{\sigma(j)}} - \frac{1}{z_{n+1} - z_j}
				\right)
		\right).
	\end{multline*}
	Finally, notice that
	\begin{equation*}
		\frac{1}{z_{n+1} - z_{\sigma(1)}} - \frac{1}{z_{n+1} - \tau(z_1)}
			+ \sum_{j=2}^{n}
			\left(
				\frac{1}{z_{n+1} - z_{\sigma(j)}} - \frac{1}{z_{n+1} - z_j}
			\right)
		=
		\frac{z_1 - \tau(z_1)}{(z_1 - z_{n+1})(\tau(z_1) - z_{n+1})}
	\end{equation*}
	does not depend on $\sigma$ and can therefore be brought outside the sum over permutations, providing the wanted prefactor. This concludes the proof of the lemma.
\end{proof}

\begin{lemma}\label{lemma:poly}
	Let $n$ be a positive integer and let $z_1, \dots, z_n$ be formal variables. Set $\zeta = e^{\frac{2\pi\iu}{r}}$ for $r \ge 2$. Then for every $\alpha = 1,\dots, r-1$:
	\begin{multline}
		\frac{1}{\pi\iu}
		\sum_{i=1}^n \left(
			\frac{\zeta^{\alpha/2} (1-\zeta^{\alpha})^{-2g} }{\prod_{j \ne i} (z_i - z_j)(\zeta^{\alpha} z_i - z_j) }
			- \frac{\zeta^{-\alpha/2} (1-\zeta^{-\alpha})^{-2g} }{\prod_{j \ne i} (z_i - z_j)(\zeta^{-\alpha} z_i - z_j) }
		\right) z_i^{-(r+1)(2g-2+n)+n-2} \\
		=
		(-1)^{g-1+\alpha n}
		\frac{2^n}{2\pi \sin(\frac{\alpha}{r} \pi)}
		\frac{1}{\left( 2 \sin(\frac{\alpha}{r} \pi) \right)^{2g-2+n}}
		\sum_{ \substack{k_1,\dots,k_n \ge 0 \\ |k| = (r+1)(2g-2+n)}}
			\prod_{i=1}^n \frac{\sin(\frac{\alpha k_i}{r}\pi)}{z_i^{k_i+1}} \,.
	\end{multline}
\end{lemma}

\begin{proof}
	Consider the function $f_{\alpha}(q) \coloneqq \frac{1}{\pi\iu} \frac{\zeta^{\alpha/2} (1-\zeta^{\alpha})^{1-2g}}{\prod_{i=1}^n (q - z_i)(\zeta^{\alpha} q - z_i)}$. We omit its dependence on $z_1, \dots, z_n$ for simplicity. As a function of $q$, it has simple poles at $q = z_i$ and $q = \zeta^{-\alpha} z_i$ for $i = 1, \dots, n$ and no other poles. Moreover, its residues are computed as
	\begin{align*}
		\Res_{q = z_i} f_{\alpha}(q)
		& =
		- \frac{1}{\pi\iu}
		\frac{\zeta^{\alpha/2} (1-\zeta^{\alpha})^{1-2g}}{(1-\zeta^{\alpha}) \prod_{j \ne i}(z_i - z_j)(\zeta^{\alpha} z_i - z_j)}
		z_i^{-1} \,, \\
		\Res_{q = \zeta^{-\alpha} z_i} f_{\alpha}(q)
		& =
		- \frac{1}{\pi\iu}
		\frac{\zeta^{-\alpha/2} (1-\zeta^{\alpha})^{1-2g}}{(1-\zeta^{-\alpha}) \prod_{j \ne i}(z_i - z_j)(\zeta^{-\alpha} z_i - z_j)}
		z_i^{-1} \,.
	\end{align*}
	Thus we find
	\[
	\begin{split}
		f_{\alpha}(q)
		& =
		\frac{1}{\pi\iu} \sum_{i=1}^n \left(
			\frac{1}{z_i-q}
				\frac{\zeta^{\alpha/2} (1-\zeta^{\alpha})^{1-2g}}{(1-\zeta^{\alpha}) \prod_{j \ne i}(z_i - z_j)(\zeta^{\alpha} z_i - z_j) }
		\right. \\
		& \qquad \qquad \qquad \left.
			+
			\frac{1}{\zeta^{-\alpha}z_i-q}
				\frac{\zeta^{-\alpha/2} (1-\zeta^{\alpha})^{1-2g}}{(1-\zeta^{-\alpha}) \prod_{j \ne i}(z_i - z_j)(\zeta^{-\alpha} z_i - z_j) }
		\right) z_i^{-1} \,.
	\end{split}
	\]
	Let $D \coloneqq (r+1)(2g-2+n)$. Expanding around $q = 0$ and taking the coefficient of $q^{D-n}$, we obtain
	\[
	\begin{split}
		\bigl[ q^{D-n} \bigr] f_{\alpha}(q)
		=
		\frac{1}{\pi\iu}
		\sum_{i=1}^n \left(
			\frac{\zeta^{\alpha/2} (1-\zeta^{\alpha})^{-2g}}{\prod_{j \ne i}(z_i - z_j)(\zeta^{\alpha} z_i - z_j) }
			-
			\frac{\zeta^{-\alpha/2} (1-\zeta^{-\alpha})^{-2g}}{\prod_{j \ne i}(z_i - z_j)(\zeta^{-\alpha} z_i - z_j) }
		\right) z_i^{-D+n-2} \, .
	\end{split}
	\]
	On the other hand, we can compute the expansion of $f_{\alpha}$ around $q = 0$ starting from its definition. Start by considering the expansion
	\[
		\prod_{i=1}^n \frac{1}{(q - z_i)(\zeta^{\alpha} q - z_i)}
		=
		q^{-n} z_1^{-1} \cdots z_n^{-1}
		\sum_{k_1,\dots,k_n \ge 0} \prod_{i=1}^n
			\frac{1-\zeta^{\alpha k_i}}{1-\zeta^{\alpha}} \left( \frac{q}{z_i} \right)^{k_i} \, .
	\]
	Taking the coefficient of $q^{D-n}$, we obtain
	\begin{multline*}
		\bigl[ q^{D-n} \bigr] f_{\alpha}(q)
		=
		\frac{1}{\pi\iu} \zeta^{\alpha/2} (1-\zeta^{\alpha})^{1-2g}
		\sum_{ \substack{k_1,\dots,k_n \ge 0 \\ |k| = D}}
		\prod_{i=1}^n
			\frac{1-\zeta^{\alpha k_i}}{1-\zeta^{\alpha}} \, z_i^{-k_i-1} \\
		=
		- \frac{1}{\pi\iu}
		\frac{1}{(\zeta^{\alpha/2} - \zeta^{-\alpha/2})^{2g-1+n}}
		\sum_{ \substack{k_1,\dots,k_n \ge 0 \\ |k| = D}}
		\zeta^{\alpha\frac{|k|- (2g-2+n)}{2}}
		\prod_{i=1}^n (\zeta^{\alpha k_i/2} - \zeta^{-\alpha k_i/2}) \, z_i^{-k_i-1} \,.
	\end{multline*}
	Notice that $\zeta^{\beta/2} - \zeta^{-\beta/2} = 2 \iu \sin(\frac{\beta}{r}\pi)$. Moreover, the condition $|k| = D =(r+1)(2g-2+n)$ implies that $\zeta^{\alpha \frac{|k|- (2g-2+n)}{2}} = (-1)^{\alpha n}$. Simplifying the remaining powers of $\iu$ yields the thesis.
\end{proof}

\subsection{Symmetric functions and polynomiality properties}
\label{app:symmetric:fncts}
The goal of this appendix is to prove some polynomiality results through the theory of symmetric functions \cite{Mac98}. Throughout this section, we consider partitions $\lambda = (\lambda_1, \dots, \lambda_n)$ consisting of $n$ ordered parts $\lambda_1 \ge \cdots \ge \lambda_n \ge 0$ (notice that we allow empty parts). Define the weight and the length of a partition $\lambda$ as
\begin{equation}
	|\lambda| \coloneqq \sum_{i=1}^n \lambda_i \,,
	\qquad\qquad
	\ell(\lambda) \coloneqq \max\Set{ i | \lambda_i > 0 } \,.
\end{equation}
Define the multiplicities and the automorphism factor as
\begin{equation}
	p_k(\lambda) \coloneqq \#\set{ \lambda_i = k } \,,
	\qquad\qquad
	z_{\lambda} \coloneqq \prod_{k = 0}^{\lambda_1} p_k(\lambda)! \,.
\end{equation}
In this paper, we consider three different bases of symmetric polynomials in the variables $\bm{u} = (u_1, \dots, u_n)$. For any integer $k \ge 0$, the elementary and complete homogeneous symmetric polynomials are defined as
\begin{equation}\label{eqn:elementary:complete}
	e_k(\bm{u}) \coloneqq \sum_{1 \le i_1 < \cdots < i_k \le n} u_{i_1} \cdots u_{i_k} \,,
	\qquad\qquad
	h_k(\bm{u}) \coloneqq \sum_{1 \le i_1 \le \cdots \le i_k \le n} u_{i_1} \cdots u_{i_k} \,,
\end{equation}
respectively. For any partition $\lambda$ as above, set $e_{\lambda} \coloneqq \prod_{i=1}^n e_{\lambda_i}$ and $h_{\lambda} \coloneqq \prod_{i=1}^n h_{\lambda_i}$. We also define the monomial symmetric polynomial as
\begin{equation}\label{eqn:monomial}
	m_{\lambda}(\bm{u})
	\coloneqq
	\frac{1}{z_{\lambda}} \sum_{\sigma \in S_n} \prod_{i=1}^n u_i^{\lambda_{\sigma(i)}} \,.
\end{equation}
Let $\mu = (\mu_1,\dots,\mu_n)$ and $\nu = (\nu_1,\dots,\nu_n)$ be partitions with $|\mu| \ge |\nu|$. The goal of this appendix is to study the coefficients of a generic monomial in the product of a monomial and a complete homogeneous polynomial:
\begin{equation}
	M_{n,\mu,\nu}
	\coloneqq
	\bigl[ u_1^{\mu_1} \cdots u_n^{\mu_n} \bigr] m_{\nu} \, h_{|\mu| - |\nu|} \,.
\end{equation}


\begin{proposition} \label{prop:polynom:monomial:coeff}
	$M_{n,\mu,\nu}$ is a polynomial in $n$ and the multiplicities $p_0(\mu), \dots, p_{\nu_1 - 1}(\mu)$, explicitly given by
	\begin{equation}\label{eqn:M:n:mu:nu}
		M_{n,\mu,\nu}
		=
		\frac{1}{z_{\nu}} \prod_{k=0}^{\nu_1}
		\Biggl(
			n
			-
			\sum_{i=0}^{k-1} p_i(\mu)
			-
			\sum_{j = k+1}^{\nu_1} p_j(\nu)
		\Biggr)^{\!\! \underline{p_k(\nu)}} \,.
	\end{equation}
	Here $x^{\underline{m}} = x (x-1) \cdots (x-m+1)$ denotes falling factorial. Moreover, the degree is equal to $|\nu|$ if we set $\deg{n} = 1$ and $\deg{p_k(\mu)} = k + 1$.
\end{proposition}

\begin{proof}
	Notice that extracting the coefficient of $u_1^{\mu_1} \cdots u_n^{\mu_n}$ means selecting a submonomial $u_1^{\nu_1} \cdots u_n^{\nu_n}$ from $m_{\nu}$ and complete it to $u_1^{\mu_1} \cdots u_n^{\mu_n}$ by extracting the missing powers from $h_{|\mu| - |\nu|}$. From the definition of complete homogeneous and monomial symmetric polynomials, \cref{eqn:elementary:complete,eqn:monomial} respectively, we find
	\[
		M_{n,\mu,\nu} = \frac{1}{z_{\nu}} \sum_{\sigma \in S_n} \prod_{i=1}^n \Theta(\mu_{i} - \nu_{\sigma(i)})  \,,
	\]
	where $\Theta$ is the Heaviside function. We can re-interpret combinatorially the formula above as follows: $M_{n,\mu,\nu}$ is the number of rearrangements $(\nu_{\sigma(1)},\dots,\nu_{\sigma(n)})$ of $\nu$, for $\sigma \in S_n$, such that $\mu_i \geq \nu_{\sigma(i)}$ for all $i = 1,\dots,n$ (up to the automorphism factor $z_{\nu}$). 
	Intuitively speaking, we are counting in how many ways the Young tableau $\mu$ can contain some rearrangement of the Young tableau $\nu$ (where both tableaux have $n$ rows with zero parts allowed). For instance, when $\mu = (4,3,2,1)$ and $\nu = (3,2,1^2)$ there are $8$ such rearrangements, up to the symmetry factor $z_{\nu} = 2!$ that permutes the last two rows of $\nu$. In other words: $M_{n,\mu,\nu} = \frac{8}{2!} = 4$. The $8$ rearrangements are shown below.
	
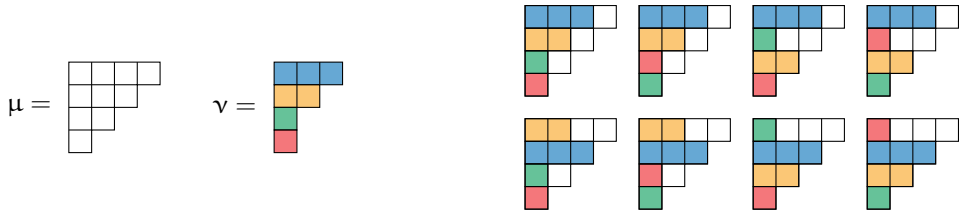
\begin{center}
	\begin{tikzpicture}[scale=.3]
		\draw (0,0) -- (0,4) -- (4,4) -- (4,3) -- (3,3) -- (3,2) -- (2,2) -- (2,1) -- (1,1) -- (1,0) --cycle; 
		\draw (1,4) -- (1,1) -- (0,1);
		\draw (2,4) -- (2,2) -- (0,2);
		\draw (3,4) -- (3,3) -- (0,3);
		\node at (-1.75,2) {$\mu =$};
		\begin{scope}[xshift=9cm]
			\filldraw[fill=NavyBlue, fill opacity=.6] (0,4) -- (3,4) -- (3,3) -- (0,3) --cycle;
			\draw (1,4) -- (1,3);
			\draw (2,4) -- (2,3);
			\filldraw[fill=YellowOrange, fill opacity=.6] (0,3) -- (2,3) -- (2,2) -- (0,2) --cycle;
			\draw (1,3) -- (1,2);
			\filldraw[fill=ForestGreen, fill opacity=.6] (0,2) -- (1,2) -- (1,1) -- (0,1) --cycle; 
			\filldraw[fill=Red, fill opacity=.6] (0,1) -- (1,1) -- (1,0) -- (0,0) --cycle; 
			\node at (-1.75,2) {$\nu =$};
		\end{scope}


		\begin{scope}[xshift= 20cm,yshift=2.5cm]
			\filldraw[fill=NavyBlue, fill opacity=.6] (0,4) -- (3,4) -- (3,3) -- (0,3) --cycle;
			\filldraw[fill=YellowOrange, fill opacity=.6] (0,3) -- (2,3) -- (2,2) -- (0,2) --cycle;
			\filldraw[fill=ForestGreen, fill opacity=.6] (0,2) -- (1,2) -- (1,1) -- (0,1) --cycle; 
			\filldraw[fill=Red, fill opacity=.6] (0,1) -- (1,1) -- (1,0) -- (0,0) --cycle; 

			\draw (0,0) -- (0,4) -- (4,4) -- (4,3) -- (3,3) -- (3,2) -- (2,2) -- (2,1) -- (1,1) -- (1,0) --cycle; 
			\draw (1,4) -- (1,1) -- (0,1);
			\draw (2,4) -- (2,2) -- (0,2);
			\draw (3,4) -- (3,3) -- (0,3);      
		\end{scope}
		\begin{scope}[xshift= 20cm, yshift=-2.5cm]
			\filldraw[fill=NavyBlue, fill opacity=.6] (0,3) -- (3,3) -- (3,2) -- (0,2) --cycle;
			\filldraw[fill=YellowOrange, fill opacity=.6] (0,4) -- (2,4) -- (2,3) -- (0,3) --cycle;
			\filldraw[fill=ForestGreen, fill opacity=.6] (0,2) -- (1,2) -- (1,1) -- (0,1) --cycle; 
			\filldraw[fill=Red, fill opacity=.6] (0,1) -- (1,1) -- (1,0) -- (0,0) --cycle; 

			\draw (0,0) -- (0,4) -- (4,4) -- (4,3) -- (3,3) -- (3,2) -- (2,2) -- (2,1) -- (1,1) -- (1,0) --cycle; 
			\draw (1,4) -- (1,1) -- (0,1);
			\draw (2,4) -- (2,2) -- (0,2);
			\draw (3,4) -- (3,3) -- (0,3);      
		\end{scope}
		\begin{scope}[xshift= 25cm,yshift=2.5cm]
			\filldraw[fill=NavyBlue, fill opacity=.6] (0,4) -- (3,4) -- (3,3) -- (0,3) --cycle;
			\filldraw[fill=YellowOrange, fill opacity=.6] (0,3) -- (2,3) -- (2,2) -- (0,2) --cycle;
			\filldraw[fill=Red, fill opacity=.6] (0,2) -- (1,2) -- (1,1) -- (0,1) --cycle; 
			\filldraw[fill=ForestGreen, fill opacity=.6] (0,1) -- (1,1) -- (1,0) -- (0,0) --cycle; 

			\draw (0,0) -- (0,4) -- (4,4) -- (4,3) -- (3,3) -- (3,2) -- (2,2) -- (2,1) -- (1,1) -- (1,0) --cycle; 
			\draw (1,4) -- (1,1) -- (0,1);
			\draw (2,4) -- (2,2) -- (0,2);
			\draw (3,4) -- (3,3) -- (0,3);      
		\end{scope}
		\begin{scope}[xshift= 25cm, yshift=-2.5cm]
			\filldraw[fill=NavyBlue, fill opacity=.6] (0,3) -- (3,3) -- (3,2) -- (0,2) --cycle;
			\filldraw[fill=YellowOrange, fill opacity=.6] (0,4) -- (2,4) -- (2,3) -- (0,3) --cycle;
			\filldraw[fill=Red, fill opacity=.6] (0,2) -- (1,2) -- (1,1) -- (0,1) --cycle; 
			\filldraw[fill=ForestGreen, fill opacity=.6] (0,1) -- (1,1) -- (1,0) -- (0,0) --cycle; 

			\draw (0,0) -- (0,4) -- (4,4) -- (4,3) -- (3,3) -- (3,2) -- (2,2) -- (2,1) -- (1,1) -- (1,0) --cycle; 
			\draw (1,4) -- (1,1) -- (0,1);
			\draw (2,4) -- (2,2) -- (0,2);
			\draw (3,4) -- (3,3) -- (0,3);      
		\end{scope}
		\begin{scope}[xshift= 30cm,yshift=2.5cm]
			\filldraw[fill=NavyBlue, fill opacity=.6] (0,4) -- (3,4) -- (3,3) -- (0,3) --cycle;
			\filldraw[fill=YellowOrange, fill opacity=.6] (0,2) -- (2,2) -- (2,1) -- (0,1) --cycle;
			\filldraw[fill=ForestGreen, fill opacity=.6] (0,3) -- (1,3) -- (1,2) -- (0,2) --cycle; 
			\filldraw[fill=Red, fill opacity=.6] (0,1) -- (1,1) -- (1,0) -- (0,0) --cycle; 

			\draw (0,0) -- (0,4) -- (4,4) -- (4,3) -- (3,3) -- (3,2) -- (2,2) -- (2,1) -- (1,1) -- (1,0) --cycle; 
			\draw (1,4) -- (1,1) -- (0,1);
			\draw (2,4) -- (2,2) -- (0,2);
			\draw (3,4) -- (3,3) -- (0,3);      
		\end{scope}
		\begin{scope}[xshift= 30cm, yshift=-2.5cm]
			\filldraw[fill=NavyBlue, fill opacity=.6] (0,3) -- (3,3) -- (3,2) -- (0,2) --cycle;
			\filldraw[fill=YellowOrange, fill opacity=.6] (0,2) -- (2,2) -- (2,1) -- (0,1) --cycle;
			\filldraw[fill=ForestGreen, fill opacity=.6] (0,4) -- (1,4) -- (1,3) -- (0,3) --cycle; 
			\filldraw[fill=Red, fill opacity=.6] (0,1) -- (1,1) -- (1,0) -- (0,0) --cycle; 

			\draw (0,0) -- (0,4) -- (4,4) -- (4,3) -- (3,3) -- (3,2) -- (2,2) -- (2,1) -- (1,1) -- (1,0) --cycle; 
			\draw (1,4) -- (1,1) -- (0,1);
			\draw (2,4) -- (2,2) -- (0,2);
			\draw (3,4) -- (3,3) -- (0,3);      
		\end{scope}
		\begin{scope}[xshift= 35cm,yshift=2.5cm]
			\filldraw[fill=NavyBlue, fill opacity=.6] (0,4) -- (3,4) -- (3,3) -- (0,3) --cycle;
			\filldraw[fill=YellowOrange, fill opacity=.6] (0,2) -- (2,2) -- (2,1) -- (0,1) --cycle;
			\filldraw[fill=Red, fill opacity=.6] (0,3) -- (1,3) -- (1,2) -- (0,2) --cycle; 
			\filldraw[fill=ForestGreen, fill opacity=.6] (0,1) -- (1,1) -- (1,0) -- (0,0) --cycle; 

			\draw (0,0) -- (0,4) -- (4,4) -- (4,3) -- (3,3) -- (3,2) -- (2,2) -- (2,1) -- (1,1) -- (1,0) --cycle; 
			\draw (1,4) -- (1,1) -- (0,1);
			\draw (2,4) -- (2,2) -- (0,2);
			\draw (3,4) -- (3,3) -- (0,3);      
		\end{scope}
		\begin{scope}[xshift= 35cm, yshift=-2.5cm]
			\filldraw[fill=NavyBlue, fill opacity=.6] (0,3) -- (3,3) -- (3,2) -- (0,2) --cycle;
			\filldraw[fill=YellowOrange, fill opacity=.6] (0,2) -- (2,2) -- (2,1) -- (0,1) --cycle;
			\filldraw[fill=ForestGreen, fill opacity=.6] (0,1) -- (1,1) -- (1,0) -- (0,0) --cycle; 
			\filldraw[fill=Red, fill opacity=.6] (0,4) -- (1,4) -- (1,3) -- (0,3) --cycle; 

			\draw (0,0) -- (0,4) -- (4,4) -- (4,3) -- (3,3) -- (3,2) -- (2,2) -- (2,1) -- (1,1) -- (1,0) --cycle; 
			\draw (1,4) -- (1,1) -- (0,1);
			\draw (2,4) -- (2,2) -- (0,2);
			\draw (3,4) -- (3,3) -- (0,3);      
		\end{scope}
	\end{tikzpicture}
	\end{center}
	\Cref{eqn:M:n:mu:nu} is obtained by counting the number of embeddings of the parts of $\nu$ into $\mu$, starting from the parts with the highest value of $\nu$ and proceeding in decreasing order. This concludes the proof.
\end{proof}

\begin{remark}
	The above counting can be performed in a dual way, namely by counting the number of coverings of $\nu$ by parts of $\mu$, starting by the smallest parts of $\mu$ and proceeding in increasing order. From this point-of-view, we obtain
	\begin{equation}
		M_{n,\mu,\nu}
		=
		\frac{1}{z_{\nu}} \prod_{k=0}^{\nu_1}
			\Biggl(
				\sum_{i = 0}^{k} p_i(\nu)
				-
				\sum_{j=0}^{k-1} p_j(\mu)
			\Biggr)^{\!\! \underline{\tilde{p}_k(\mu)}}
			,
	\end{equation}
	where and $\tilde{p}_k(\mu) = p_k(\mu)$ for all $k$, with the only exception of $k = \nu_1$ for which one sets $\tilde{p}_{\nu_1}(\mu) = \sum_{m \geq \nu_1} {p}_m(\mu)$.
\end{remark}

\begin{example}
	From \cref{eqn:M:n:mu:nu}, one can show that
	\begin{equation}
		M_{n,\mu,(1^m,0^{n-m})} = \frac{(n - p_0(\mu))^{\underline{m}}}{m!} \,,
		\quad
		M_{n,\mu,(2,1^m,0^{n-m-1})} = \frac{(n - p_0(\mu) - 1)^{\underline{m}} (n - p_0(\mu) - p_1(\mu))}{m!} \,.
	\end{equation}
\end{example}

\subsection{Visualising the large genus asymptotics}
\label{app:visual}
In the main body of the paper, we have highlighted several features of the large genus asymptotics of intersection numbers. On the one hand, thanks to determinantal formulae, we were able to provide an algorithm for the computation of subleading corrections to the leading asymptotics of $\psi$-class intersection numbers as functions of the multiplicities $p_i$. On the other hand, we have shown how, in the presence of multiple Borel plane singularities, the resurgent large-order asymptotics organise themselves according to the distance from the origin of the various singularities. This results in the leading asymptotics being determined by the singularity that is the closest to the origin, while the remaining singularities yield exponentially subleading contributions. In the case of the $W_{g,n}(\bm{x})$ correlators, because of the competition between the Borel plane singularities $\pm A(x_i)$, the values of the $x_i$ determine which singularity yields the leading asymptotics. This kind of effect manifests itself also in the case of the $r$-spin intersection numbers with $r \geq 3$, where we observe exponentially suppressed corrections to the leading asymptotics associated to Borel singularities which are at higher distances from the origin than the leading ones.

Most of the above-mentioned effects are subleading in nature, and as such hidden by the leading asymptotics. However, as it will be made clear below, one can always construct sequences out of intersection numbers which converge to various quantities highlighting subleading corrections to the leading asymptotics. Plotting such sequences provides instructive visualisations of various effects.

\subsubsection{Subleading corrections to the asymptotics of $\psi$-class intersection numbers}
The large genus asymptotics for $\psi$-class intersection numbers is captured by the following sequence $G_{d,K}$, whose behaviour is dictated by \cref{eqn:intro:psi:large:g}:
\begin{equation}
	G_{d,K}
	\coloneqq
	\ms{S}_{\ms{A}}^{-1} \, \frac{4\pi}{2^n} \,
	\frac{\ms{A}^{2g-2+n} \, \prod_{i=1}^n (2d_i + 1)!!}{\Gamma(2g-2+n)} \,
	\braket{\tau_{d_1} \cdots \tau_{d_n}}
	-
	\sum_{k=1}^{K}\frac{\ms{A}^k}{(2g-3+n)^{\underline{k}}} \, \alpha_{k}
	=
	1 + \bigO\Bigl( \frac{1}{g^{K+1}} \Bigr) \,.
\end{equation}
Notice that the sum starts from $k = 1$, i.e.~it leaves out the leading term. By construction, the higher $K$ is (that is, the more subleading terms encoded by the coefficients $\alpha_k$ we include in the sequence), the faster the sequence converges to $1$. This is shown for $n = 1$ in \cref{fig:subleading}, which illustrates how the asymptotics of intersection numbers are substantially improved by including more and more subleading corrections.
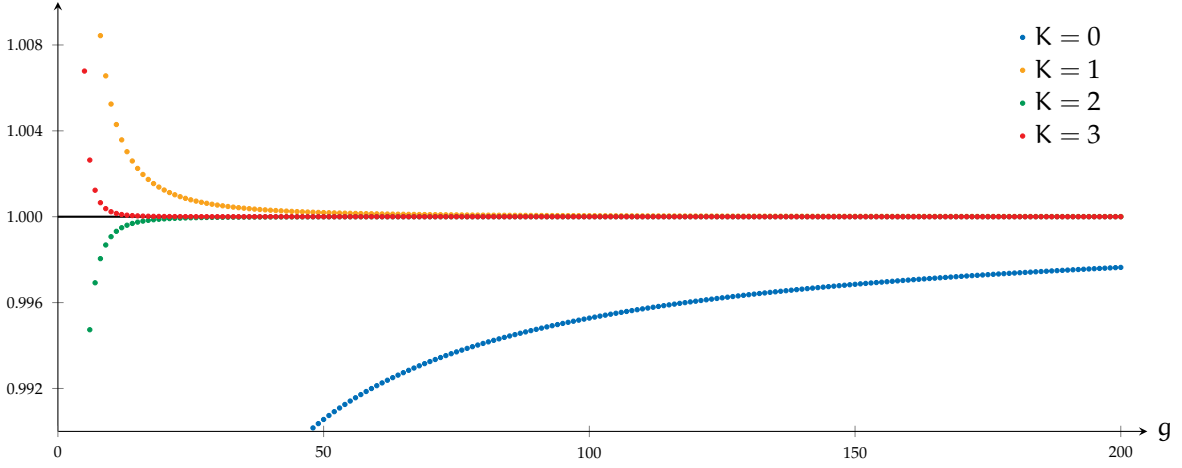
\begin{figure}
\centering
\begin{tikzpicture}
	\begin{axis}[
		x=.9*\textwidth/205,
		xmin=0, xmax=205,
		ymin=0.99, ymax=1.01,
		axis lines=center,
		legend style={draw=none}, legend pos=north east,
		xtick={0, 50, 100, 150, 200}, extra x ticks={0}, xlabel={$g$}, xlabel style={right},
		ytick={0.992,0.996,1,1.004,1.008}, yticklabels={$0.992$,$0.996$,$1.000$,$1.004$,$1.008$},
	]

	\addplot[thick, domain=0:200, forget plot] {1};
	\addplot[NavyBlue, mark = *, mark size=0.8, only marks] table {Images/subleading0.txt};
	\addlegendentry{$\, K=0$}; 
	\addplot[YellowOrange, mark = *, mark size=0.8, only marks] table {Images/subleading1.txt};
	\addlegendentry{$\, K=1$};
	\addplot[ForestGreen, mark = *, mark size=0.8, only marks] table {Images/subleading2.txt};
	\addlegendentry{$\, K=2$};
	\addplot[Red, mark = *, mark size=0.8, only marks] table {Images/subleading3.txt};
	\addlegendentry{$\, K=3$};

	\end{axis}
\end{tikzpicture}
\caption{
	The sequences $G_{3g-2,0}$ (blue), $G_{3g-2,1}$ (orange),$G_{3g-2,2}$ (green), and $G_{3g-2,3}$ (red). Notice how the convergence to 1 visibly improves as $K$ increases.
}
\label{fig:subleading}
\end{figure}

\subsubsection{Subleading corrections: dependence on the multiplicities}
\label{app:subleading:mult}
The dependence of the coefficients $\alpha_k$ on the multiplicities $p_m = \#\set{ d_i = m }$ can also be illustrated quite effectively. This is achieved by constructing the following sequences $H_{d,K}$, whose large genus asymptotics is again given by \cref{eqn:intro:psi:large:g}:
\begin{equation}
	H_{d,K}
	\coloneqq
	\frac{(2g-3+n)^{\underline{k}}}{\ms{A}^K} \, G_{d,K-1}
	=
	\alpha_K \bigl( n, p_0, \ldots, p_{\floor{\frac{3}{2}K} - 1} \bigr)
	+
	\bigO\bigl( g^{-1} \bigr) \,.
\end{equation}
The dependence on the multiplicities $p_m$ means that, according to the family of intersection numbers which are used to construct $H_{d,K}$, the latter converges to different values. This indeed can be observed in \cref{fig:multip:dep}. We consider $2$-point intersection numbers and construct the sequences $H_{d,K}$ for four distinct families of intersection numbers, corresponding to $d_1 = 0,1,2,3$, and for three different values of $K$. In \cref{fig:multip:dep:K0} ($K=0$) we see that they all converge to the same value, namely $\alpha_0 = 1$, which indeed does not depend on the the multiplicities. In \cref{fig:multip:dep:K1} ($K=1$), instead, we see that the sequence constructed from intersection numbers with $d_1 = 0$ converges to a different value compared to the other three sequences, due to the dependence on $p_0$ of $\alpha_1$. Finally, in \cref{fig:multip:dep:K2} ($K=2$) we see all four sequences converge to distinct values, as expected from the dependence on $p_0$, $p_1$ and $p_2$ of $\alpha_2$.
\begin{figure}[t]
\centering
	\begin{subfigure}[h]{1\textwidth}
		\centering
		\begin{tikzpicture}
			\begin{axis}[
				x=.9*\textwidth/75,
				yscale=.75,
				xmin=0, xmax=62.5,
				ymin=0.8, ymax=1.03,
				axis lines=center,
				legend style={draw=none, at={(1.2,1.7)},anchor=east},
				xtick={0, 10, 20, 30, 40, 50, 60}, extra x ticks={0}, xlabel={$g$}, xlabel style={right},
				ytick={0.85,0.90,0.95,1}, yticklabels={$0.85$,$0.90$,$0.95$,$1.00$},
			]

			\addplot[thick, NavyBlue, domain=0:60, forget plot] {1};

			\addplot[NavyBlue, mark = *, mark size=1, only marks, mark options={yscale=1.33}] table {Images/alpha00.txt};
			\addlegendentry{$\, d_1=0$}; 
			\addplot[YellowOrange, mark = *, mark size=1, only marks, mark options={yscale=1.33}] table {Images/alpha01.txt};
			\addlegendentry{$\, d_1=1$};
			\addplot[ForestGreen, mark = *, mark size=1, only marks, mark options={yscale=1.33}] table {Images/alpha02.txt};
			\addlegendentry{$\, d_1=2$};
			\addplot[Red, mark = *, mark size=1, only marks, mark options={yscale=1.33}] table {Images/alpha03.txt};
			\addlegendentry{$\, d_1=3$};

			\end{axis}
		\end{tikzpicture}
		\caption{
			The sequences $H_{(0,3g-1),0}$ (blue), $H_{(1,3g-2),0}$ (orange), $H_{(2,3g-3),0}$ (green), and $H_{(3,3g-4),0}$ (red): all four sequences converge to $\alpha_0=1$.
		}
		\label{fig:multip:dep:K0}
	\end{subfigure}
	\begin{subfigure}[h]{1\textwidth}
		\centering
		\begin{tikzpicture}
			\begin{axis}[
				x=.9*\textwidth/75,
				yscale=.75,
				xmin=0, xmax=62.5,
				ymin=-1.6, ymax=-0.05,
				axis lines=center, axis x line* = bottom,
				legend style={draw=none, at={(1.2,0.7)},anchor=east},
				xtick={0, 10, 20, 30, 40, 50, 60}, extra x ticks={0}, xlabel={$g$}, xlabel style={right},
				ytick={-0.2,-0.6,-1.0,-1.4}, yticklabels={$-0.2$,$-0.6$,$-1.0$,$-1.4$},
			]

			\addplot[thick, NavyBlue, domain=0:60, forget plot] {-5/12};
			\addplot[thick, YellowOrange, domain=0:60, forget plot] {-17/12};

			\addplot[NavyBlue, mark = *, mark size=1, only marks, mark options={yscale=1.25}] table {Images/alpha10.txt};
			\addlegendentry{$\, d_1=0$}; 
			\addplot[YellowOrange, mark = *, mark size=1, only marks, mark options={yscale=1.25}] table {Images/alpha11.txt};
			\addlegendentry{$\, d_1=1$};
			\addplot[ForestGreen, mark = *, mark size=1, only marks, mark options={yscale=1.25}] table {Images/alpha12.txt};
			\addlegendentry{$\, d_1=2$};
			\addplot[Red, mark = *, mark size=1, only marks, mark options={yscale=1.25}] table {Images/alpha13.txt};
			\addlegendentry{$\, d_1=3$};

			\end{axis}
		\end{tikzpicture}
		\caption{
			The sequences $H_{(0,3g-1),1}$ (blue), $H_{(1,3g-2),1}$ (orange), $H_{(2,3g-3),1}$ (green), and $H_{(3,3g-4),1}$ (red): the first one  converges to $\alpha_1( 2, 1) = -\tfrac{5}{12}$ (blue line), while the others converge to $\alpha_1( 2, 0) = -\tfrac{17}{12}$ (orange line).
		}
		\label{fig:multip:dep:K1}
	\end{subfigure}
	\begin{subfigure}[h]{1\textwidth}
		\centering
		\begin{tikzpicture}
			\begin{axis}[
				x=.9*\textwidth/75,
				yscale=.75,
				xmin=0, xmax=62.5,
				ymin=0, ymax=4.6,
				axis lines=center, axis x line* = bottom,
				legend style={draw=none, at={(1.2,0.75)},anchor=east},
				xtick={0, 10, 20, 30, 40, 50, 60}, extra x ticks={0}, xlabel={$g$}, xlabel style={right},
			]

			\addplot[thick, NavyBlue, domain=0:60, forget plot] {205/288};
			\addplot[thick, YellowOrange, domain=0:60, forget plot] {613/288};
			\addplot[thick, ForestGreen, domain=0:60, forget plot] {1045/288};
			\addplot[thick, Red, domain=0:60, forget plot] {1225/288};

			\addplot[NavyBlue, mark = *, mark size=1, only marks, mark options={yscale=1.25}] table {Images/alpha20.txt};
			\addlegendentry{$\, d_1=0$}; 
			\addplot[YellowOrange, mark = *, mark size=1, only marks, mark options={yscale=1.25}] table {Images/alpha21.txt};
			\addlegendentry{$\, d_1=1$};
			\addplot[ForestGreen, mark = *, mark size=1, only marks, mark options={yscale=1.25}] table {Images/alpha22.txt};
			\addlegendentry{$\, d_1=2$};
			\addplot[Red, mark = *, mark size=1, only marks, mark options={yscale=1.25}] table {Images/alpha23.txt};
			\addlegendentry{$\, d_1=3$};

			\end{axis}
		\end{tikzpicture}
		\caption{
			The sequences $H_{(0,3g-1),2}$ (blue), $H_{(1,3g-2),2}$ (orange), $H_{(2,3g-3),2}$ (green), and $H_{(3,3g-4),2}$ (red): they converge to $\alpha_2( 2, 1,0,0) = \tfrac{205}{288}$ (blue line), $\alpha_2( 2, 0,1,0) = \tfrac{613}{288}$ (orange line), $\alpha_2( 2, 0,0,1) = \tfrac{1045}{288}$ (green line), and $\alpha_2( 2, 0,0,0) = \tfrac{1225}{288}$ (red line) respectively.
		}
		\label{fig:multip:dep:K2}
	\end{subfigure}
\caption{Visualisation of the dependence of the subleading contribution on the multiplicities.}
\label{fig:multip:dep}
\end{figure}
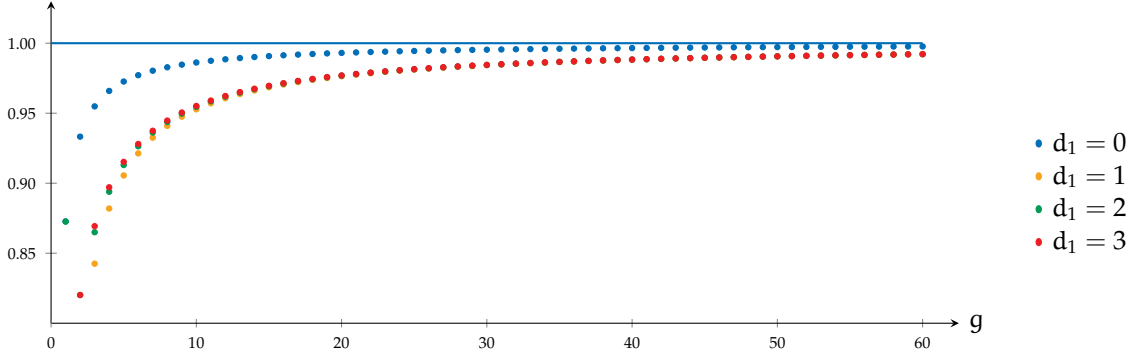
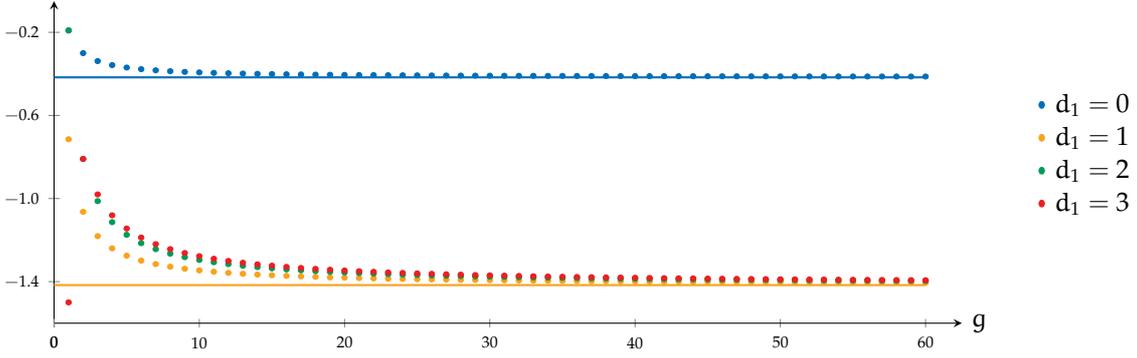
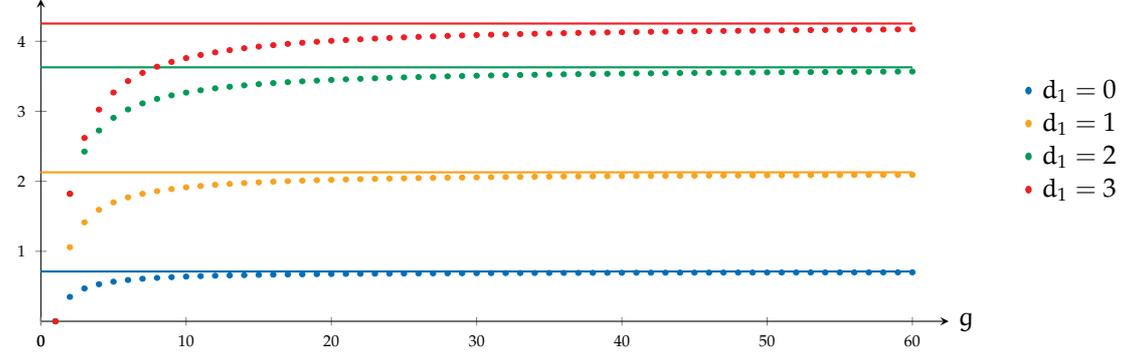

\subsubsection{Exponentially subleading corrections}
\label{app:exp}
The Borel plane singularity structure of $r$-spin intersection numbers with $r \geq 3$ leads to large genus asymptotics which sharply differ from those of $\psi$-class and $\Theta$-class intersection numbers. First of all, the leading contribution to the asymptotics is given by pairs of complex conjugate instanton actions, which leads to the oscillating asymptotics given in \eqref{eqn:intro:rspin:large:g}. Moreover, for $r \geq 4$ the asymptotics are characterised by the presence of exponentially subleading corrections, due to the remaining Borel plane singularities. The leading oscillating behaviour is particularly easy to visualise. One simply has to introduce the sequences $I^{r\textup{-spin}}_{d,a}$, whose asymptotics is given by \cref{eqn:intro:rspin:large:g}:
\begin{equation}
\begin{split}
	I^{r\textup{-spin}}_{d,a}
	& \coloneqq
	\ms{S}_{r,1}^{-1} \,
	\frac{2\pi}{2^n} \,
	\frac{(-r)^{g-1-|d|} |\ms{A}_{r,1}|^{2g-2+n} \prod_{i=1}^n (r d_i + a_i)!_{(r)}}{\Gamma(2g-2+n)} \,
	\braket{\tau_{d_1,a_1} \cdots \tau_{d_n,a_n}}^{r\textup{-spin}} \\
	& =
	(-1)^{|d|-g+1} \, \gamma^{(r,1)}_0 + \bigO\bigl(g^{-1}\bigr) \,.
\end{split}
\end{equation} 
This is shown in \cref{fig:spin4leading} for $r=4$, $n=1$. 
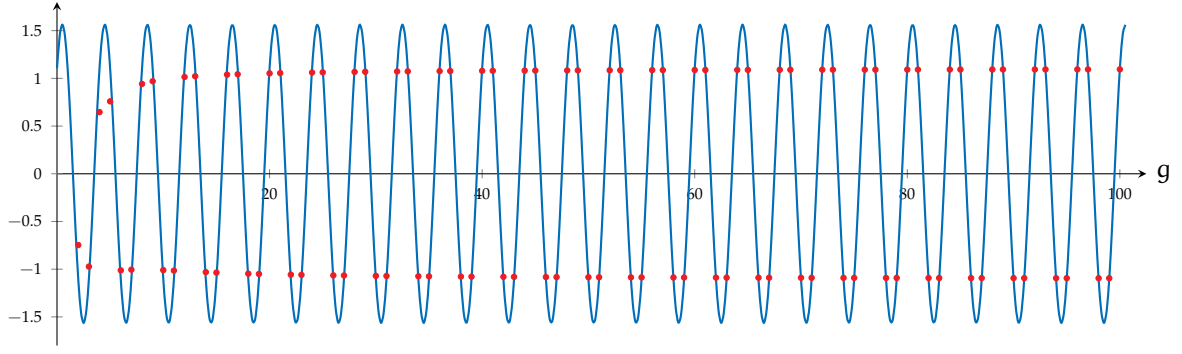
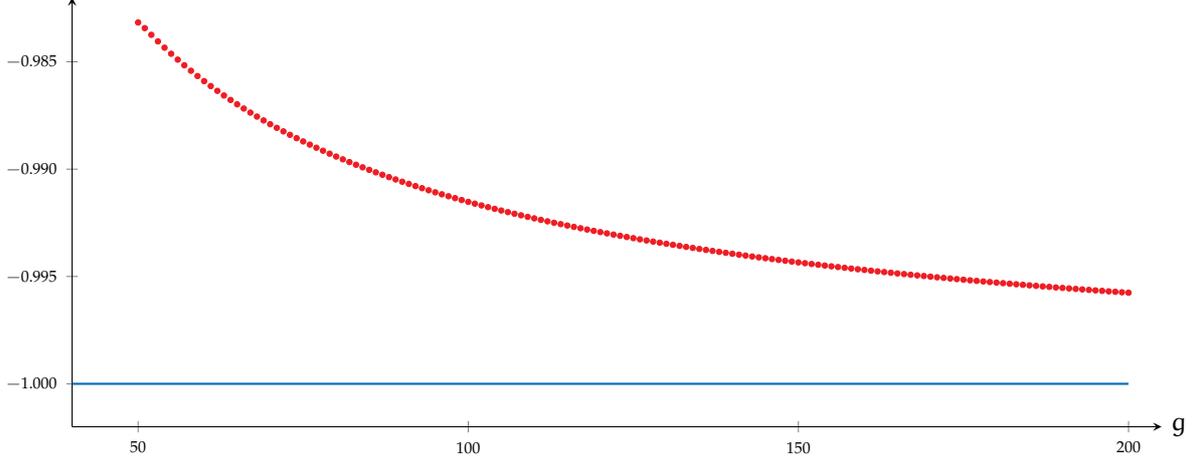
\begin{figure}[t]
\centering
	\begin{subfigure}[h]{1\textwidth}
		\centering
		\centering
		\begin{tikzpicture}
			\begin{axis}[
				x=.9*\textwidth/102.5,
				yscale=.8,
				xmin=0, xmax=102.5,
				ymin=-1.8, ymax=1.8,
				axis lines=center,
				xtick={0, 20, 40, 60, 80, 100}, xlabel={$g$}, xlabel style={right},
				ytick={-1.5,-1,-0.5,0,0.5,1,1.5}, extra y ticks={0},
			]
			
			\addplot[thick, NavyBlue, samples=400, smooth, domain=0:100.53] {25/(16*sqrt(2))*(cos(deg(x*pi/2))+sin(deg(x*pi/2)))};
			\addplot[Red, mark = *, mark size=1, only marks, mark options={yscale=1.25}] table {Images/spin4leading.txt};

			\end{axis}
		\end{tikzpicture}
		\caption{
			The sequence $I^{4\textup{-spin}}_{d,a}$ (red) and $(-1)^{d-g+1}\,\gamma^{(4,1)}_0$ (blue), plotted as a continuous function of $g$.}
		\label{fig:spin4leading}
	\end{subfigure}
	\begin{subfigure}[h]{1\textwidth}
		\centering
		\centering
		\begin{tikzpicture}
			\begin{axis}[
				x=.9*\textwidth/165,
				xmin=40, xmax=205,
				ymin=-1.002, ymax=-0.982,
				axis lines=center, axis x line* = bottom,
				xtick={50, 100, 150, 200}, xlabel={$g$}, xlabel style={right},
				ytick={-0.985,-0.990,-0.995,-1.000}, yticklabels={$-0.985$,$-0.990$,$-0.995$,$-1.000$},
			]
			
			\addplot[thick, NavyBlue, domain=40:200] {-1};
			\addplot[Red, mark = *, mark size=1, only marks] table {Images/spin4subleading.txt};

			\end{axis}
		\end{tikzpicture}
		\caption{
			The sequence $J^{4\textup{-spin}}_{d,a,80}$ (red) and $(-1)^{d-g+1}\,\gamma^{(4,2)}_0$ (blue), plotted as a continuous function of $g$.
		}
		\label{fig:spin4subleading}
	\end{subfigure}
\caption{
	Visualisation of the exponentially leading and subleading contributions in the $1$-point $4$-spin intersection numbers. Here $d$ and $a$ are functions of $g$, uniquely determined by $4d+a=5(2g-1)$ with $a \in \set{1,2,3}$.
}
\label{fig:spin4}
\end{figure}
The contributions to the asymptotics due to the remaining Borel plane singularities are exponentially suppressed, and therefore a bit trickier to visualise. Nevertheless, this can be achieved by constructing the following sequences:
\begin{equation}
	J^{r\textup{-spin}}_{d,a,K}
	\coloneqq
	\frac{|\ms{A}_{r,2}|^{2g-2+n}}{|\ms{A}_{r,1}|^{2g-2+n}}
	\left(
		I^{r\textup{-spin}}_{d,a}
		-
		\sum_{k=0}^{K} \frac{|\ms{A}_{r,1}|^k}{(2g-3+n)^{\underline{k}}} \,(-1)^{|d|-g+1}\,
		\gamma^{(r,1)}_{k}
	\right) \,.
\end{equation}
For $K \gg 0$, we are removing a large number of subleading corrections from the first line of \cref{eqn:intro:rspin:large:g}, allowing for the second line to show up. This is illustrated , again for $r=4$, $n=1$, in \cref{fig:spin4subleading}. It is important to note that, for $r=4$, there is only one exponentially subleading correction to the asymptotics, corresponding to the last line of \cref{eqn:intro:rspin:large:g}. Such a term, which is specific to the even $r$ case, does not come from a pair of complex-conjugate instanton actions, but rather from a single one. As such, it does not lead to oscillating asymptotics, much like in the $r=2$ case.

\subsubsection{Competition between singularities in the Borel plane}
\label{app:sings}
The competition between singularities in the Borel plane is particularly well illustrated by the large genus asymptotics of the $n$-point correlators $W_{g,n}(\bm{x})$. In order to show this effect, let us consider the case of Airy correlators, whose large genus asymptotics were given in \cref{eqn:large:g:corr:WK}. One immediately realises that the large genus asymptotics are dominated by the instanton action $A(x_i)$ with the smallest absolute value. Therefore, as the values of the $x_i$ change, so does the dominance in the large genus asymptotics. This can be seen quite distinctly for two-point correlators by introducing the sequence $L_g$, whose asymptotic is given by \cref{eqn:large:g:corr:WK}:
\begin{equation}
	L_g(x_1,x_2)
	\coloneqq
	2g\sqrt{\frac{W_{g,2}(x_1,x_2)}{W_{g+1,2}(x_1,x_2)}}
	=
	\min\set{ \left|A(x_1)\right|, \, \left|A(x_2)\right| } + \bigO\bigl( g^{-1} \bigr) \,,
\end{equation}
with $A(x) = \frac{4}{3}x^{3/2}$. Hence, if for example $x_1$ and $x_2$ are allowed to vary on the positive real axis, one expects $L_g(x_1,x_2)$ to converge to $A(x_1)$ as long as $x_1 \leq x_2$, and to $A(x_2)$ as long as $x_2 \leq x_1$. This is indeed what happens, as it can be seen in \cref{fig:W2}.

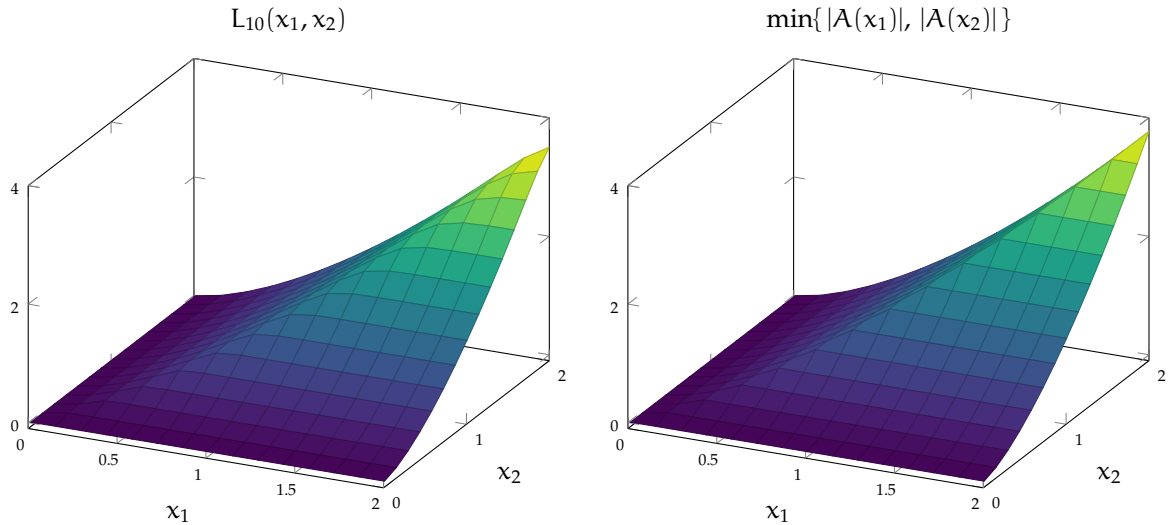
\begin{figure}[t]
\centering
\begin{subfigure}[h]{0.49\textwidth}
\centering
\begin{tikzpicture}[
	declare function={
		f(\x,\y)=sqrt(((1804857108504066435*x^(32)+1749323043627018237*x^(31)*y+1751967522906877675*x^(30)*y^2+1752942945592071730*x^(29)*y^3+1752685772002755030*x^(28)*y^4+1752726378358962930*x^(27)*y^5+1752740590583635695*x^(26)*y^6+1752731626538963385*x^(25)*y^7+1752734263022690535*x^(24)*y^8+1752735025271387160*x^(23)*y^9+1752734039213498760*x^(22)*y^(10)+1752734499373846680*x^(21)*y^(11)+1752734591673451350*x^(20)*y^(12)+1752734329032299850*x^(19)*y^(13)+1752734510860789350*x^(18)*y^(14)+1752734525112211500*x^(17)*y^(15)+1752734373639953220*x^(16)*y^(16)+1752734525112211500*x^(15)*y^(17)+1752734510860789350*x^(14)*y^(18)+1752734329032299850*x^(13)*y^(19)+1752734591673451350*x^(12)*y^(20)+1752734499373846680*x^(11)*y^(21)+1752734039213498760*x^(10)*y^(22)+1752735025271387160*x^9*y^(23)+1752734263022690535*x^8*y^(24)+1752731626538963385*x^7*y^(25)+1752740590583635695*x^6*y^(26)+1752726378358962930*x^5*y^(27)+1752685772002755030*x^4*y^(28)+1752942945592071730*x^3*y^(29)+1751967522906877675*x^2*y^(30)+1749323043627018237*x*y^(31)+1804857108504066435*y^(32))/(12800*x^3*y^3*(x+y)*(238436656380769*x^(28)-8082598521382*x^(27)*y+238862056302947*x^(26)*y^2-7927907640590*x^(25)*y^3+238816086843089*x^(24)*y^4-7919795382968*x^(23)*y^5+238818818521676*x^(22)*y^6-7921810238096*x^(21)*y^7+238819490140052*x^(20)*y^8-7921633222400*x^(19)*y^9+238819203029228*x^(18)*y^(10)-7921478624264*x^(17)*y^(11)+238819227054614*x^(16)*y^(12)-7921581590204*x^(15)*y^(13)+238819311299474*x^(14)*y^(14)-7921581590204*x^(13)*y^(15)+238819227054614*x^(12)*y^(16)-7921478624264*x^(11)*y^(17)+238819203029228*x^(10)*y^(18)-7921633222400*x^9*y^(19)+238819490140052*x^8*y^(20)-7921810238096*x^7*y^(21)+238818818521676*x^6*y^(22)-7919795382968*x^5*y^(23)+238816086843089*x^4*y^(24)-7927907640590*x^3*y^(25)+238862056302947*x^2*y^(26)-8082598521382*x*y^(27)+238436656380769*y^(28))))^(-1));
	}]
	\begin{axis}[
		colormap/viridis,
		xmin=0, xmax=2,
		ymin=0, ymax=2,
		zmin=-0.1, zmax=4,
		xlabel={$x_1$}, ylabel={$x_2$}, title={$L_{10}(x_1,x_2)$},
	]
		\addplot3[surf, samples=15, domain=0.01:2] {f(\x,\y)};
	\end{axis}
\end{tikzpicture}
\end{subfigure}
\begin{subfigure}[h]{0.49\textwidth}
\centering
\begin{tikzpicture}[
	declare function={
		f(\x,\y)=(x<=y)*4/3*x^(3/2) + (y<x)*4/3*y^(3/2);
	}]
	\begin{axis}[
		colormap/viridis,
		xmin=0, xmax=2,
		ymin=0, ymax=2,
		zmin=-0.1, zmax=4,
		xlabel={$x_1$}, ylabel={$x_2$}, title={$\min\{\, |A(x_1)|,\,|A(x_2)| \,\}$},
	]
		\addplot3[surf, samples=15, domain=0.01:2] {f(\x,\y)};
	\end{axis}
\end{tikzpicture}
\end{subfigure}
\caption{
	The function $L_{10}(x_1,x_2)$ (left), and the function  $\min\{\, |A(x_1)|, \, |A(x_2)| \,\}$ (right).
}
\label{fig:W2}
\end{figure}

\printbibliography

\end{document}